\documentclass[a4paper,leqno]{amsart}

\usepackage{amsmath,amsfonts,amsthm,amssymb,amscd}

\usepackage[
    a4paper,
    left=3cm,
    right=3cm,
    top=3cm,
    bottom=3cm,
    twoside=false
]{geometry}

\usepackage[T1]{fontenc}
\usepackage[english]{babel}
\usepackage{graphicx,xcolor}
\usepackage{cite}
\usepackage[official]{eurosym}

\usepackage{hyperref}
\hypersetup{
    colorlinks=true,
    filecolor=black
}

\usepackage{hyperref}

\newtheorem{theorem}{Theorem}[section]

\newtheorem{lemma}[theorem]{Lemma}
\newtheorem{proposition}[theorem]{Proposition}

\theoremstyle{definition}

\newtheorem{remark}[theorem]{Remark}

\numberwithin{equation}{section}

\newcommand\ud{\, \mathrm{d}}

\newcommand{\RomanUpper}[1]{\expandafter\uppercase\expandafter{\romannumeral #1}}

\renewcommand{\le}{\leqslant}
\renewcommand{\ge}{\geqslant}

\usepackage{mathtools}

\renewcommand{\Re}{\mathop{\rm Re}\nolimits}
\renewcommand{\Im}{\mathop{\rm Im}\nolimits}

\newcommand{\R}{{\mathbb R}}

\theoremstyle{plain}

\def\s{\sigma}

\numberwithin{equation}{section}

\def\det{\operatorname{det}}

\def\s{\sigma}
\def\l{\lambda}

\begin{document}

   \parindent=8pt

    \title[Blowup for the energy supercritical NLS]
    {Slow blowup solutions for the focusing energy supercritical nonlinear Schr\"odinger equation}%\lz{suggested title: Slow type II blow-up solutions for the focusing energy-supercritical nonlinear Schr\"odinger equation}

    \author[X. ~Li]{Xuanying Li}
    \address{The Institute of Geometry and Physics, University of Science and Technology of China, Hefei 230026, Anhui, China}
    \email{lxy0418@ustc.edu.cn}

 \begin{abstract}
We consider the focusing energy supercritical nonlinear Schr\"odinger equation
\begin{equation*}
i\partial_t\psi+\Delta\psi+|\psi|^{2p}\psi=0
\end{equation*}
in dimensions $d\geq11$, with $p\in\mathbb N^\ast$, under the Joseph--Lundgren condition $2p+1>p_{JL}(d)$. For every $\nu>\nu_0(d,p)$,
there exists a radial finite-time type II blow-up solution concentrating the ground state $Q$ in the form
\begin{equation*}
\psi(t,x)=\lambda^{\frac{1}{p}}(t)Q(\lambda(t)x)+\zeta(t,x),\quad  \lambda(t)=t^{-\frac12-\nu}.
\end{equation*}
All homogeneous Sobolev norms $\|\psi(t)\|_{\dot H^s}$, $1\leq s<s_c$, remain bounded.
The radiation converges to a limiting profile
\begin{equation*}
    \zeta^*\in \dot H^s(\mathbb R^d),  \qquad s_c<s<s_c+2\nu|\lambda_+|,
\end{equation*}
where $\lambda_+$ is determined by the leading correction in the asymptotic expansion of $Q$ at infinity. Our proof combines matched asymptotic expansions across three spatial scales with modified energy estimates. The key ingredient is the systematic treatment of the singular nonlinear background and coupled oscillations generated by the slow algebraic tail of $Q$.
\end{abstract}

  \maketitle

\section{Introduction}

\subsection{Setting of the problem}

We consider the focusing nonlinear Schr\"{o}dinger equation
\begin{equation}\label{NLS}
\left\{
\begin{aligned}
&i\partial_t\psi=-\Delta\psi-|\psi|^{2p}\psi,\qquad (t,x)\in I\times\mathbb R^d,\\
&\psi|_{t=0}=\psi_0,
\end{aligned}
\right.
\end{equation}
where $I\subset\mathbb R$ is an interval. For sufficiently regular
solutions, the
energy and mass are conserved:
\begin{equation*}%\label{def energy mass}
\begin{gathered}
E(\psi(t))=\frac12\int_{\mathbb R^d}|\nabla\psi(t,x)|^2\,\mathrm dx-\frac{1}{2(p+1)}\int_{\mathbb R^d}|\psi(t,x)|^{2p+2}\,\mathrm dx=E(\psi_0),\\[0.5em]
M(\psi(t))=\int_{\mathbb R^d}|\psi(t,x)|^2\,\mathrm dx=M(\psi_0).
\end{gathered}
\end{equation*}

Equation \eqref{NLS} is invariant under the scaling
\begin{equation*}
\psi(t,x)\longmapsto\psi_\lambda(t,x):=\lambda^{\frac1p}\psi(\lambda^2t,\lambda x),\qquad \lambda>0,
\end{equation*}
and the corresponding scaling critical Sobolev exponent is
\begin{equation*}
s_c=\frac {d}{2}-\frac{1}{p}.
\end{equation*}
Throughout this paper, we work in the energy supercritical regime
\begin{equation*}
s_c>1,\qquad\text{equivalently}\qquad p>\frac{2}{d-2}, \qquad d\ge3.
\end{equation*}

\medskip

The blowup solution is constructed around the positive radial stationary profile $Q$, normalized by
\begin{equation*}%\label{ground state equation introduction}
\Delta Q+Q^{2p+1}=0, \qquad Q(0)=1, \qquad Q'(0)=0.
\end{equation*}
The asymptotic behavior of $Q$ at infinity is explicitly given in terms of the Joseph--Lundgren exponent
\begin{equation*}%\label{def p JL}
p_{JL}(d)=
\begin{cases}
+\infty, & 3\le d\le10,\\[0.4em]\displaystyle1+\frac{4}{d-4-2\sqrt{d-1}},
& d\ge11.
\end{cases}
\end{equation*}
This exponent was introduced by Joseph and Lundgren in their study of positive radial solutions to Lane--Emden-type equations \cite{JosephLundgren}.

\medskip

We shall assume throughout that $p\in\mathbb N^\ast$ and
\begin{equation}\label{JL condition}
2p+1>p_{JL}(d).
\end{equation}
Linearizing the stationary equation around $Q$ and seeking perturbations with asymptotic behavior $r^{-\frac1p+\lambda_{+}}$
as $r\to\infty$ yields the characteristic equation
\begin{equation*}
\lambda^2+\left(d-2-\frac{2}{p}\right)\lambda+2\left(d-2-\frac1p\right)=0.
\end{equation*}
Its discriminant is
\begin{equation}\label{def Discr}
\operatorname{Discr}:=(d-2)^2-4(2p+1)c_\infty^{2p}.
\end{equation}
Condition \eqref{JL condition} is equivalent to
\begin{equation*}
\operatorname{Discr}>0,
\end{equation*}
and hence the two characteristic exponents are real. We write them as
\begin{equation}\label{def lambda pm}
\lambda_\pm=\frac1p-\frac{d-2}{2}\pm\frac{\sqrt{\operatorname{Discr}}}{2},
\end{equation}
so that
\begin{equation*}
\lambda_-<\lambda_+<0.
\end{equation*}
 The far field asymptotics of $Q$ are generated by the characteristic exponents $\lambda_+$ and $\lambda_-$, see \cite{MerleRaphaelRodnianski15}. In particular, one has
\begin{equation}\label{expansion Q}
Q(r)=r^{-\frac1p}\Big(c_\infty+\sum\limits_{m+n\geq 1}c_{m,n}r^{\lambda_+ m+\lambda_{-}n}\Big),\qquad r\to\infty.
\end{equation}
The higher order asymptotic expansion of $Q$ will be developed below and will play an essential role in the matching construction.

\medskip
\subsection{Previous results}

A central mechanism for singularity formation in focusing nonlinear
Schr\"odinger equations is the concentration of a rescaled ground state.
Schematically, a solution with a single concentrating core has the form
\begin{equation*}
u(t,x)= e^{i\gamma(t)}\lambda(t)^{\frac{1}{p}} Q(\lambda(t)x)+\text{radiation},
\end{equation*}
where $Q$ denotes the ground state of the equation under consideration
and $\lambda(t)\to+\infty$ is the inverse concentration scale.
Identifying the concentrating profile is only part of the problem:
one must also determine the possible concentration rates and understand
how the concentrating core interacts with the radiation.

For the mass critical nonlinear Schr\"{o}dinger equation, ground state
concentration gives rise to distinct blowup regimes.
The pseudoconformal transformation applied to the ground state
standing wave yields an explicit minimal mass blowup solution with
concentration rate
\begin{equation*}
\lambda(t)\approx (T-t)^{-1}.
\end{equation*}
In dimensions one and two, Bourgain and Wang constructed solutions with the same pseudoconformal rate and a prescribed smooth regular component satisfying suitable vanishing conditions at the blowup point \cite{BourgainWang97}. Another regime is characterized by the log-log law
\begin{equation*}
\lambda(t)\approx\left(\frac{\log|\log(T-t)|}{T-t}\right)^{\frac12}.
\end{equation*}
Blowup solutions with this rate were constructed by Perelman in one dimension for a non-generic class of initial data \cite{Perelman01}. Merle and Rapha\"{e}l subsequently established log-log blowup for
open sets of $H^1$ initial data and developed a detailed analysis of the near ground state dynamics \cite{MerleRaphael05a,MerleRaphael06}. The stability of the log-log rate was proved by Rapha\"{e}l \cite{Raphael05}. Merle and Rapha\"{e}l also proved that, after subtracting the concentrating ground state, the radiation converges strongly in $L^2$ to a limiting profile \cite{MerleRaphael05b}. 

In the mass supercritical, energy subcritical range $0<s_c<1$, self-similar collapse provides another mechanism for finite time singularity formation, with inverse concentration scale
\begin{equation*}
\lambda(t)\approx (T-t)^{-\frac{1}{2}}.
\end{equation*}
The formal and numerical background is discussed in \cite{SulemSulem99}. For $1\leq d\leq5$ and slightly mass-supercritical nonlinearities, Merle, Rapha\"{e}l and Szeftel proved the existence and stability in $H^1$ of finite time blowup dynamics with the self-similar rate \cite{MerleRaphaelSzeftel10}. Their argument perturbs the analysis of the mass critical log-log regime and controls the concentration rate without identifying an exact limiting self-similar profile. Bahri, Martel and Rapha\"{e}l subsequently constructed radial finite energy self-similar profiles in every dimension $d\geq1$ for $0<s_c\ll1$, using matched asymptotic expansions to obtain a branch bifurcating from the ground state \cite{BahriMartelRaphael21}. Li \cite{Li23} established finite codimensional asymptotic stability of self-similar blow-up profiles, allowing nonradial perturbations, through Strichartz estimates obtained by adapting Beceanu's operator-valued Wiener theorem \cite{Beceanu11}. His subsequent spectral analyses yield asymptotic stability modulo symmetries for slightly mass-supercritical profiles in dimensions $1\leq d\leq10$, using the corresponding spectral coercivity input \cite{Li25a,Li25b}. Beyond the perturbative regime $0<s_c\ll1$, Donninger and Sch\"{o}rkhuber \cite{DonningerSchorkhuber26} and Dahne and Figueras \cite{DahneFigueras24} obtained computer-assisted constructions of finite energy radial self-similar blow-up solutions for the three-dimensional cubic NLS. Donninger and Lichtnecker subsequently extended the cubic construction to an open interval of nonlinear powers containing the cubic exponent  \cite{DonningerLichtnecker25}.

For the three dimensional energy critical nonlinear Schr\"{o}dinger equation, Ortoleva and Perelman constructed solutions concentrating
the ground state at infinite time with power-law concentration rates \cite{OrtolevaPerelman14}. More recently, Schmid constructed
finite time type II blowup solutions with a continuously
parameterized family of rates
\begin{equation*}
\lambda(t)=(T-t)^{-\frac12-\nu},
\qquad \nu>1,
\end{equation*}
together with convergence of the radiation in $\dot H^1\cap\dot H^2$ \cite{Schmid26}. Very recently, Kim \cite{Kim26} constructed smooth radial finite time type II blowup solutions concentrating a single ground state with quantized concentration rates
\begin{equation*}
\lambda(t)= (T-t)^{-2L}, \qquad L\in\mathbb N^\ast.
\end{equation*}

In the energy supercritical regime, Merle, Rapha\"{e}l and Rodnianski constructed radial finite time type II blowup solutions in dimensions
$d\geq11$ for a class of polynomial nonlinearities above the Joseph--Lundgren threshold \cite{MerleRaphaelRodnianski15}.
Their construction combines a finite-dimensional modulation system with energy estimates and produces smooth blowup solutions with an
integer indexed sequence of power law concentration rates. The algebraic tail of the ground state plays an essential role in their
analysis, and all Sobolev norms below the scaling regularity remain bounded up to the blowup time. Thus, the type II regimes constructed
in \cite{MerleRaphaelRodnianski15} exhibit a quantization of the admissible concentration rates.

A different line of work shows that the concentration rate can also be prescribed continuously. Krieger, Schlag and Tataru developed such a prescribed rate construction for $1$-corotational wave maps from $\mathbb R^{2+1}$ to $\mathbb S^2$ \cite{KriegerSchlagTataru08} and for the three dimensional energy critical focusing nonlinear wave equation \cite{KriegerSchlagTataru09}. In both settings, finite time blowup
occurs through concentration of a rescaled harmonic map or stationary
ground state, with
\begin{equation*}
\lambda(t)=(T-t)^{-1-\nu}.
\end{equation*}
Their original constructions covered the range $\nu>\frac{1}{2}$. The full polynomial range $\nu>0$ was subsequently obtained by Gao and
Krieger for critical wave maps \cite{GaoKrieger15} and by Krieger and Schlag for the energy critical wave equation \cite{KriegerSchlag14}.
The same general prescribed rate philosophy was later implemented by Perelman for the $1$-equivariant Schr\"{o}dinger map equation from
$\mathbb R^2$ to $\mathbb S^2$, where she constructed finite time blowup solutions with
\begin{equation*}
\lambda(t)=(T-t)^{-\frac12-\nu}, \qquad \nu>1,
\end{equation*}
\cite{Perelman14}.
In these constructions, one starts from a prescribed concentration law, constructs sufficiently accurate corrections to the rescaled
stationary profile, and then controls the remaining error perturbatively. In particular, the regularity of the resulting solution and of its limiting radiation is closely related to the prescribed blowup exponent.

A complementary mechanism is to prescribe the asymptotic radiation rather than the concentration rate itself. For the energy critical one-equivariant wave map equation, Jendrej, Lawrie and Rodriguez constructed bubbling solutions with prescribed radiation, for which the blow-up rate is determined by the interaction between the radiation and the concentrating harmonic map
\cite{JendrejLawrieRodriguez22}. More recently, Kim, Kwon and Oh adapted this mechanism to the radial self-dual Chern--Simons--Schr\"odinger equation and constructed finite energy blowup solutions with a continuum of rates generated by prescribed asymptotic profiles \cite{KimSoonsikOh26}.

A natural question is: \emph{for the focusing energy-supercritical nonlinear Schr\"{o}dinger equation, does there exist finite time blowup solutions with prescribed continuous blowup rates?}  This question is suggested by Perelman in her ICM survey \cite[Section~5.4]{PerelmanICM2022}. The present paper carries out such a construction under the assumption $(\mathrm{H1})$. For every $\nu>\nu_0(d,p)$, we construct a radial finite time blowup solution with $$\lambda(t)=t^{-\frac12-\nu}.$$ The algebraic tail of $Q$ links the inner, self-similar, and remote regions. Matching across these scales yields an approximate solution with a controlled error, which is then lifted to an exact solution by linear energy estimates and a fixed-point argument. We also determine the Sobolev asymptotics on both sides of the scaling regularity and identify the limiting radiation profile.

\medskip

\subsection{Main result}
We impose the following nonresonance assumption on $d$ and $p$.
\medskip

 {\bf{(H1)}.}\label{H 1}
 Let $d\geq11$ and $p\in\mathbb{N}^{\ast}$ satisfy $2p+1>p_{JL}$. Let $\lambda_{\pm}$ be defined by \eqref{def lambda pm} and set
 \begin{equation}\label{def K}
 K=\frac{(d-2)p-2}{\gcd(2p,(d-2)p-2)}-1.
 \end{equation}
 Equivalently, the integer $K$ is given by 
\begin{equation*}%\label{def K 1}
K=\begin{cases}
(d-2)p-3,&\text{$d$ and $p$ are both odd},\\[3pt]
\dfrac{(d-2)p-2}{4}-1,&\text{$d$ is odd and $p\equiv 2\,(\text{mod 4})$},\\[5pt]
\dfrac{(d-2)p-2}{2}-1,&\text{otherwise}.
\end{cases}
\end{equation*}
 Assume that the following nonresonance conditions on $p, d$ hold for every integer $1\leq k\leq K$ and every $m,n\in\mathbb N_0$ with
$m+n\geq1$:
  \begin{equation}\label{ass H1}
  \begin{aligned}
  2j+\lambda_{+}m+\lambda_{-}n\neq& \lambda_{+}+\lambda_{-},&&j=1,2,\cdots, k, \quad k\in 2\mathbb{N}^{\ast}-1,\\
  2j+\lambda_{+}m+\lambda_{-}n\neq& \lambda_{\pm},&&j=2,3,\cdots, k, \quad k\in2\mathbb{N}^{\ast},\\
  2j+\lambda_{+}m+\lambda_{-}n\neq&0, &&j=3,4,\cdots, k, \quad k\in 2\mathbb{N}^{\ast}+1.
  \end{aligned}
  \end{equation}  
The set of pairs $(d,p)$ satisfying \eqref{ass H1} is nonempty. For example, $(d,p)=(11,5)$ satisfies $d\geq 11$, $2p+1>p_{JL}(d)$ and gives $\lambda_\pm=\frac{-43\pm\sqrt{89}}{10}$, $K=42$.

  \medskip

 %For even $k$, 
  %\begin{equation}\label{ass even}
 % \begin{aligned}
 % 2k+\lambda_{+}m+\lambda_{-}n\notin& \lambda_{\pm},&& m+n\geq 1\\
  %2(k-1)+\lambda_{+}m+\lambda_{-}n\notin& \lambda_{\pm}+4l,&&m,n\geq 1,\quad 0\leq l\leq [\frac{k-2}{2}]-1,\\
  %2(k-2)+\lambda_{+}m+\lambda_{-}n\notin& \lambda_{\pm}+4l,&& m\geq 0,\,n\geq 1,\quad 0\leq l\leq [\frac{k-3}{2}].
  %\end{aligned}
 % \end{equation} 
%For odd $k$,
  % \begin{equation}\label{ass odd}
  %\begin{aligned}
  %2k+\lambda_{+}m+\lambda_{-}n\notin& \{0,\,\lambda_{+}+\lambda_{-}\},&& m+n\geq 1\\
  %2(k-1)+\lambda_{+}m+\lambda_{-}n\neq& 4l_1,&&m,n\geq 1,\quad 0\leq l_1\leq [\frac{k-2}{2}]-1,\\
 % 2(k-1)+\lambda_{+}m+\lambda_{-}n\neq& \lambda_{+}+\lambda_{-}+4l_2,&& m,n\geq 1,\quad 0\leq l_2\leq [\frac{k-2}{2}],\\
  %2(k-2)+\lambda_{+}m+\lambda_{-}n\neq& 4l_1,&& m\geq 0,\,n\geq 1,\quad 0\leq l_1\leq [\frac{k-3}{2}]-1,\\
  %2(k-2)+\lambda_{+}m+\lambda_{-}n\neq& \lambda_{+}+\lambda_{-}+4l_2,&& m\geq 0,\,n\geq 1,\quad 0\leq l_2\leq [\frac{k-3}{2}].
   % \end{aligned}
 % \end{equation} 

We are now in a position to state the main result.
\begin{theorem}[Finite time type II blowup]\label{main thm}
{\sl Let $d\geq 11$ and $p\in\mathbb N^{\ast}$ satisfy
\begin{equation*}
2p+1>p_{JL}(d).
\end{equation*}
Let $\lambda_{\pm}$ be given by \eqref{def lambda pm}, and suppose that $(\mathrm{H1})$ holds. 
Set
\begin{equation*}
\nu_0(d,p)=\frac{\lfloor d/2\rfloor+1-s_c}{2|\lambda_+|}.
\end{equation*}
For every $\nu>\nu_0(d,p)$, fix an integer $\sigma$ satisfying
\begin{equation*}
\frac{d}{2}<\sigma<s_c+\min\Big\{2\nu|\lambda_+|,\frac{4\nu K}{1+2\nu}\Big\}.
\end{equation*}
Then there exist $T>0$ and a radial solution
\begin{equation*}
\psi\in C\big((0,T], \dot H^1(\mathbb R^d)\cap\dot H^\sigma(\mathbb R^d)\big)
\end{equation*}
to \eqref{NLS} of the form
\begin{equation}\label{form psi}
\psi(t,x)
=\lambda^{\frac1p}(t)Q(\lambda(t)x)+\zeta(t,x),
\end{equation}
where $\lambda(t)=t^{-\frac12-\nu}$, with the following properties:

\medskip
\noindent
{\rm (i) Asymptotic behavior above the scaling regularity.}
For every $s_c<s\leq\sigma$,
\begin{equation}\label{high singularity}
\lim\limits_{t\rightarrow 0}\lambda^{s_c-s}(t)\big\|\nabla^{s}\zeta(t,\cdot)\big\|_{L^2}=0.
\end{equation}
Consequently,
\begin{equation*}
\|\psi\|_{\dot{H}^s}=\big(\|Q\|_{\dot{H}^s}+o_{t\rightarrow 0}(1)\big)t^{-(\frac{1}{2}+\nu)(s-s_c)}.
\end{equation*}

\medskip
\noindent
{\rm (ii) Boundedness below the scaling regularity.}
For every $1\leq s<s_c$,
\begin{equation}\label{low singularity}
\limsup_{t\to 0}\|\psi(t)\|_{\dot H^s}<\infty.
\end{equation}

\noindent
{\rm (iii) Limiting radiation profile.} There exists a radial function $\zeta^\ast$ such that
\begin{equation}\label{regularity zeta star main}
\zeta^\ast\in\dot H^s(\mathbb R^d)\qquad\text{for every }s_c<s<s_c+2\nu|\lambda_+|,
\end{equation}
and
\begin{equation}\label{radiation convergence main}
\lim_{t\to 0}
\|\zeta(t)-\zeta^\ast\|_{\dot H^s}=0
\end{equation}
for every $s$ satisfying
\begin{equation*}
s_c<s\leq\sigma,\qquad s<s_c+\frac{4\nu}{1+2\nu}.
\end{equation*}}
\end{theorem}   
   
\bigbreak

\noindent{\em Comments on Theorem \ref{main thm}.}
\bigbreak

\noindent{\em 1. On the assumptions.}
The Joseph--Lundgren condition
\begin{equation*}
    2p+1>p_{\mathrm{JL}}(d)
\end{equation*}
provides real characteristic exponents, a nonoscillatory expansion of $Q$, and the Hardy coercivity of the linearized operators $L_{\pm}$. These properties supply both the asymptotic structure needed for matching and the energy control needed for the remainder equation. The restriction $p\in\mathbb N^{\ast}$ makes
\begin{equation*}
    F(u)=|u|^{2p}u
\end{equation*}
polynomial in $u$ and $\bar{u}$. This allows the nonlinear powers and oscillatory phases to be tracked explicitly throughout the construction.

Assumption $(\mathrm{H1})$ excludes resonances in the inner construction up to order $K$. Since the characteristic exponents at infinity are $\lambda_\pm$ for $L_+$ and $0,\lambda_++\lambda_-$ for $L_-$, the assumption prevents the exponents arising at each order of the inner approximation from coinciding with the corresponding characteristic exponents. Resonances in the self-similar and outer constructions are treated separately, and the resulting logarithmic terms are included in the corresponding expansions.

\medskip

\noindent{\em 2. Effects of the slow decay of the ground state.}
The algebraic tail of $Q$ affects the construction of the approximate solution in all three regions. In the inner region, the different asymptotic behaviors
$$\Lambda Q(r)\sim r^{-\frac{1}{p}+\lambda_{+}},\qquad Q(r)\sim r^{-\frac{1}{p}},\qquad\text{as $r\rightarrow \infty$},$$
imply that the real and imaginary components of the inner iterative system have different asymptotic orders at infinity. This asymmetry is accounted for at each step by choosing the initial conditions to eliminate the unwanted homogeneous components. The resulting expansion of the approximate solution contains the full asymptotic structure of $Q$ and supplies the matching conditions for the self-similar construction.

In the self-similar region, the algebraic tail of $Q$ leads to an expansion around the singular homogeneous profile $c_\infty r^{-\frac1p}$, with linearized operator
\begin{equation*}
    \mathcal L(f,\overline f) =\left(-\Delta+\frac{2}{p}y^{-1}\partial_y-pc_\infty^{2p}y^{-2}+\frac{i}{2}y\partial_y\right)f-pc_\infty^{2p}y^{-2}\overline {f}.
\end{equation*}
The conjugate coupling links the two oscillatory branches: the positive phase branch generates a negative phase component whose amplitude is four powers of $y$ lower. This weaker branch must be considered because it enters the subsequent nonlinear interactions and outer matching.

Finally, the slow decay of $Q$ changes the outer expansion. For every \(r>0\),
\begin{equation*}
    \lambda(t)^{\frac1p}Q(\lambda(t)r)\longrightarrow c_\infty r^{-\frac1p}\qquad\text{as }t\to0.
\end{equation*}
The nonzero limit makes the nonlinear contribution of the nonoscillatory outer profile remain at leading order and requires the corrections $t^kf_k$. The outer expansion therefore combines this nonlinear contribution with both oscillatory branches and the additional oscillatory terms generated by their interactions.

\medskip

\noindent{\em 3. Transfer of asymptotic information across three scales.}
The construction of the approximate solution uses the three space variables
\begin{equation*}
    \rho=\lambda(t)r,\qquad y=t^{-\frac{1}{2}}r,\qquad  r,
\end{equation*}
corresponding respectively to the inner, self-similar, and remote regions. Matching transfers the asymptotic information carried by $Q$ from the concentrating core to the oscillatory profiles and the physical scale radiation.

To our knowledge, this is the first construction of such supercritical blowup regimes to treat the slow algebraic tail of $Q$, both oscillatory branches, and the first-order linear correction by combining matched asymptotic expansions with energy estimates. 

Gluing the local approximations by smooth cutoffs gives a global approximate solution. With $N_{in}=K$ fixed by $(\mathrm{H1})$, we choose $N_{ss}$ and $N_{out}$ sufficiently large so that, for every fixed $\frac{d}{2}<s<s_c+2\nu|\lambda_+|$, the error $R^{ap}$ for the approximate solution satisfies
\begin{equation*}%\label{global residual derivative estimate}
\begin{aligned}
\big\|R^{ap}(t)\big\|_{\dot{H}^s}\leq C_s\Big(&t^{2\nu(K+1)-(\frac12+\nu)(2+s-s_c)}\\
&+t^{2K\epsilon_1+(\nu-\epsilon_1)|\lambda_+|-1+(\frac12+\epsilon_1)(s_c-s)}
\Big),\quad 0<t\leq T.
\end{aligned}
\end{equation*}
This finite order estimate provides the error control used in the construction of the exact solution.

\medskip

\noindent{\em 4. The linear estimate for the remainder equation.}
The scaling $\lambda(t)$ introduces the term
\begin{equation*}
    -\frac{(1+2\nu)i}{4\nu\tau}(\frac{1}{p}+\rho\partial_{\rho}) f,\qquad\tau=\frac1{2\nu}t^{-2\nu},
 \end{equation*}
into the remainder equation. The approximate solution also contributes a linear perturbation whose leading part is determined by the first inner correction:
$$\mathcal F_1(f)=\frac1{2\nu\tau}\mathbf B(f)+\hat{\mathcal F}_1(f).$$
Both effects produce $\tau^{-1}$ contributions that are not controlled by the coercivity of the linearized operators alone. We introduce a quadratic correction to the first-order energy that cancels the leading contribution of $\frac1{2\nu\tau}\mathbf B(f)$. At the same time, the Morawetz correction provides the positive weighted control needed to cancel the localized contributions generated by the scaling term. Together with the higher-order energies, these corrections yield a backward propagator estimate with a polynomial loss.

\medskip With $N_{in}=K$ fixed, suitable choices of $\epsilon_1$, $N_{ss}$, and $N_{out}$ ensure that the error estimate of the approximate solution remains integrable after the polynomial loss in the linear estimate. The contraction mapping principle then yields a unique exact solution with the prescribed asymptotic behavior.

\medskip
\noindent{\em 5. On the role of the topology and the limiting radiation.}
The behavior of the blowup solution is fundamentally different on the two sides of the scaling regularity. For $s_c<s<\sigma$, the ground state belongs
to $\dot H^s$, and the radiation satisfies
\begin{equation*}
    \lambda(t)^{s_c-s}\|\zeta(t)\|_{\dot H^s}\to0.
\end{equation*}
The concentrating ground state therefore determines the leading behavior of the $\dot H^s$ norm:
\begin{equation*}
    \|\psi(t)\|_{\dot H^s}=\bigl(\|Q\|_{\dot H^s}+o(1)\bigr)t^{-\left(\frac12+\nu\right)(s-s_c)}.
\end{equation*}
For $1\leq s<s_c$, the core and the radiation have nonintegrable tails separately. Their cancellation in the full solution is therefore essential to the boundedness of $\|\psi(t)\|_{\dot H^s}$.

This cancellation also determines the limiting radiation. The outer approximation converges to $f_0$, while for every fixed $r>0$, the concentrating ground state satisfies
\begin{equation*}
    \lambda(t)^{\frac1p}Q(\lambda(t)r)\longrightarrow c_\infty r^{-\frac1p},\quad t\rightarrow 0.
\end{equation*}
By the decomposition \eqref{form psi}, the preceding limits give $\zeta^{\ast}(r)=f_0(r)-c_\infty r^{-\frac1p}$. The matching conditions at the origin give
\begin{equation*}
    \zeta^\ast(r)=O\left(r^{-\frac1p+2\nu|\lambda_+|}\right),
    \qquad r\to0,
\end{equation*}
while the cutoff of $f_0$ implies
\begin{equation*}
    \zeta^\ast(r)\sim-c_\infty r^{-\frac1p},
    \qquad r\to\infty.
\end{equation*}
These two asymptotics yield the natural regularity range
\begin{equation*}
   \zeta^\ast\in\dot H^s(\mathbb R^d),\qquad s_c<s<s_c+2\nu|\lambda_+|
\end{equation*}
of the limiting radiation. The strong convergence \eqref{radiation convergence main} proved here holds in the range
$$s_c<s\leq\sigma,\qquad s<s_c+\frac{4\nu}{1+2\nu}.$$

\medskip

\noindent{\em 6. Other type II blow-up regimes for supercritical NLS.}\,
Merle--Rapha\"{e}l--Rodnianski \cite{MerleRaphaelRodnianski15} construct quantized type II blowup regimes near $Q$ on finite codimensional manifolds of initial data. The present theorem establishes existence with continuously prescribed rates: for every $\nu>\nu_0(d,p)$ and every admissible integer $s$, we construct a radial finite time blowup solution in $C\bigl((0, T],\allowbreak \dot H^1(\mathbb R^d)\cap\allowbreak \dot H^s(\mathbb R^d)\bigr)$ with
\begin{equation*}
\lambda(t)=t^{-\frac12-\nu},
\end{equation*}
and determine its asymptotic structure from the concentrating core through the self-similar oscillations to the limiting radiation in physical variables.

\bigbreak

\medskip

\subsection{Strategy of the proof.} We now outline the overall strategy, which is implemented in detail throughout Section \ref{sec 2}- Section \ref{sec 6}. The proof of Theorem \ref{main thm} proceeds in two steps. The first step is a construction of a sufficiently good approximate solution of \eqref{NLS} that has the form \eqref{form psi}, \eqref{high singularity}, in the spirit of the work of Perelman \cite{Perelman14}. The second step is to construct the exact solution by solving the terminal Duhamel equation for a small remainder tending to zero as $\tau\to+\infty$. The main tools are energy estimates for the linearized evolution, proved in Section \ref{sec 6}, which yield a backward propagator bound and allow us to close the contraction argument.

\medskip To this end, we start by constructing a radial approximate solution and give the explicit oscillatory behavior in the self-similar region. 
The core of this step is to prove the following Proposition: 

\begin{proposition}\label{main prop}
{\sl Let $d\geq11$ and $p\in\mathbb N^{\ast}$ satisfy $2p+1>p_{JL}$, and assume $(\mathrm{H1})$. Fix $\nu>0$ and an integer $\sigma$ satisfying
$$\frac {d}{2}<\sigma<s_c+2\nu|\lambda_{+}|.$$
Let
\begin{equation*}
0<\epsilon_1<\nu, \qquad 0<\epsilon_2<\min\big\{\nu, \frac12\big\}, \qquad N_{\rm in}=K.
\end{equation*}
There exist sufficiently large integers
\begin{equation*}
N_{\rm ss}\geq
\max\Big\{2,\big\lfloor\frac{K-1}{2}\big\rfloor\Big\}, \qquad N_{\rm out}\geq 2,
\end{equation*}
and $\delta_0>0$ such that, for every $0<\delta\leq\delta_0$, there exist $T>0$ and a radial approximate solution
\begin{equation*}
    \psi^{ap}\in C^\infty\bigl((0,T]\times\mathbb R^d;\mathbb C\bigr)
\end{equation*}
to \eqref{NLS} with the following properties.

\medskip
\noindent
{\rm (i)} The approximate solution has the form
\begin{equation*}%\label{global approximate form}
    \psi^{ap}(t,x) =\lambda^{\frac1p}(t)\bigl(Q(\rho)+\chi^{ap}(t,\rho)\bigr), \qquad \rho=\lambda(t)|x|,\quad \lambda(t)=t^{-\frac{1}{2}-\nu}.
\end{equation*}
For all integers $k\geq 0$ and $0\leq l\leq\sigma$,
\begin{equation*}%\label{global Linfty estimate}
\bigl\|\langle\rho\rangle^{-k}\partial_\rho^l\chi^{ap}(t)\bigr\|_{L^\infty(\mathbb R^d)}\leq C_{k,l}\big(t^{\frac{\nu}{p}}+t^{(\nu-\epsilon_1)(\frac1p+|\lambda_+|)}\big),
\end{equation*}
and
\begin{equation*}
\begin{aligned}
&\Big\|\nabla_\rho^l\big(\chi^{ap}(t,\rho)+\Big(1-\Theta\big(\frac{\rho}{\lambda(t)\delta}\big)\big)Q(\rho)\Big)\Big\|_{L^2(\mathbb R^d)}
\leq Ct^{-(\frac12+\nu)(s_c-l)},&& 1\leq l<s_c,\\%\label{global low Sobolev estimate}\\
&\Big\|\nabla_\rho^l\chi^{ap}(t)\Big\|_{L^2(\mathbb R^d)}\leq C\Big(t^{2\nu}+t^{(1-2\epsilon_2)\nu|\lambda_+|}+t^{(\nu+\epsilon_2)(l-s_c)}\Big),&&s_c<l\leq\sigma.%\label{global high Sobolev estimate}
\end{aligned}
\end{equation*}

\medskip
\noindent
{\rm (ii)} The corresponding error 
$$R^{ap}=-i\partial_t \psi^{ap}-\Delta \psi^{ap}-|\psi^{ap}|^{2p}\psi^{ap}$$
satisfies, for every integer $1\leq l\leq\sigma$,
\begin{equation*}%\label{global residual derivative estimate}
\begin{aligned}
\big\|\nabla_x^{l}R^{ap}(t)\big\|_{L^2(\mathbb R^d)}\leq C_l\Big(&t^{2\nu(K+1)-(\frac12+\nu)(l+2-s_c)}\\
&+t^{2K\epsilon_1+(\nu-\epsilon_1)|\lambda_+|-1+(\frac12+\epsilon_1)(s_c-l)}
\Big),\quad 0<t\leq T.
\end{aligned}
\end{equation*}
In particular,
\begin{equation*}%\label{whole error est}
\begin{aligned}
\big\|R^{ap}(t)\big\|_{\dot H^1\cap\dot H^\sigma(\mathbb R^d)}\leq C\Big(&t^{2\nu(K+1)-(\frac12+\nu)(\sigma+2-s_c)}\\
&+t^{2K\epsilon_1+(\nu-\epsilon_1)|\lambda_+|-1+(\frac12+\epsilon_1)(s_c-\sigma)}\Big),\quad 0<t\leq T.
\end{aligned}
\end{equation*}}
\end{proposition}

\medskip The construction of $\psi^{ap}(t)$ will be achieved by considering separately the three regions that correspond to three different space scales: the inner region corresponding to $ |x| \leq t^ {1/2+ \epsilon_1}$,  the self-similar region defined by $ \frac 1 { 10} \, t^ { 1/2+\epsilon_1} \leq |x| \leq 10 \, t^ {1/2- \epsilon_2}$, and the remote region where $ |x|\geq t^ {1/2-\epsilon_2}$. Here $0< \epsilon_1< \nu$ and $0< \epsilon_2<\min\{\nu, \frac12\}$ are chosen later. The inner region is characterized as the zone of blow-up concentration, where the solution is viewed as a perturbation of the profile $t^{-\frac{1}{p}(1/2+\nu)} Q\big(\frac{x}{t^{1/2+\nu}}\big)$. In the self-similar region, the profile is uniquely determined by the matching conditions coming from the inner region. In the remote region, the profile remains essentially a free parameter $\delta$ of the construction.

\medskip 
 In Section \ref{sec 2}, we investigate the equation in the inner region $ |x| \leq t^ {1/2+ \epsilon_1}$. In this regime, we seek an approximate solution of the form
\begin{equation*}%\label{strategy inner ansatz}
    \psi_{\mathrm{in}}^{ap}(t,x) =t^{-\frac{1}{p}(\frac{1}{2}+\nu)}\Big( Q(\rho)+\sum_{k=1}^{K}t^{2\nu k}\chi_{k}(\rho)\Big),\quad \rho=t^{-\frac{1}{2}-\nu}|x|.
\end{equation*}
The functions $\chi_k$ for  $1\leq k \leq K$ are determined recursively by solving the system:
$$ \left\{\begin{array}{l} 
H\vec{\chi}_{k}=\mathcal{D}_{k}(\vec{\chi}_{1},\cdots,\vec{\chi}_{k-1}), \\
\chi_k(0)= i^kC_k ,\qquad\chi_k'(0)= 0,
\end{array}\right.$$
where the real constants $C_k$ are selected in Subsection \ref{subsec 2.2} to eliminate the unwanted homogeneous components at infinity. $H$ is the linearized operator around the ground state, defined by:
\begin{equation*}%\label{linQ}
H=-\Delta\sigma_3-(p+1)Q^{2p}\sigma_3-pQ^{2p}\sigma_3\sigma_1,
 \end{equation*}
with the Pauli matrices
\begin{equation*}
\  \sigma_1=\left(\begin{matrix}
0 &1\\
1 &0
\end{matrix}\right), \quad\sigma_3=\left(\begin{matrix}
1 &0\\
0 &-1
\end{matrix}\right).
\end{equation*}
As established in Subsection \ref{subsec 2.2}, the expansions of $Q$ and the fundamental solutions of $L_\pm$ yield termwise differentiable expansions of $\chi_k$ at infinity. For $1\leq k\leq K$,
\begin{equation*}
\chi_k(\rho)=O\bigl(\rho^{-\frac1p+2k+\lambda_+}\bigr), \qquad \rho\to+\infty.
\end{equation*}
Under $(\mathrm{H1})$, these expansions contain no logarithmic terms. The constants $C_k$ eliminate the unwanted homogeneous components, and the resulting profiles determine the matching data for the self-similar region. We use this approximation for $\rho\leq10t^{-(\nu-\epsilon_1)}$.

\medskip The aim of Section  \ref{sec 3} is to extend the approximate solution built in Section \ref{sec 2} to the self-similar region $ \frac {1}{10} t^ { 1/2+\epsilon_1} \leq |x| \leq 10 t^  {1/2-\epsilon_2}$. Taking into account the matching conditions from the inner region,  we shall look for an  approximate solution of the form:
\begin{align*}
\psi_{ss}^ {ap}( t, x) = t^{-\frac{1}{2p}}y^{-\frac{1}{p}} \big(c_{\infty}+v_{ss}^{ap}(t,y)\big),\,\,\,\,\,\,y=t^{-\frac{1}{2}}|x|,
\end{align*}
where $\alpha_{m,n}=|\lambda_{+}|m+|\lambda_{-}|n$ and
\begin{equation}\label{form ss intro}
 \begin{aligned}
v_{ss}^{ap}(t,y)=\sum\limits_{m+n\geq 1}&t^{\nu\alpha_{m,n}} A_{m,n}^{0,0}(y)+\sum\limits_{\substack{m,n\geq 1\\l\geq 0}}t^{\nu(\alpha_{m,n}+2+4l)}A_{m,n}^{l,1}(y)\\
&+\sum\limits_{\substack{m\geq 0,n\geq 1\\l\geq 0}}t^{\nu(\alpha_{m,n}+4+4l)}A_{m,n}^{l,2}(y).
\end{aligned}
\end{equation}
The expansion is organized by the profiles $A_{m,n}^{l,j}$, with indices in $\mathcal S$ defined in \eqref{def admissible}. They satisfy the recursive system
\begin{equation*}
\mathcal L(A_{m,n}^{l,j},\overline{A_{m,n}^{l,j}})
+i\mu_{m,n}^{l,j}A_{m,n}^{l,j}=g_{m,n}^{l,j},
\end{equation*}
where $\mu_{m,n}^{l,j}=-\nu(\alpha_{m,n}+2j+4l)$. The expansions at the origin match the inner profiles through $k+2l+j\leq K$. Beyond this range, logarithmic terms are allowed, with the homogeneous constants fixed as in Lemma \ref{lem ss 0}.

At infinity, the homogeneous solutions have a nonoscillatory branch of order $y^{-2\mu_{m,n}^{l,j}}$ and two coupled oscillatory branches
with respective leading orders
\begin{equation*}
e^{\frac{i}{4}y^2}y^{-2s_c+2\mu_{m,n}^{l,j}}, \qquad e^{-\frac{i}{4}y^2}y^{-2s_c+2\mu_{m,n}^{l,j}-4}.
\end{equation*}
Nonlinear interactions generate additional phases, whose orders are given in Lemma \ref{lem ss infinity}. These expansions determine
the outer matching data. We use the self-similar approximation for $10 t^{-\epsilon_1}\leq y\leq \frac{1}{10}t^{-\epsilon_2}$, with $0<\epsilon_2<\min\{\nu,\frac12\}$.

\medskip
In Section \ref{sec 4}, we construct an approximate solution $\psi_{out}^ {ap}$ that extends $\psi_{ss}^ {ap}$ to the whole space by solving the equation \eqref{NLS} with adapted Cauchy data in the remote region $|x|\geq t^ {1/2-\epsilon_2}$. 
Let $\Theta$ be a radial smooth cutoff function on $\mathbb{R}^d$ that equals $1$ on the unit ball and vanishes outside the ball of radius $2$. For a sufficiently small $\delta > 0$, we define $\Theta_\delta(r) = \Theta_0(r/\delta)$. With $\Omega$ defined
in \eqref{def outer Omega}, based on \eqref{form ss intro} and the expansion from the self-similar region, we seek an approximate solution in the following form:
\begin{equation*}%\label{third}
\begin{aligned}
 \psi_{out}^ {ap}(t, x)=& f_0(r)+\sum_{k=1}^{N_{out}}t^kf_k(r)+e^{\frac{ir^2}{4t}}\sum_{k=0}^{N_{out}}t^{s_c+k}A_{1,k}^{+}(t,r)+e^{-\frac{ir^2}{4t}}\sum_{k=0}^{N_{out}}t^{s_c+k+2}A_{1,k}^{-}(t,r)\\
+&e^{\frac{ir^2}{4t}}\sum_{\substack{2\leq q\leq 2p+1\\q-1\in2\mathbb{Z}}}\sum_{\substack{0\leq h\leq q\\ \mathcal{A}_{q,1,h}\neq \emptyset}}\sum_{k=0}^{N_{out}}t^{qs_c+2h+k}A_{q,1,h,k}(t,r)\\
&+\sum_{\substack{(q,s,h)\in\Omega\\2\leq q\leq N_{\rm out}}}e^{\frac{isr^2}{4t}}\sum_{k=0}^{N_{out}}t^{qs_c+2h+k}A_{q,s,h,k}(t,r),\qquad r=|x|,
\end{aligned}
\end{equation*}
where 
\begin{equation*}%\label{eq f0}
f_0(r)=\Theta(\frac{r}{\delta})r^{-\frac{1}{p}}\Big(c_{\infty}+\sum_{\substack{(m,n,l,j)\in\mathcal S\\m,n,l\leq N_{\rm ss}}}\beta_{m,n}^{l,j}r^{2\nu(\alpha_{m,n}+2j+4l)}\Big),
\end{equation*}
is determined by the matching conditions from the self-similar region and 
\begin{equation*}
\begin{aligned}
A_{1,k}^{\pm}(t, r)=&\sum_{(m,n,l,j)\in\mathcal{S}}t^{2\nu(\alpha_{m,n}+2j+4l)}g_{m,n}^{l,j,k,\pm}(r),\\
A_{q,s,h,k}(t,r)=&\sum_{(m,n,l,j)\in\mathcal{S}_q}t^{2\nu(\alpha_{m,n}+2j+4l)}g_{q,s,h,m,n}^{l,j,k}(r).
\end{aligned}
\end{equation*}
A specific feature of the energy supercritical problem is the nonvanishing trace
\begin{equation*}
    \lambda^{\frac1p}(t)Q(\lambda(t)r)
    \longrightarrow
    c_\infty r^{-\frac1p}
\end{equation*}
on fixed physical scales. Consequently, the nonlinear contribution of the non-oscillatory part survives at leading order and requires the corrections \(t^k f_k\). The outer expansion also retains both oscillatory branches and the higher harmonics generated by their interactions. The recursive systems treat $s=0$, $s=1$, and $s\notin\{0,1\}$ separately. Matching fixes the homogeneous integration constants, and resonant integrations produce the logarithmic terms described in Subsection \ref{subsec 4.2}.

\medskip In Section \ref{sec 5} we complete the proof of Proposition \ref{main prop} by gluing the local approximations:
\begin{equation*}%\label{approximate psi intro}
\begin{aligned}
 \psi^{ap}(t,x):=& \Theta \big( t^{-\frac{1}{2}- \epsilon_1}r\big)\psi_{in}^ {ap}(t,x) +\big(1-\Theta \big( t^{-\frac{1}{2}- \epsilon_1}r\big)\big)\Theta \big( t^{-\frac{1}{2}+ \epsilon_2}r\big)\psi_{ss}^{ap}(t,x)\\
&+\big(1-\Theta \big( t^{-\frac{1}{2}+ \epsilon_2}r\big)\big)\psi_{out}^{ap}(t,x).
\end{aligned}
\end{equation*}
The local estimates and matching bounds in Section \ref{sec 2}- Section \ref{sec 4} yield the required global estimates. With $N_{in}=K$ fixed, we choose $N_{ss}$ and $N_{out}$ sufficiently large,
\begin{equation*}
\big\|R^{ap}(t)\big\|_{\dot{H}^l} \leq C_l\big(t^{2\nu(K+1)-(\frac12+\nu)(l+2-s_c)}+t^{2K\epsilon_1+(\nu-\epsilon_1)|\lambda_+|-1+(\frac12+\epsilon_1)(s_c-l)}\big),
\end{equation*}
for every fixed $\frac d2<s<s_c+2\nu|\lambda_+|$. Together with the bounds on $\psi^{ap}$, this completes the proof of Proposition \ref{main prop}.

\medskip Finally, Section \ref{sec 6} is devoted to lifting the approximate solution to an exact one. For $\nu$ and $\sigma$ as in Theorem \ref{main thm}, we write
\begin{align*}%\label{real solution intro}
\psi(t,x)=\psi^{ap}(t,x)+\lambda^{\frac{1}{p}}(t)f(\tau,y),\qquad y=\lambda(t)x,\,\,\,\,\tau=\frac{1}{2\nu}t^{-2\nu},
\end{align*}
where $\psi^{ap}$ is the approximate solution constructed in Sections \ref{sec 2}--\ref{sec 5}. 
%Substituting \eqref{approximate psi intro} and \eqref{real solution intro} into \eqref{NLS} yields an equation for the remainder $f$. Its analysis is based on suitable energy estimates for the linear equation
The associated linear equation is
\begin{equation*}
    i\partial_\tau f= -Hf-\frac{1+2\nu}{4\nu\tau}i\Lambda f+\mathcal F_1(f),
\end{equation*}
where $\mathcal F_1$ is the linear perturbation generated by the approximate solution. The assumptions on $d$ and $p$ imply the coercivity of $H$. A quadratic correction to the first-order energy cancels the leading linear contribution generated by $t^{2\nu}\chi_1$. Together with localized Morawetz corrections and higher-order energies, this yields the
backward propagator estimate in Proposition \ref{prop linear est}:
\begin{equation*}
    \|U(\tau,s)f\|_{\dot H^1\cap\dot H^\sigma} \lesssim\left(\frac{s}{\tau}\right)^{\frac{1+2\nu}{4\nu}(\sigma-s_c)+\epsilon}\|f\|_{\dot H^1\cap\dot H^\sigma},\qquad 0<\tau\leq s,
\end{equation*}
for sufficiently small $\epsilon>0$. The equation for $f$ is then written in the integral form
\begin{equation*}
    f(\tau) = i\int_\tau^\infty U(\tau,s) \big(\mathcal F_2(f)(s)+r(s)\big)\ud s,
\end{equation*}
where $\mathcal F_2$ is at least quadratic in $f$, and $r$ is the renormalized error of the approximate solution. We solve this equation in the weighted space
\begin{equation*}
X:=\Big\{f\in C([\tau_0,\infty),\,\dot{H}^{1}\cap\dot{H}^\sigma): \sup_{\tau\geq\tau_0}\tau^{\gamma}\|f(\tau)\|_{\dot{H}^1\cap\dot{H}^{\sigma}}<\infty\Big\},
\end{equation*}
where $\max\{1,\frac{1+2\nu}{4\nu}(\sigma-s_c)+\epsilon\}<\gamma<K$. Choosing $\epsilon_1$ as in \eqref{range of epsilon1} makes the renormalized error integrable with this weight after the polynomial loss in the linear estimate. For sufficiently large $\tau_0$, the integral operator is a contraction on a suitable closed ball of $X$. It therefore yields a unique remainder $f$, and hence a unique exact solution in the prescribed asymptotic class.

 \subsection{Acknowledgments} I would like to thank my advisor Lifeng Zhao for providing many invaluable suggestions. X. Li is supported
by National Natural Science Foundation of China under contract No. 12271497.

  \section{Approximate solution in the inner region}\label{sec 2}
\medbreak
\subsection{General scheme of the construction of the approximate solution in the inner region.} \label{subsec 2.1}
Define
\begin{align*}%\label{def inner region}
\mathcal{S}_{in}=\{(t,x)|0\leq |x|\leq 10t^{\frac{1}{2}+\epsilon_1}\},
\end{align*}
(where $0<\epsilon_1<\nu$ is a positive real number to be fixed later). In this section, we construct a family of radial approximate solutions $\psi_{in}^{ap}$ to \eqref{NLS} as a perturbation of the profile $\lambda^{\frac{1}{p}}(t)Q(\lambda(t)x)$. 
\medbreak
Writing 
\begin{align}\label{inner decomposition}
\psi(t,x)=\lambda^{\frac{1}{p}}(t)\big(Q(\rho)+\chi(t,\rho)\big),\quad \rho=\lambda(t)|x|,%=\lambda^{\frac{1}{p}}(t)(Q(\rho)+\chi(\rho,t)),
\end{align}
and $\lambda(t)=t^{-\frac{1}{2}-\nu}$. Performing a straightforward computation, we obtain
\begin{align}\label{eq of U in}
it^{1+2\nu}\chi_t=-\Delta \chi+(\frac{1}{2}+\nu)i t^{2\nu}\big(\Lambda Q+\Lambda \chi\big)+Q^{2p+1}-|Q+\chi|^{2p}(Q+\chi)
\end{align}
where the operator $\Lambda =\frac{1}{p}+\rho\cdot\nabla_{\rho}$.

\medskip In what follows, we seek a solution $\chi$ to \eqref{eq of U in} in the form
\begin{equation}\label{ansatz of chi}
\begin{aligned}
\chi(t,\rho)=\sum\limits_{k\geq 1}t^{2\nu k}\chi_k(\rho).
\end{aligned}
\end{equation}
By inserting this ansatz into \eqref{eq of U in}, we deduce the following system of equations for the vector $\vec{\chi}_k=\left(\begin{array}{l}
\chi_k\\
\overline{\chi_k}
\end{array}\right)$ with $k\geq 1$,
\begin{equation}\label{eq of chi k in}
\begin{aligned}
0=t^{2\nu}(H\vec{\chi}_1+\mathcal{N}_0)+\sum\limits_{k\geq 2}t^{2\nu k}\big(H\vec{\chi}_k+\mathcal{N}_1(\chi)\big)+\mathcal{N}_2(\chi)
\end{aligned}
\end{equation}
where
\begin{align*}
&H=-\Delta\sigma_3-(p+1)Q^{2p}\sigma_3-pQ^{2p}\sigma_3\sigma_1, \,\,\,
\sigma_1=\left(\begin{array}{cc}
0 & 1\\
1 & 0
\end{array} \right),\,\,
\sigma_3=\left(\begin{array}{cc}
1 & 0\\
0 & -1
\end{array} \right),\\
&\mathcal{N}_i(\chi)=\left(\begin{array}{l}
N_i(\chi)\\
-\overline{N_i(\chi)}
\end{array}\right),\,\,\,i=0,1,2,\\
&N_0=(\frac{1}{2}+\nu)i\Lambda Q,\quad N_1(\chi)=(\frac{1}{2}+\nu)i\Lambda\chi_{k-1}-2\nu i(k-1)\chi_{k-1},
\end{align*}
and
\begin{align*}
&N_2(\chi)=-|Q+\chi|^{2p}(Q+\chi)+Q^{2p+1}+(p+1)Q^{2p}\chi+pQ^{2p}\overline{\chi}.
\end{align*}

\medskip Therefore, we solve the following recurrent equations for $k\geq 1$,
\begin{equation}\label{chi ab system}
\begin{aligned}
H\vec{\chi}_{k}=\mathcal{D}_{k}(\vec{\chi}_{1},\cdots,\vec{\chi}_{k-1}),\quad \mathcal{D}_{k}=\left(\begin{array}{l}
D_{k} \\
-\overline{D_{k}}
\end{array}\right),
\end{aligned}
\end{equation}
subject to the initial conditions 
\begin{equation}\label{inner initial condtions}
\begin{aligned}
\chi_{k}(0)=i^{k}C_k,\quad\text{and}\quad \chi'_{k}(0)=0,
\end{aligned}
\end{equation}
where $C_k$ is to be determined subsequently, and $D_{k}$ depends on $Q$, $\chi_j$ for $j=1,\cdots ,k-1$. Furthermore, the source term $D_k$ can be decomposed into two parts:
\begin{align*}
 D_{k}=D_{k}^{(1)}+D_{k}^{(2)},
\end{align*}
where $D_{k}^{(1)}$ and $D_{k}^{(2)}$ are being contributions of $N_0+N_1(\chi)$ and $N_2(\chi)$ respectively. In view of \eqref{eq of chi k in}, we have 
\begin{equation}\label{D11 form}
D_{1}^{(1)}=-(\frac{1}{2}+\nu)i\Lambda Q,\,\,\,\,\,\,\,\text{and}\,\,\,\,\, \,\,D_{1}^{(2)}=0,
\end{equation}
 while for $k\geq 2$, $D_{k}^{(1)}$ takes the explicit form
\begin{equation}\label{def: D1}
\begin{aligned}
D_{k}^{(1)}(\rho)=&-(\frac{1}{2}+\nu)i\Lambda\chi_{k-1}(\rho)+2\nu i(k-1)\chi_{k-1}(\rho),
\end{aligned}
\end{equation}
and $D_{k}^{(2)}$ is determined by 
\begin{equation}\label{def: D2}
\begin{aligned}
N_2(\chi)=-\sum\limits_{k\geq 2}t^{2\nu k}D_{k}^{(2)}(\rho).
\end{aligned}
\end{equation}
\medskip

\subsection{Analysis of the functions $\chi_k$.}\label{subsec 2.2}
We subject \eqref{chi ab system} to initial conditions \eqref{inner initial condtions}. The goal of the present paragraph is to prove the following result.

\begin{lemma}\label{lem: inner region}
{\sl The system \eqref{chi ab system} admits a unique solution $\{\chi_{k}\}_{1\leq k\leq K}$ verifying 
\begin{itemize}
  \item[(1)]  For any $1\leq k\leq K$, $\chi_{k}$ is a $\mathcal{C}^{\infty}$-function on $\mathbb{R}_{+}$ with an even Taylor expansion at $0$.  \item[(2)] As $\rho\rightarrow+\infty$, the function $\chi_{k}$ has an asymptotic expansion of the following form:
  \begin{equation}\label{inner expansion infinity k1}
  \begin{aligned}
\chi_{1}(\rho)=&i\rho^{-\frac{1}{p}+2}\Big(\sum\limits_{m+n\geq 1}c_{m,n}^{1,0,0}\rho^{\lambda_{+}m+\lambda_{-}n}+\rho^{-2}\sum\limits_{m,n\geq 1}c_{m,n}^{1,0,1}\rho^{\lambda_{+}m+\lambda_{-}n}\Big),
\end{aligned}
\end{equation}
and, for $2\leq k\leq K$
  \begin{equation}\label{inner expansion infinity k2}
  \begin{aligned}
\chi_{k}(\rho)=i^k\rho^{-\frac{1}{p}+2k}\Big(&\sum\limits_{m+n\geq 1}c_{m,n}^{k,0,0}\rho^{\lambda_{+}m+\lambda_{-}n}+\rho^{-2}\sum\limits_{\substack{m,n\geq 1\\0\leq l\leq [\frac{k-1}{2}]}}c_{m,n}^{k,l,1}\rho^{\lambda_{+}m+\lambda_{-}n-4l}\\
&+\rho^{-4}\sum\limits_{\substack{m\geq 0,n\geq 1\\0\leq l\leq [\frac{k-2}{2}]}}c_{m,n}^{k,l,2}\rho^{\lambda_{+}m+\lambda_{-}n-4l}\Big),
\end{aligned}
\end{equation}
where the coefficients $c_{m,n}^{k,l,j}$ are real for $j=0,1,2$. Furthermore, this asymptotic expansion can be differentiated any number of times with respect to $\rho$.
\end{itemize}}
\end{lemma}
\bigbreak

\medbreak

We now proceed with the proof of Lemma \ref{lem: inner region}. For $k\geq 1$, define
\begin{equation}\label{def G}
v_{k}^{+}=\Re{\chi_k},\,\,\,\,v_{k}^{-}=\Im{\chi_k},\,\,\,\,\text{and}\,\,\,\,G_{k}^{+}=\Re{D_k},\,\,\,\,G_{k}^{-}=\Im{D_k}.
\end{equation}
Solving equation \eqref{chi ab system} is then equivalent to solving the system
\begin{equation}\label{system vk}
L_{\pm}v_{k}^{\pm}=G_{k}^{\pm},
\end{equation}
where $L_{\pm}$ are the linearized operators 
\begin{equation*}%\label{Linearized op}
	L_{+} = - \Delta - (2p+1) Q^{2p},\qquad L_{-} = - \Delta - Q^{2p}.
\end{equation*}
Consequently, the subsequent analysis in Lemma \ref{lem: inner region} relies frequently on the asymptotic expansions of the basis solutions to $L_{\pm}f=0$, which are given as follows.
\medbreak

   \begin{lemma}\label{fundamental-inner} 
   {\sl Let $d\geq11$ and $p\in\mathbb N^{\ast}$ satisfy $2p+1>p_{JL}$. Choose the fundamental solutions $\{\Phi_{\pm}, \Theta_{\pm}\}$ of $L_{\pm}$ by
 \begin{equation*}%\label{basis inner}
 \begin{aligned}
 \Theta_-(r)= Q(r),\quad \Theta_+(r)=p\Lambda Q,\quad \Phi_{\pm}=\Theta_{\pm}\int_{r}^{\infty}\Theta_{\pm}^{-2}(s)s^{1-d}\ud s.
 \end{aligned}
 \end{equation*}
 Then, $\Theta_{\pm}(0)=1$ and $\Theta'_{\pm}(0)=0$ and 
 \begin{equation*}
    W[\Theta_{\pm},\Phi_{\pm}]=\Theta_{\pm}'(r)\Phi_{\pm}(r)-\Theta_{\pm}(r)\Phi_{\pm}'(r)=r^{1-d}.
\end{equation*}
In addition, as $r\rightarrow \infty$, 
  \begin{equation}\label{basis inner -}
\Theta_-(r)= Q(r),\,\,\,\,\,\Phi_-(r) =r^{-\frac{1}{p}+\lambda_{-}+\lambda_{+}}\sum\limits_{m+n\geq 0}\tilde{c}_{m,n}r^{\lambda_{+}m+\lambda_{-}n} ,
\end{equation}
and
\begin{equation}\label{basis inner +}
 \begin{aligned}
\Theta_+(r) =&  r^{-\frac{1}{p}}\Big(d_{0,0}^1\,r^{\lambda_{+}}+\sum\limits_{m+n\geq 1}d_{m,n}^1 r^{\lambda_{+}m+\lambda_{-}n}\Big),\\
  \Phi_+(r) =& r^{-\frac{1}{p}}\Big(d_{0,0}^{2}\,r^{\lambda_{-}}+\sum\limits_{m+n\geq 1}d_{m,n}^2 r^{\lambda_{+}m+\lambda_{-}n}\Big),
 \end{aligned}
\end{equation}
with some coefficients $\tilde{c}_{m,n}$, $d_{m,n}^j$, j=1,2 verifying $\tilde{c}_{0,0}=\frac{1}{2(s_c-1)c_{\infty}}$, $d_{0,0}^{1}=p\lambda_{+} c_{1,0}$ and $d_{0,0}^{2}=\frac{1}{p\lambda_{+}c_{1,0}(\lambda_{+}-\lambda_{-})}$.
}
\end{lemma}

\medskip

\begin{proof}[Proof of Lemma \ref{lem: inner region}]
{\bf Step 1. The case $k=1$.}
By \eqref{inner initial condtions}, \eqref{D11 form} and \eqref{system vk},
\begin{equation*}
L_{+}v_1^{+}=0,\qquad L_{-}v_1^{-}=-(1/2+\nu)\Lambda Q,
\end{equation*}
with initial condition
\begin{equation*}
v_1^{+}(0)=\partial_{\rho}v_{1}^{+}(0)=0,\quad\text{and}\quad v_{1}^{-}(0)=C_1,\,\,\, \partial_{\rho}v_1^{-}(0)=0.
\end{equation*}
According to the assumption $Q(0)=1$, one has $v_{1}^{+}(\rho)=0$ and 
\begin{equation*}
v_{1}^{-}(\rho)=C_1\Theta_{-}(\rho)+(\frac{1}{2}+\nu)\int_{0}^{\rho}\Big(\Theta_{-}(\rho)\Phi_{-}(s)-\Theta_{-}(s)\Phi_{-}(\rho)\Big)\Lambda Q(s)s^{d-1}\ud s.
\end{equation*}
Since $\Lambda Q$ is a $C^{\infty}$ function at $\rho = 0$, one can easily write a Taylor expansion for $v_1^{-}$, which proves (i) for $k = 1$. According to \eqref{basis inner -}, $\lambda_{+}<-2$ and $2s_c=2-\lambda_+-\lambda_-$, 
\begin{equation*}
\Phi_{-}(s)\Lambda Q(s)s^{d-1}\sim s^{\lambda_{+}+1},
\end{equation*}
we then choose the initial data
\begin{equation*}
C_1=-(\frac{1}{2}+\nu)\int_{0}^{\infty}\Phi_{-}(s)\Lambda Q(s)s^{d-1}\ud s.
\end{equation*}
Therefore,
\begin{equation*}
v_{1}^{-}(\rho)=-(\frac{1}{2}+\nu)\Big(\Theta_{-}(\rho)\int_{\rho}^{\infty}\Phi_{-}(s)\Lambda Q(s)s^{d-1}\ud s+\Phi_{-}(\rho)\int_{0}^{\rho}\Theta_{-}(s)\Lambda Q(s)s^{d-1}\ud s\Big),
\end{equation*}
as $\rho\rightarrow \infty$, the asymptotic behavior of $v_{1}^{-}$ is
\begin{equation*}
v_{1}^{-}(\rho)=\rho^{-\frac{1}{p}+2}\sum\limits_{m+n\geq 1}c_{m,n}^{1,0,0}\rho^{\lambda_{+}m+\lambda_{-}n}+\rho^{-\frac{1}{p}}\sum\limits_{m,n\geq 1}c_{m,n}^{1,0,1}\rho^{\lambda_{+}m+\lambda_{-}n}.
\end{equation*}
Combining this expansion with $v_{1}^{+}(\rho)=0$, one has
\begin{equation}\label{expansion chi1}
\begin{aligned}
\chi_{1}(\rho)=i\rho^{-\frac{1}{p}+2}\Big(\sum\limits_{m+n\geq 1}c_{m,n}^{1,0,0}\rho^{\lambda_{+}m+\lambda_{-}n}+\rho^{-2}\sum\limits_{m,n\geq 1}c_{m,n}^{1,0,1}\rho^{\lambda_{+}m+\lambda_{-}n}\Big),
\end{aligned}
\end{equation}
with some real coefficients $c_{m,n}^{1,0,j}$ for $j=0,1$.

\medskip

{\bf Step 2. In the case $k=2, 3$.} We treat the case $k=2$ first. By \eqref{def: D1} and \eqref{def: D2}, one can check that  $D_2$ is a smooth real valued radial function at the origin. Using \eqref{expansion chi1}, we obtain the following asymptotic expansion as $\rho\rightarrow \infty$:
\begin{equation}\label{expansion D21}
\begin{aligned}
D_2^{(1)}=&\rho^{-1/p+2}\Big(\sum\limits_{m+n\geq 1}d_{m,n}^{2,1,0}\rho^{\lambda_{+}m+\lambda_{-}n}+\rho^{-2}\sum\limits_{m,n\geq 1}d_{m,n}^{2,1,1}\rho^{\lambda_{+}m+\lambda_{-}n}\Big)
\end{aligned}
\end{equation}
and
\begin{equation}\label{expansion D22}
\begin{aligned}
D_{2}^{(2)}(\rho)=\rho^{-1/p+2}\Big(&\sum\limits_{m+n\geq 2}d_{m,n}^{2,2,0}\rho^{\lambda_{+}m+\lambda_{-}n}+\rho^{-2}\sum\limits_{\substack{m+n\geq 3\\m,n\geq 1}}d_{m,n}^{2,2,1}\rho^{\lambda_{+}m+\lambda_{-}n}\\
&+\rho^{-4}\sum\limits_{m,n\geq 2}d_{m,n}^{2,2,2}\rho^{\lambda_{+}m+\lambda_{-}n}\Big).
\end{aligned}
\end{equation}
Taking \eqref{expansion D21} and \eqref{expansion D22} into the definition \eqref{def G}, we find that $G_2^{-}(\rho)=0$ and
\begin{equation*}%\label{G2}
\begin{aligned}
G_2^{+}(\rho)=\rho^{-1/p+2}\Big(&\sum\limits_{m+n\geq 1}d_{m,n}^{2,0}\rho^{\lambda_{+}m+\lambda_{-}n}+\rho^{-2}\sum\limits_{m,n\geq 1}d_{m,n}^{2,1}\rho^{\lambda_{+}m+\lambda_{-}n}\\
&+\rho^{-4}\sum\limits_{m,n\geq 2}d_{m,n}^{2,2}\rho^{\lambda_{+}m+\lambda_{-}n}\Big),
\end{aligned}
\end{equation*}
where the coefficients $d_{m,n}^{2,j}$, $j=0,1,2$ are determined by the coefficients appearing in the asymptotic expansions \eqref{expansion D21} and \eqref{expansion D22}.  Solving $L_{\pm}v_2^{\pm}=G_2^{\pm}$ with the initial data
\begin{equation*}
v_2^{+}(0)=-C_2,\quad\partial_{\rho}v_2^{+}(0)=0,\quad\text{and}\quad v_2^{-}(0)=\partial_{\rho}v_2^{-}(0)=0.
\end{equation*}

\medskip For $v_2^{+}$, since $Q(0)=1$, we have $\Theta_{+}(0)=p\Lambda Q(0)=1$. Thus, the representation formula gives
\begin{equation*}
v_2^{+}(\rho)=-C_2\Theta_{+}(\rho)-\int_{0}^{\rho}\Big(\Theta_{+}(\rho)\Phi_{+}(s)-\Theta_{+}(s)\Phi_{+}(\rho)\Big)G_2^{+}(s)s^{d-1}\ud s.
\end{equation*}
It implies that $v_2^{\pm}$ are $C^{\infty}$ functions at zero.  According to the definitions \eqref{def Discr} and \eqref{def lambda pm}, under the assumption \eqref{ass H1}, for $m+n\geq 1$ with $m,n\geq 0$, we obtain
%\begin{equation*}
%\frac{1}{2}\Big(\lambda_{+}m+\lambda_{-}n-\lambda_{+}\Big)=\frac{1}{2}\Big(-(s_c-1)(m+n)+\frac{\sqrt{Discr}}{2}(m-n)+|\lambda_{+}|\Big)\notin \mathbb{N},
%\end{equation*}
%thus, for $s_c>1$,
\begin{equation}\label{k2 not resonance}
\lambda_{+}m+\lambda_{-}n-\lambda_{\pm}\neq -4.
\end{equation}
To preserve regularity at the origin, define the function $\eta\in C^\infty([0,\infty))$ such that
\begin{equation*}
\eta(\rho)=0\quad \text{for $\rho\leq1$},\qquad\eta(\rho)=1\quad\text{for $\rho\geq2$}.
\end{equation*}
 Therefore, we write 
\begin{equation}\label{decomp v2+}
\begin{aligned}
v_2^{+}(\rho)=\eta(\rho)\rho^{-\frac{1}{p}+4}\sum\limits_{\substack{m+n\geq 1\\\lambda_{+}m+\lambda_{-}n>-4+\lambda_{+}}}v_{m,n}^{2,0}\rho^{\lambda_{+}m+\lambda_{-}n}+\tilde{v}_2^{+}(\rho),
\end{aligned}
\end{equation}
where the constants $v_{m,n}^{2,0}$ depend on $d_{m',n'}^{2,0}$ for $0\leq m'\leq m$, $0\leq n'\leq n$.
Using the fact 
\begin{equation}\label{resonance +}
L_{+}(\rho^{-\frac{1}{p}+a})=-\rho^{-\frac{1}{p}+a-2}\Big((a-\lambda_{+})(a-\lambda_{-})+(2p+1)\sum\limits_{m+n\geq 1}c_{a}^{m,n,+}\rho^{\lambda_{+}m+\lambda_{-}n}\Big),
\end{equation}
then, by \eqref{k2 not resonance}, for $a=\lambda_{+}m+\lambda_{-}n+4$ with $a>\lambda_{+}$, the equation for $\tilde{v}_2^{+}$ is
\begin{equation}\label{eq tilde v2}
L_{+}\tilde{v}_2^{+}=\tilde{G}_2^{+},
\end{equation}
where
\begin{equation*}
\tilde G_2^+(\rho)=G_2^+(\rho)-L_+\Big(\eta(\rho)\rho^{-\frac{1}{p}+4}\sum\limits_{\substack{m+n\geq 1\\\lambda_{+}m+\lambda_{-}n>-4+\lambda_{+}}}v_{m,n}^{2,0}\rho^{\lambda_{+}m+\lambda_{-}n}\Big),
\end{equation*}
with the following asymptotic expansion as $\rho\rightarrow\infty$:
\begin{equation*}
\begin{aligned}
\tilde{G}_{2}^{+}(\rho)=\rho^{-\frac{1}{p}+2}\Big(&\sum_{\substack{m+n\geq 1\\ \lambda_{+}m+\lambda_{-}n<-4+\lambda_{+}}}\tilde{d}_{m,n}^{2,0}\rho^{\lambda_{+}m+\lambda_{-}n}+\sum\limits_{m,n\geq 1}d_{m,n}^{2,1}\rho^{\lambda_{+}m+\lambda_{-}n-2}\\
&+\sum\limits_{m,n\geq 2}d_{m,n}^{2,2}\rho^{\lambda_{+}m+\lambda_{-}n-4}\Big).
\end{aligned}
\end{equation*}

\medskip Choose the initial data
\begin{equation*}%\label{def initial C2}
C_2=-\int_{0}^{\infty}\Phi_{+}(s)\tilde{G}_2^{+}(s)s^{d-1}\ud s,
\end{equation*} 
then, the solution of \eqref{eq tilde v2} is
\begin{equation*}
\tilde{v}_2^{+}(\rho)=\Theta_{+}(\rho)\int_{\rho}^{\infty}\Phi_{+}(s)\tilde{G}_2^{+}(s)s^{d-1}\ud s+\Phi_{+}(\rho)\int_{0}^{\rho}\Theta_{+}(s)\tilde{G}_2^{+}(s)s^{d-1}\ud s.
\end{equation*}
By \eqref{basis inner +} and \eqref{k2 not resonance}, as $\rho\rightarrow+\infty$, it has the expansion
\begin{equation}\label{expan tilde v2}
\begin{aligned}
\tilde{v}_2^{+}(\rho)=\rho^{-\frac{1}{p}+4}\Big(&\sum\limits_{\substack{m+n\geq 1\\\lambda_{+}m+\lambda_{-}n<-4+\lambda_{+}}}\tilde{c}_{m,n}^{2,0,0}\rho^{\lambda_{+}m+\lambda_{-}n}+\sum\limits_{m,n\geq 1}\tilde{c}_{m,n}^{2,0,1}\rho^{\lambda_{+}m+\lambda_{-}n-2}\\
&+\sum\limits_{m,n\geq 2}\tilde{c}_{m,n}^{2,0,2}\rho^{\lambda_{+}m+\lambda_{-}n-4}\Big)+\rho^{-\frac{1}{p}}\sum_{m\geq 0,n\geq 1}c_{m,n}^{2,0,2}\rho^{\lambda_{+}m+\lambda_{-}n}.
\end{aligned}
\end{equation}
Combining \eqref{expan tilde v2} with \eqref{decomp v2+}, and taking into account that $v_{2}^{-}(\rho)=0$, this yields
\begin{equation}\label{expansion chi2}
\begin{aligned}
\chi_2(\rho)=\rho^{-\frac{1}{p}+4}\Big(&\sum\limits_{m+n\geq 1}c_{m,n}^{2,0,0}\rho^{\lambda_{+}m+\lambda_{-}n}+\sum\limits_{m,n\geq 1}c_{m,n}^{2,0,1}\rho^{\lambda_{+}m+\lambda_{-}n-2}\\
&+\sum\limits_{m\geq 0,n\geq 1}c_{m,n}^{2,0,2}\rho^{\lambda_{+}m+\lambda_{-}n-4}\Big),\quad \rho\rightarrow +\infty.
\end{aligned}
\end{equation}

\medskip Similarly, for the case $k=3$, we solve $L_{\pm}v_3^{\pm}=G_3^{\pm}$ with the initial data
\begin{equation*}
v_3^{-}(0)=-C_3,\quad\partial_{\rho}v_3^{-}(0)=0,\quad\text{and}\quad v_3^{+}(0)=\partial_{\rho}v_3^{+}(0)=0.
\end{equation*}
Then, by \eqref{expansion Q}, \eqref{expansion chi1} and \eqref{expansion chi2}, $G_3^{+}$ is zero and $G_3^{-}$ has the expansion
\begin{equation}\label{expansion G3-}
\begin{aligned}
G_3^{-}(\rho)=\rho^{-1/p+4}\Big(&\sum\limits_{m+n\geq 1}d_{m,n}^{3,0}\rho^{\lambda_{+}m+\lambda_{-}n}+\rho^{-2}\sum\limits_{m,n\geq 1}d_{m,n}^{3,1}\rho^{\lambda_{+}m+\lambda_{-}n}\\
&+\rho^{-4}\sum\limits_{m\geq 0,n\geq 1}d_{m,n}^{3,2}\rho^{\lambda_{+}m+\lambda_{-}n}+\rho^{-6}\sum_{m\geq 1,n\geq 2}
 d_{m,n}^{3,3}\rho^{\lambda_{+}m+\lambda_{-}n}
\Big),
\end{aligned}
\end{equation}
as $\rho$ goes to infinity, thus, $v_{3}^{+}=0$. Using \eqref{basis inner -} and \eqref{expansion G3-}, we find that
\begin{equation*}
\begin{aligned}
\Phi_{-}(\rho)G_3^{-}(\rho)\rho^{d-1}\sim\rho^{5}\Big(&\sum\limits_{m+n\geq 1}\hat{d}_{m,n}^{3,0}\rho^{\lambda_{+}m+\lambda_{-}n}+\sum\limits_{m,n\geq 1}\hat{d}_{m,n}^{3,1}\rho^{\lambda_{+}m+\lambda_{-}n-2}\\
&+\sum\limits_{m\geq 0,n\geq 1}\hat{d}_{m,n}^{3,2}\rho^{\lambda_{+}m+\lambda_{-}n-4}+\sum_{m\geq 1,n\geq 2}
 \hat{d}_{m,n}^{3,3}\rho^{\lambda_{+}m+\lambda_{-}n-6}\Big).
\end{aligned}
\end{equation*}
Since $\lambda_{-}<\lambda_{+}<-2$ and \eqref{ass H1},
\begin{equation}\label{k3 not resonance}
\lambda_{+}m+\lambda_{-}n+6\neq \{0,\lambda_{+}+\lambda_{-}\},\quad m+n\geq 1,
\end{equation}
and 
\begin{equation*}
3+\lambda_{-}+\lambda_{+}<-1,\quad \text{and}\quad 1+\lambda_{-}<-1.
\end{equation*}
Thus, write
\begin{equation}\label{decomp v3-}
\begin{aligned}
v_3^{-}(\rho)=\eta(\rho)\rho^{-\frac{1}{p}+6}\sum\limits_{\substack{m+n\geq 1\\\lambda_{+}m+\lambda_{-}n>-6}}v_{m,n}^{3,0}\rho^{\lambda_{+}m+\lambda_{-}n}+\tilde{v}_3^{-}(\rho),
\end{aligned}
\end{equation}
where the constants $v_{m,n}^{3,0}$ depend on $d_{m',n'}^{3,0}$ for $0\leq m'\leq m$, $0\leq n'\leq n$.
Using the fact 
\begin{equation}\label{resonance -}
L_{-}(\rho^{-\frac{1}{p}+a})=-\rho^{-\frac{1}{p}+a-2}\Big(a(a-\lambda_{+}-\lambda_{-})+\sum\limits_{m+n\geq 1}c_{a}^{m,n,-}\rho^{\lambda_{+}m+\lambda_{-}n}\Big),
\end{equation}
then
\begin{equation*}%\label{eq tilde v3}
L_{-}\tilde{v}_3^{-}=\tilde{G}_3^{-},
\end{equation*}
where
\begin{equation*}
\tilde G_3^{-}(\rho)=G_3^{-}(\rho)-L_-\Big(\eta(\rho)\rho^{-\frac{1}{p}+6}\sum\limits_{\substack{m+n\geq 1\\\lambda_{+}m+\lambda_{-}n>-6}}v_{m,n}^{3,0}\rho^{\lambda_{+}m+\lambda_{-}n}\Big),
\end{equation*}
with the following asymptotic expansion 
\begin{equation*}
\begin{aligned}
\tilde{G}_{3}^{-}(\rho)=\rho^{-\frac{1}{p}+4}\Big(&\sum_{\substack{m+n\geq 1\\ \lambda_{+}m+\lambda_{-}n<-6}}\tilde{d}_{m,n}^{3,0}\rho^{\lambda_{+}m+\lambda_{-}n}+\sum\limits_{m,n\geq 1}d_{m,n}^{3,1}\rho^{\lambda_{+}m+\lambda_{-}n-2}\\
&+\sum\limits_{m\geq 0,n\geq 1}d_{m,n}^{3,2}\rho^{\lambda_{+}m+\lambda_{-}n-4}+\sum_{m\geq 1,n\geq 2}
 d_{m,n}^{3,3}\rho^{\lambda_{+}m+\lambda_{-}n-6}\Big),\quad \text{as $\rho\rightarrow +\infty$.}
\end{aligned}
\end{equation*}

\medskip By \eqref{decomp v3-}, choose $C_3=-\int_{0}^{\infty}\Phi_{-}(s)\tilde{G}_3^{-}(s)s^{d-1}\ud s$, use \eqref{k3 not resonance}, we obtain
\begin{equation*}
\tilde{v}_3^-(\rho)=\Theta_-(\rho)\int_\rho^\infty\Phi_-(s)\widetilde G_3^-(s)s^{d-1}\,\ud s+\Phi_-(\rho)\int_0^\rho\Theta_-(s)\widetilde G_3^-(s)s^{d-1}\,\ud s.
\end{equation*}
This solution is regular at the origin and satisfies $\tilde v_3^-(0)=-C_3$ and $\partial_\rho\tilde v_3^-(0)=0$. The choice of $C_3$ eliminates the independent $\Theta_-$
component at infinity. Therefore, $\chi_3$ has the expansion
\begin{equation*}%\label{expan tilde v3}
\begin{aligned}
\chi_3(\rho)=i\rho^{-\frac{1}{p}+6}\Big(&\sum\limits_{m+n\geq 1}c_{m,n}^{3,0,0}\rho^{\lambda_{+}m+\lambda_{-}n}+\sum\limits_{\substack{m,n\geq 1\\0\leq l\leq 1}}c_{m,n}^{3,l,1}\rho^{\lambda_{+}m+\lambda_{-}n-2-4l}\\
&+\sum\limits_{m\geq 0,n\geq 1}c_{m,n}^{3,0,2}\rho^{\lambda_{+}m+\lambda_{-}n-4}\Big),\quad \text{as $\rho\rightarrow +\infty$.}
\end{aligned}
\end{equation*}

{\bf Step 3. In the case $k\geq 4$.} We proceed by induction and assume that we have found $\{\chi_{l}\}_{1\leq l\leq k-1}$ that verify  $(1)$, $(2)$. Then one can check that $D_{k}$  %defined by
%\begin{equation}\label{Dk}
%\begin{aligned}
%D_{k}(\rho)=&-i\big((\frac{1}{2}+\nu)\Lambda -2\nu(k-1)\big)\chi_{k-1}(\rho)\\
%&+\sum\limits_{\substack{2\leq l\leq k,\,j_1+j_2+...+j_l=k\\1\leq j_1,j_2,...,j_l\leq k-1} }d_{i,j_1,j_2,...,j_i}^{k,(2)}Q^{2p+1-l}(\rho)\chi_{j_1}(\rho)\chi_{j_2}(\rho)\cdots\chi_{j_l}(\rho),
%\end{aligned}
%\end{equation}
remains uniformly bounded of order $\mathcal{O}(1)$ at the origin. Inserting the asymptotic expansions of $\chi_j$ ($j=1,2,\cdots, k-1$) at infinity into $D_k$ yields distinct asymptotic expansions at infinity depending on whether $k$ is odd or even.  Specifically, if $k$ is even, $G_{k}^{-}(\rho)=0$ and
\begin{align}\label{even Gk}
%\begin{aligned}
G_{k}^{+}(\rho)
=&\rho^{-\frac{1}{p}+2(k-1)}\Big(\sum\limits_{m+n\geq 1}d_{m,n}^{k,0,0}\rho^{\lambda_{+}m+\lambda_{-}n}+\rho^{-2}\sum\limits_{\substack{m,n\geq 1\\0\leq l\leq [\frac{k-2}{2}]}}d_{m,n}^{k,l,1}\rho^{\lambda_{+}m+\lambda_{-}n-4l}\nonumber\\
&+\rho^{-4}\sum\limits_{\substack{m\geq 0,n\geq 1\\0\leq l\leq [\frac{k-3}{2}]}}d_{m,n}^{k,l,2}\rho^{\lambda_{+}m+\lambda_{-}n-4l}+\rho^{-2k}\sum\limits_{m,n\geq 2}d_{m,n}^{k,\frac{k-2}{2},2}\rho^{\lambda_{+}m+\lambda_{-}n}\Big).
%\end{aligned}
\end{align}
Conversely, if $k$ is odd, $G_{k}^{+}(\rho)=0$ and
\begin{align}\label{odd Gk}
%\begin{aligned}
G_{k}^{-}(\rho)
=&\rho^{-\frac{1}{p}+2(k-1)}\Big(\sum\limits_{m+n\geq 1}d_{m,n}^{k,0,0}\rho^{\lambda_{+}m+\lambda_{-}n}+\rho^{-2}\sum\limits_{\substack{m,n\geq 1\\0\leq l\leq [\frac{k-2}{2}]}}d_{m,n}^{k,l,1}\rho^{\lambda_{+}m+\lambda_{-}n-4l}\nonumber\\
&+\rho^{-4}\sum\limits_{\substack{m\geq 0,n\geq 1\\0\leq l\leq [\frac{k-3}{2}]}}d_{m,n}^{k,l,2}\rho^{\lambda_{+}m+\lambda_{-}n-4l}+\rho^{-2k}\sum\limits_{m\geq 1, n\geq 2}d_{m,n}^{k,\frac{k-1}{2},2}\rho^{\lambda_{+}m+\lambda_{-}n}\Big).
%\end{aligned}\Big)
\end{align}
The system $L_{\pm}v_k^{\pm}=G_k^{\pm}$ is solved subject to the initial conditions
\begin{equation}\label{inner initial k}
\begin{aligned}
\chi_{k}(0)=i^{k}C_k,\qquad \partial_{\rho}\chi_k(0)=0.
\end{aligned}
\end{equation}

\medskip Under the assumption \eqref{ass H1}, if $k$ is even, we write
\begin{equation}\label{decom vk+}
\begin{aligned}
v_{k}^{+}(\rho)=\rho^{-\frac{1}{p}+2k}\Big(&\sum\limits_{\substack{m+n\geq 1\\ \lambda_{+}m+\lambda_{-}n>-2k+\lambda_{+}}}v_{m,n}^{k,0,0}\rho^{\lambda_{+}m+\lambda_{-}n}\\
&+\sum\limits_{\substack{m,n\geq 1, 0\leq l\leq [\frac{k-2}{2}]-1\\\lambda_{+}m+\lambda_{-}n>-2(k-1)+4l+\lambda_{+}}}v_{m,n}^{k,l,1}\rho^{\lambda_{+}m+\lambda_{-}n-2-4l}\\
&+\sum\limits_{\substack{m\geq 0,n\geq 1, 0\leq l\leq [\frac{k-3}{2}]\\\lambda_{+}m+\lambda_{-}n>-2(k-2)+4l+\lambda_{+}}}v_{m,n}^{k,l,2}\rho^{\lambda_{+}m+\lambda_{-}n-4-4l}\Big)\eta(\rho)+\tilde{v}_{k}^{+}(\rho).%\quad \text{for even $k$,}
\end{aligned}
\end{equation}
At the same time, by \eqref{ass H1}, if $k$ is odd, we write
\begin{equation*}%\label{decom vk+}
\begin{aligned}
v_{k}^{-}(\rho)=\rho^{-\frac{1}{p}+2k}\Big(&\sum\limits_{\substack{m+n\geq 1\\ \lambda_{+}m+\lambda_{-}n>-2k}}v_{m,n}^{k,0,0}\rho^{\lambda_{+}m+\lambda_{-}n}+\sum\limits_{\substack{m,n\geq 1, 0\leq l\leq [\frac{k-2}{2}]-1\\\lambda_{+}m+\lambda_{-}n>-2(k-1)+4l}}v_{m,n}^{k,l,1}\rho^{\lambda_{+}m+\lambda_{-}n-2-4l}\\
&+\sum\limits_{\substack{m\geq 0,n\geq 1, 0\leq l\leq [\frac{k-3}{2}]-1\\\lambda_{+}m+\lambda_{-}n>-2(k-2)+4l}}v_{m,n}^{k,l,2}\rho^{\lambda_{+}m+\lambda_{-}n-4-4l}\Big)\eta(\rho)+\tilde{v}_{k}^{-}(\rho),%\quad \text{for odd $k$,}
\end{aligned}
\end{equation*}
where the coefficient $v_{m,n}^{k,l,j}$, $j=0,1,2$ depends on $d_{m',n'}^{k,l,j}$ for $0\leq m'\leq m$, $0\leq n'\leq n$. Then $\tilde{v}_{k}^{\pm}$ solves the equation $L_{\pm}\tilde{v}_{k}^{\pm}=\tilde{G}_{k}^{\pm}$ with the following expansions as $\rho$ goes to infinity:
\begin{itemize}
\item If $k$ is even, using \eqref{resonance +} and \eqref{even Gk}, we have
\begin{equation*}%\label{expansion tilde Gk+}
\begin{aligned}
\tilde{G}_{k}^{+}(\rho)=\rho^{-\frac{1}{p}+2(k-1)}\Big(&\sum\limits_{\substack{m+n\geq 1\\ \lambda_{+}m+\lambda_{-}n<-2k+\lambda_{+}}}\tilde{d}_{m,n}^{k,0,0}\rho^{\lambda_{+}m+\lambda_{-}n}\\
&+\sum\limits_{\substack{m,n\geq 1, 0\leq l\leq [\frac{k-2}{2}]-1\\\lambda_{+}m+\lambda_{-}n<-2(k-1)+4l+\lambda_{+}}}\tilde{d}_{m,n}^{k,l,1}\rho^{\lambda_{+}m+\lambda_{-}n-2-4l}\nonumber\\
&+\sum\limits_{\substack{\substack{m\geq 0,n\geq 1, 0\leq l\leq [\frac{k-3}{2}]\\\lambda_{+}m+\lambda_{-}n<-2(k-2)+4l+\lambda_{+}}}}\tilde{d}_{m,n}^{k,l,2}\rho^{\lambda_{+}m+\lambda_{-}n-4-4l}\\
&+\sum_{m,n\geq 1}\tilde{d}_{m,n}^{k,[\frac{k-2}{2}],1}\rho^{\lambda_{+}m+\lambda_{-}n-2-4[\frac{k-2}{2}]}+\sum_{m,n\geq 2}d_{m,n}^{k,\frac{k-2}{2},2}\rho^{\lambda_{+}m+\lambda_{-}n-2k}\Big).
\end{aligned}
\end{equation*}
\item If $k$ is odd, by \eqref{resonance -} and \eqref{odd Gk}, we obtain
\begin{equation*}
\begin{aligned}
\tilde{G}_{k}^{-}(\rho)=\rho^{-\frac{1}{p}+2(k-1)}\Big(&\sum\limits_{\substack{m+n\geq 1\\ \lambda_{+}m+\lambda_{-}n<-2k}}\tilde{d}_{m,n}^{k,0,0}\rho^{\lambda_{+}m+\lambda_{-}n}\\
&+\sum\limits_{\substack{m,n\geq 1, 0\leq l\leq [\frac{k-2}{2}]-1\\\lambda_{+}m+\lambda_{-}n<-2(k-1)+4l}}\tilde{d}_{m,n}^{k,l,1}\rho^{\lambda_{+}m+\lambda_{-}n-2-4l}\\
&+\sum\limits_{\substack{\substack{m\geq 0,n\geq 1, 0\leq l\leq [\frac{k-3}{2}]-1\\\lambda_{+}m+\lambda_{-}n<-2(k-2)+4l}}}\tilde{d}_{m,n}^{k,l,2}\rho^{\lambda_{+}m+\lambda_{-}n-4-4l}\\
&+\sum\limits_{m,n\geq 1}\tilde{d}_{m,n}^{k,[\frac{k-2}{2}],1}\rho^{\lambda_{+}m+\lambda_{-}n-2-4[\frac{k-2}{2}]}+\sum_{m\geq 0,n\geq 1}\tilde{d}_{m,n}^{k,[\frac{k-3}{2}],2}\rho^{\lambda_{+}m+\lambda_{-}n-4-4[\frac{k-3}{2}]}\\
&+\sum_{m\geq 1,n\geq 2}d_{m,n}^{k,\frac{k-1}{2},2}\rho^{\lambda_{+}m+\lambda_{-}n-2k}
\Big).
\end{aligned}
\end{equation*}
\end{itemize}

 Thus, we select $C_k$
\begin{equation*}
\begin{split}
C_k=\left\{
\begin{aligned}
&(-1)^{\frac{k}{2}} \int_{0}^{\infty}\Phi_{+}(s)\tilde{G}_{k}^{+}(s)s^{d-1}\ud s,\,\,\,\,\,\,\text{for even $k$}, \\
&(-1)^{\frac{k-1}{2}}\int_{0}^{\infty}\Phi_{-}(s)\tilde{G}_{k}^{-}(s)s^{d-1}\ud s,\,\,\,\,\,\,\text{for odd $k$}, 
\end{aligned}
\right.
\end{split}
\end{equation*}
then \eqref{inner initial k} implies that $\chi_k$ is smooth at the origin and has an even Taylor expansion. The above choice of $C_k$ eliminates the slow homogeneous component proportional to $\Theta_\pm$ in the asymptotic expansion of $\tilde v_k^\pm$ at infinity. Furthermore, the asymptotic expansion at infinity is a direct consequence of the integral representation
\begin{equation*}
\tilde{v}_k^{\pm}(\rho)=\Phi_{\pm}(\rho)\int_{0}^{\rho}\Theta_{\pm}(s)\tilde{G}_{k}^{\pm}(s)s^{d-1}\ud s+\Theta_{\pm}(\rho)\int_{\rho}^{\infty}\Phi_{\pm}(s)\tilde{G}_{k}^{\pm}(s)s^{d-1}\ud s.
\end{equation*}
We only focus on the case where $k$ is even, as the odd case is analogous. Substituting the expansion of $\tilde{G}_{k}^{+}$ and Lemma \ref{fundamental-inner} into this expression, and combining with \eqref{ass H1} and \eqref{decom vk+}, the asymptotic expansion of $\chi_k$ at infinity is given by
\begin{equation}\label{even chi expansion}
\begin{aligned}
\chi_{k}(\rho)=&\rho^{-\frac{1}{p}}\sum\limits_{m\geq 0,n\geq 1}c_{m,n}^{k,+}\rho^{\lambda_{+}m+\lambda_{-}n}+\rho^{-\frac{1}{p}+2k}\Big(\sum\limits_{m+n\geq 1}c_{m,n}^{k,0,0}\rho^{\lambda_{+}m+\lambda_{-}n}\\
&+\sum\limits_{\substack{m,n\geq 1\\0\leq l\leq [\frac{k-2}{2}]}}c_{m,n}^{k,l,1}\rho^{\lambda_{+}m+\lambda_{-}n-2-4l}+\sum\limits_{\substack{m\geq 0,n\geq 1\\0\leq l\leq [\frac{k-3}{2}]}}c_{m,n}^{k,l,2}\rho^{\lambda_{+}m+\lambda_{-}n-4-4l}\Big).
\end{aligned}
\end{equation}
Observing the index relation $[\frac{k-1}{2}]=[\frac{k-2}{2}]$, and the identity $-4-4l=-2k$ for $l=\frac{k-2}{2}$, the expansion \eqref{even chi expansion} can be recast into the form \eqref{inner expansion infinity k2}.
\end{proof}
\bigbreak

\subsection{Estimate of the approximate solution in the inner region}\label{sebsec 2.3}
Under the above notations, we choose $N_{in} = K$,
\begin{align}\label{def psi in}
\psi_{in}^{ap}(t,x)=\lambda^{\frac{1}{p}}(t)u_{in}^{ap}(t,\rho)=\lambda^{\frac{1}{p}}(t)\Big(Q(\rho)+\chi_{in}^{ap}(t,\rho)\Big),
\end{align}
where $\chi_{in}^{ap}(t,\rho)=\sum\limits_{k=1}^{K}t^{2\nu k}\chi_k(\rho)$ with $\rho=\lambda(t)|x|$. Thanks to Lemma \ref {lem: inner region}, $\chi_{in}^{ap}$  satisfies the following $L^{\infty}$
 estimates in $\mathcal{S}_{in}$ .
 \medbreak

\begin{lemma}\label{lem: sol chi in L infinity}
{\sl For any  integers $k, l\geq 0$, there exist a positive constant $C$ and a small positive time $T=T(N_{\rm in})$ such that for all  $0 < t \leq T$, the following estimates hold: 
\begin{equation}\label{inner L infinity}  
\left\|\langle\rho\rangle^{-k}\nabla_{\rho}^{l}\chi_{in}^{ap}\right\|_{L^{\infty}(\mathcal{S}_{in})} \leq C\,t^{2\nu}.
\end{equation}
}\end{lemma}

\medbreak

\medskip Along the same lines taking advantage of Lemma \ref {lem: inner region},  we get the following $L^{2}$  estimates:

\begin{lemma}\label{lem: sol chi in L 2}\label{L 2 est inner}
{\sl Under the above notations, we have for all $0 < t \leq T$:
\begin{equation*} 
\begin{aligned}
\big\|  \langle \rho\rangle^{-k}\partial_{\rho}^{l}\chi_{in}^ {ap}\big\|_{L_{\rho}^{2}(\mathcal{S}_{in})}& \leq Ct^{(\nu-\epsilon_1)(|\lambda_{+}|-s_c+k+l)+2\epsilon_1},&\\&  \qquad  \qquad \qquad\quad \forall \,\,0\leq k+l< s_c+2-|\lambda_{+}|, 
\end{aligned}
\end{equation*}
\begin{equation*} 
\begin{aligned}
\big \| \langle \rho\rangle ^{-k} \partial_{\rho}^{l}   \chi_{in}^ {ap}\big\|_{L_{\rho}^{2}(\mathcal{S}_{in})} &\leq Ct^{2\nu},\qquad \forall \,\,k+l>s_c+2-|\lambda_{+}|.
\end{aligned}
\end{equation*}
}
\end{lemma}

\medbreak

\begin{remark}
Under assumption $(\mathrm{H1})$, with integer $p$ satisfies $2p+1>p_{JL}$, one has
$$s_c+2-|\lambda_+|=3+\frac{\sqrt{\operatorname{Discr}}}{2}\notin\mathbb Z.$$
Indeed, otherwise $\sqrt{\operatorname{Discr}}$ would be a positive even integer. By the definition \eqref{def Discr}, this implies $2/p\in\mathbb Z$, hence
$p\in\{1,2\}$. For these two values,
\begin{equation*}
\begin{aligned}
\operatorname{Discr}
=\begin{cases}
(d-8)^2-24,&p=1,\\
(d-7)^2-20,&p=2.
\end{cases}
\end{aligned}
\end{equation*}
Therefore, $\sqrt{\operatorname{Discr}}$ is a positive even integer only when $(d,p)=(13,2)$, which is excluded by the assumption $(\mathrm{H1})$.
\end{remark}

\medskip Let us end this section by estimating the remainder term:
\begin{equation}\label{new def error inner}
\begin{aligned}
\mathcal{R}_{in}^{ap}(t,\rho)=%&[\eqref{eq of U in}U_{in}^{N_{in}}](t,\rho).
-it^{1+2\nu}\partial_{t}u_{in}^{ap}+(\frac{1}{2}+\nu)it^{2\nu}\Lambda u_{in}^{ap}-\Delta u_{in}^{ap}-|u_{in}^{ap}|^{2p}u_{in}^{ap}.
\end{aligned}
\end{equation}
One has
\begin{lemma}\label {lem: error chi in}
{\sl For any integer $l\geq 1$,  there exist a positive constant~$C_{l, K}$ and a small positive  time $T=T(K)$ such that for all  $0 < t \leq T$,  the  remainder term~$\mathcal{R}_{in}^ {ap}$ satisfies 
  \begin{equation}\label{error inner 1}
  \begin{aligned}
 \big \|\nabla^{l}\mathcal{R}_{in}^{ap}\big\|_{L_{\rho}^2(\mathcal{S}_{in})}\leq C t^{2\nu+(\nu-\epsilon_1)(|\lambda_{+}|+l-s_c)+2K\epsilon_1}+Ct^{2\nu(K+1)}.
 \end{aligned}
    \end{equation}  }
\end{lemma}

\medbreak

\begin{proof}
In view of the computations carried out in the last subsection, substituting \eqref{ansatz of chi} into \eqref{new def error inner} and using the equation \eqref{eq of chi k in}, we have 
\begin{equation*}
\begin{aligned}
\mathcal{R}_{in}^{ap}(t,\rho)=\sum\limits_{K+1\leq k\leq (2p+1)K}t^{2\nu k}G_k(\rho),
\end{aligned}
\end{equation*}
where
\begin{equation*}
\begin{aligned}
G_k=&-\Delta\chi_k-(p+1)Q^{2p}\chi_k-pQ^{2p}\overline{\chi_k}\\
&+\big(\frac12+\nu\big)i\Lambda\chi_{k-1}-2\nu i(k-1)\chi_{k-1}-D_k^{(2)}.
\end{aligned}
\end{equation*}
Here $\chi_k=0$ for $k>K$, and $D_k^{(2)}$ is determined by $\chi_1,\ldots,\chi_K$ through \eqref{def: D2}.

\medskip Using the expansion of $\chi$ given in Lemma \ref{lem: inner region}, $G_k$ is $O(1)$ as $\rho$ close to $0$. And as $\rho\rightarrow +\infty$, it admits the following expansion that can be differentiated any number of times with respect to the variable $\rho$, 
\begin{equation}\label {Gk infinity}
\begin{aligned}
G_{k}(\rho) =i^k \rho^{-\frac{1}{p}+2(k-1)}\Big(&\sum\limits_{m+n\geq 1}g_{k}^{m,n,0}\rho^{\lambda_{+}m+\lambda_{-}n}+\rho^{-2}\sum\limits_{\substack{m,n\geq 1\\0\leq l\leq [\frac{k-1}{2}]}}g_{k,l}^{m,n,1}\rho^{\lambda_{+}m+\lambda_{-}n-4l}\\
&+\rho^{-4}\sum\limits_{\substack{m\geq 0,n\geq 1\\0\leq l\leq [\frac{k-2}{2}]}}g_{k,l}^{m,n,2}\rho^{\lambda_{+}m+\lambda_{-}n-4l}\Big),
\end{aligned}
\end{equation}
with some real constants $g_{k}^{m,n,0}$ and  $g_{k,l}^{m,n,j}$, $j=1,2$ that can be determined recursively in terms of the functions~$\chi_j$, for $j=1, \cdots, K$. Thus, for any $l\geq 0$, by \eqref{Gk infinity}, we deduce that
\begin{align*}
\big \|\nabla_{\rho}^{l}\mathcal{R}_{in}^{ap}\big\|^2_{L^2(\mathcal{S}_{in})}\lesssim& \sum\limits_{k=K+1}^{(2p+1)K}t^{4\nu k}\Big(C+\int_{1}^{t^{-(\nu-\epsilon_1)}}\rho^{-\frac{2}{p}+4(k-1)+2\lambda_{+}-2l}\rho^{d-1}\ud \rho\Big)\\
\leq& Ct^{4\nu+2(\nu-\epsilon_1)(|\lambda_{+}|-s_c+l)+4\epsilon_1K}+Ct^{4\nu(K+1)}.
\end{align*}
This ends the proof of the lemma.
\end{proof}

\bigbreak

\section{Approximate solution in the self-similar region}\label{sec 3}
\subsection{General scheme of the construction of the approximate solution in the  self-similar region}\label{subsec 3.1}  
This section aims to construct an approximate solution $\psi_{ss}^{ap}$ to \eqref{NLS} in the region:
\begin{align*}%\label{def self similar region}
\mathcal{S}_{\rm ss}=\big\{(t,x) \big| \frac{1}{10}t^{\frac{1}{2}+\epsilon_1}\leq |x|\leq 10t^{\frac{1}{2}-\epsilon_2}\big\},
\end{align*}
which extends the approximate solution~$\psi_{in}^ {ap}$ constructed in the inner region~$\mathcal{S}_{in}$. Here  $0<\epsilon_2<\min\{\nu,\frac12\}$ to be fixed later.

\medskip We shall look for the radial solution of \eqref{NLS} in the form: 
\begin{align}\label{exp ss 1}
\psi(t,x)=t^{-\frac{1}{2p}}y^{-\frac{1}{p}}(c_{\infty}+v(t,y)),\,\,\,\text{with}\,\,\, y=t^{-\frac{1}{2}}|x|.
\end{align}
By straightforward computations, we find that $v$ solves
\begin{align}\label{eq of v}
itv_t=-\Delta v+\frac{i}{2}y\partial_{y}v+\frac{2}{p}y^{-1}\partial_{y}v-py^{-2}c_{\infty}^{2p}(v+\bar{v})-\tilde{W},
\end{align}
with
\begin{align}\label{def tilde W}
\tilde{W}(t,y)=y^{-2}\Big(\big|c_{\infty}+v(t,y)\big|^{2p}\big(c_{\infty}+v(t,y)\big)-c_{\infty}^{2p+1}-(p+1)c_{\infty}^{2p}v(t,y)-pc_{\infty}^{2p}\overline{v}(t,y)\Big).
\end{align}
 The asymptotic behavior of the solution \eqref{exp ss 1} near the origin has to be coherent with that of \eqref{inner decomposition} at infinity. To determine this expansion, we combine the expansion of $Q$ with \eqref{inner expansion infinity k1} and \eqref{inner expansion infinity k2}. By using $\rho=t^{-\nu} y$, we obtain that as $y\rightarrow 0$,
\begin{equation*}
\begin{aligned}
\psi_{in}(t,x)=t^{-\frac{1}{2p}}y^{-\frac{1}{p}}\Big(&c_{\infty}+\sum\limits_{m+n\geq 1}t^{-\nu(\lambda_{+}m+\lambda_{-}n)}\sum\limits_{k\geq 0}i^{k}c_{m,n}^{k,0,0}y^{\lambda_{+}m+\lambda_{-}n+2k}\\
&+\sum\limits_{m,n\geq 1}t^{-\nu(\lambda_{+}m+\lambda_{-}n-2)}\sum\limits_{\substack{k\geq 1\\0\leq l\leq [\frac{k-1}{2}]}}i^{k}c_{m,n}^{k,l,1}t^{4\nu l}y^{\lambda_{+}m+\lambda_{-}n+2(k-1)-4l}\\
&+\sum\limits_{m\geq 0,n\geq 1}t^{-\nu(\lambda_{+}m+\lambda_{-}n-4)}\sum\limits_{\substack{k\geq 2\\0\leq l\leq [\frac{k-2}{2}]}}i^{k}c_{m,n}^{k,l,2}t^{4\nu l}y^{\lambda_{+}m+\lambda_{-}n+2(k-2)-4l}\Big),
\end{aligned}
\end{equation*}
where $c_{m,n}^{0,0,0}=c_{m,n}$ from the expansion of $Q$ and $c_{m,n}^{k,l,j}$ for $j=0,1,2$ defined in \eqref{inner expansion infinity k1}. Interchanging the summations, this expansion can be rewritten as
\begin{equation}\label{inner expansion in the ss region}
\begin{aligned}
\psi_{in}(t,x)=t^{-\frac{1}{2p}}y^{-\frac{1}{p}}\Big(&c_{\infty}+\sum\limits_{m+n\geq 1}t^{-\nu(\lambda_{+}m+\lambda_{-}n)}\sum\limits_{k\geq 0}i^{k}c_{m,n}^{k,0,0}y^{\lambda_{+}m+\lambda_{-}n+2k}\\
&+\sum\limits_{\substack{m,n\geq 1\\l\geq 0}}t^{-\nu(\lambda_{+}m+\lambda_{-}n-2-4l)}\sum\limits_{k\geq 0}(-1)^l i^{k+1}c_{m,n}^{k+2l+1,l,1}y^{\lambda_{+}m+\lambda_{-}n+2k}\\
&+\sum\limits_{\substack{m\geq 0,n\geq 1\\l\geq 0}}t^{-\nu(\lambda_{+}m+\lambda_{-}n-4-4l)}\sum\limits_{k\geq 0}(-1)^{l} i^{k+2}c_{m,n}^{k+2l+2,l,2}y^{\lambda_{+}m+\lambda_{-}n+2k}\Big),
\end{aligned}
\end{equation}
where $c_{m,n}^{k,0,0}=0$ for $k>K$ and $c_{m,n}^{k+2l+j,l,j}=0$ for $k+2l+j>K$, $j=1,2$.

\medskip Combining \eqref{inner expansion in the ss region} with \eqref{exp ss 1}, this leads us to look for the function $v(t,y)$ under the form: 
 \begin{equation}\label{form v}
 \begin{aligned}
v(t,y)=\sum\limits_{m+n\geq 1}&t^{\nu\alpha_{m,n}} A_{m,n}^{0,0}(y)+\sum\limits_{\substack{m,n\geq 1\\l\geq 0}}t^{\nu(\alpha_{m,n}+2+4l)}A_{m,n}^{l,1}(y)\\
&+\sum\limits_{\substack{m\geq 0,n\geq 1\\l\geq 0}}t^{\nu(\alpha_{m,n}+4+4l)}A_{m,n}^{l,2}(y),
\end{aligned}
\end{equation}
where $\alpha_{m,n}:=|\lambda_{+}|m+|\lambda_{-}|n$.
Now, substituting the ansatz \eqref{form v} into \eqref{def tilde W}, the nonlinear term has the form
\begin{equation}\label{tilde W}
\begin{aligned}
\tilde{W}(t,y)=\sum\limits_{m+n\geq 2}&t^{\nu\alpha_{m,n}}g_{m,n}^{0,0}(y)+\sum\limits_{\substack{m+n\geq 3\\m,n\geq 1, l\geq 0}}t^{\nu(\alpha_{m,n}+2+4l)}g_{m,n}^{l,1}(y)\\
&+\sum\limits_{\substack{m+n\geq 2\\n\geq 1, \,m,l\geq 0}}t^{\nu(\alpha_{m,n}+4+4l)}g_{m,n}^{l,2}(y).
\end{aligned}
\end{equation}

\medskip For $j\in\{0,1,2\}$, we define the admissible index sets:
\begin{equation*}
\begin{aligned}
S_0=&\{(m,n,l)\in \mathbb{N}^3:\, m+n\geq 1, l=0\},\\
S_1=&\{(m,n,l)\in \mathbb{N}^3:\, m,n\geq 1, l\geq 0\},\\
S_2=&\{(m,n,l)\in \mathbb{N}^3:\, m\geq 0, n\geq 1, l\geq 0\},
\end{aligned}
\end{equation*}
and 
\begin{equation}\label{def admissible}
\begin{aligned}
\tilde{\mathcal{S}}:=&\Big\{(1,0,0,0), (0,1,0,0), (1,1,l,1),(0,1,l,2): l\geq 0\Big\},\\
\mathcal{S}:=&\Big\{(m,n,l,j):\, j\in\{0,1,2\},\,(m,n,l)\in S_j\Big\}.
\end{aligned}
\end{equation}
%Then, $\widetilde{\mathcal S}$ is a subset of $\mathcal S$.

\medskip Taking \eqref{form v} and \eqref{tilde W} into the equation \eqref{eq of v}, for $(m,n,l,j)\in\mathcal{S}$, the function $A_{m,n}^{l,j}$ solves the following equation:
\begin{equation}\label{system ss}
\mathcal{L}(A_{m,n}^{l,j},\bar{A}_{m,n}^{l,j})+i\mu_{m,n}^{l,j}A_{m,n}^{l,j}=g_{m,n}^{l,j},
\end{equation}
with $g_{1,0}^{0,0}=g_{0,1}^{0,0}=g_{1,1}^{l,1}=g_{0,1}^{l,2}=0$ for $l\geq 0$. Here the operator $\mathcal{L}$ is
\begin{align*}%\label{def L ss}
\mathcal{L}(f,\bar{f})=\big(-\Delta+\frac{2}{p}y^{-1}\partial_{y}-pc_{\infty}^{2p}y^{-2}+\frac{i}{2}y\partial_{y}\big)f-pc_{\infty}^{2p}y^{-2}\bar{f},
\end{align*} 
and $\mu_{m,n}^{l,j}=-\nu(\alpha_{m,n}+2j+4l)$, $j=0,1,2$. Denote the functions
$$A_{m,n}^{l,j,+}:=\Re{A_{m,n}^{l,j}},\quad A_{m,n}^{l,j,-}:=\Im{A_{m,n}^{l,j}},$$
and 
$$g_{m,n}^{l,j,+}:=\Re{g_{m,n}^{l,j}},\quad g_{m,n}^{l,j,-}:=\Im{g_{m,n}^{l,j}}.$$
Writing \eqref{system ss} in real and imaginary parts gives
\begin{equation}\label{ss system 2}
\left\{\begin{aligned}
&\Big(-\partial_y^2-\frac{2s_c-1}{y}\partial_y-2p c_\infty^{2p}y^{-2}\Big)A_{m,n}^{l,j,+}-\Big(\frac {y}{2}\partial_y+\mu_{m,n}^{l,j}\Big)A_{m,n}^{l,j,-}=g_{m,n}^{l,j,+},\\
&\Big(-\partial_y^2-\frac{2s_c-1}{y}\partial_y\Big)A_{m,n}^{l,j,-}+\Big(\frac{y}{2}\partial_y+\mu_{m,n}^{l,j}\Big)A_{m,n}^{l,j,+}=g_{m,n}^{l,j,-}.
\end{aligned}
\right.
\end{equation}

\medbreak

\subsection{Analysis of the vector functions $A_{m,n}^{l,j}$} \label{subsec 3.2}
To construct the solutions $A_{m,n}^{l,j}$ for the system \eqref{system ss}, we first provide the fundamental solution for the homogeneous equation
\begin{equation}\label{homo ss}
\mathcal{L}(f,\bar{f})+i\mu_{m,n}^{l,j}f=0,
\end{equation}
 as stated in the following lemma.
 
\medskip

\begin{lemma}\label{lem basis of Lss 0}
{\sl Under the assumption $(\mathrm{H1})$, for each fixed $(m,n,l,j)\in\mathcal S$, equation \eqref{homo ss} admits a fundamental system $\{e_1^+,e_2^+,e_1^-,e_2^-\}$ over $\mathbb R$. As $y\rightarrow 0$, solutions have the expansions:
\begin{equation}\label{expansion basis ss 0}
\begin{aligned}
&e_1^{+}(y)=y^{\lambda_+}\big(\sum_{0\leq k\leq K}i^k a_{1,k}^{+}y^{2k}+o(y^{2K})\big), \quad  e_2^{+}(y) =y^{\lambda_-}\big(\sum_{0\leq k\leq K}i^k a_{2,k}^{+}y^{2k}+o(y^{2K})\big),\\
&e_1^{-}(y)=\sum_{0\leq k\leq K}i^{k+1} a_{1,k}^{-}y^{2k}+o(y^{2K}),\quad e_2^{-}(y)=y^{\lambda_{+}+\lambda_{-}} \big(\sum_{0\leq k\leq K}i^{k+1} a_{2,k}^{-}y^{2k}+o(y^{2K})\big),
\end{aligned}
\end{equation}
where $a_{q,0}^{\pm}=1$ and $a_{q,k}^{\pm}\in\mathbb R$ depend on $\mu_{m,n}^{l,j}$ for $q=1,2$. These asymptotics can be differentiated any number of times with respect to $y$. Furthermore, the Wronskian is
\begin{equation*}
\det\begin{pmatrix}
\Re e_1^+ & \Re e_2^+ & \Re e_1^- & \Re e_2^-\\
\partial_y\Re e_1^+ & \partial_y\Re e_2^+& \partial_y\Re e_1^- & \partial_y\Re e_2^-\\
\Im e_1^+ & \Im e_2^+ & \Im e_1^- & \Im e_2^-\\
\partial_y\Im e_1^+ & \partial_y\Im e_2^+& \partial_y\Im e_1^- & \partial_y\Im e_2^-
\end{pmatrix}=2(s_c-1)(\lambda_+-\lambda_-)\,y^{2-4s_c}.
\end{equation*}}
\end{lemma}

\medskip 
\begin{remark}\label{remark basis ss 0}
Note that the fundamental solutions $\{e_1^{\pm},e_2^{\pm}\}$ may contain $\ln y$ terms when one of the following equalities holds for some integer $k\geq 1$:
\begin{equation*}
2(s_c-1)=2k,\qquad |\lambda_{+}|=2k,\qquad |\lambda_{-}-\lambda_{+}|=2k.
\end{equation*}
Here we have used the identity $|\lambda_{+}|+|\lambda_{-}|=2s_c-2$. Under assumption $(\mathrm{H1})$, no such logarithmic terms occur through order $K$, yielding the expansions \eqref{expansion basis ss 0}.
\end{remark}

\medskip

To satisfy the matching condition imposed by the inner expansion \eqref{inner expansion in the ss region}, as $y\rightarrow 0$, the solutions $A_{m,n}^{l,j}$ have to satisfy
\begin{equation}\label{match condition 0}
A_{m,n}^{l,j}(y)=\sum_{k= 0}^{K-2l-j}i^{k+j+2l} c_{m,n}^{k+2l+j,l,j}y^{\lambda_{+}m+\lambda_{-}n+2k}+o\big(y^{\lambda_+m+\lambda_-n+2(K-2l-j)}\big),
\end{equation}
for $(m,n,l,j)\in\mathcal{S}$ with $2l+j\leq K$. The higher-order coefficients are determined by \eqref{system ss} and the prescribed homogeneous constants.

\begin{lemma}\label{lem ss 0}
{\sl There exists a unique familly of solution $\{A_{m,n}^{l,j}\}_{(m,n,l,j)\in\mathcal{S}}$ to \eqref{system ss} admitting the asymptotic expansions of the form
\begin{equation}\label{expansion ss 0}
A_{m,n}^{l,j}(y)=\sum\limits_{k\geq 0}\sum\limits_{\ell\geq 0}d_{m,n}^{l,j,k,\ell}(\ln y)^{\ell}y^{\lambda_{+}m+\lambda_{-}n+2k}, \quad y\rightarrow 0,
\end{equation}
where the coefficients satisfy
\begin{equation}\label{ss coe 1}
d_{m,n}^{l,j,k,0}=i^{k+2l+j}c_{m,n}^{k+2l+j,l,j},\qquad d_{m,n}^{l,j,k,\ell}=0,\quad \ell\geq 1, \quad k+2l+j\leq K,
\end{equation}
and
\begin{equation}\label{ss coe 2}
\begin{aligned}
\Re{d_{m,n}^{l,j,k,0}}=&0,\quad \lambda_{+}m+\lambda_{-}n+2k\in\{\lambda_{+},\, \lambda_{-}\},\\
\Im{d_{m,n}^{l,j,k,0}}=&0,\quad \lambda_{+}m+\lambda_{-}n+2k\in\{0,\, \lambda_{+}+\lambda_{-}\},
\end{aligned}
\qquad k+2l+j> K.
\end{equation}
For each fixed $(m,n,l,j)\in\mathcal S$ and $k\geq0$, only finitely many $\ell\geq 1$ such that $d_{m,n}^{l,j,k,\ell}\neq 0$. The asymptotic expansion \eqref{expansion ss 0} can be differentiated any number of times with respect to $y$.}
\end{lemma}

\begin{proof}
The first two lines of \eqref{system ss} together with the matching condition \eqref{match condition 0} yield
\begin{equation*}
A_{1,0}^{0,0}(y)=c_{1,0}^{0,0,0}e_1^{+}(y,\,\mu_{1,0}^{0,0}),\,\,\,\,\,A_{0,1}^{0,0}(y)=c_{0,1}^{0,0,0}e_2^{+}(y,\,\mu_{0,1}^{0,0}),
\end{equation*}
and for any $l\geq 0$,
\begin{equation*}
A_{1,1}^{l,1}(y)=(-1)^l c_{1,1}^{2l+1,l,1}e_2^{-}(y,\,\mu_{1,1}^{l,1}),\,\,\,\,\,A_{0,1}^{l,2}(y)=(-1)^{l+1}c_{0,1}^{2l+2,l,2}e_2^{+}(y,\,\mu_{0,1}^{l,2}),
\end{equation*}
where $e_i^{\pm}$, $i=1,2$ are defined in Lemma \ref{lem basis of Lss 0}. 

\medskip We proceed by induction to determine $A_{m,n}^{l,j}$ for all $(m,n,l,j)\in\mathcal S\setminus\widetilde{\mathcal S}$. Assume that we have found the solutions $A_{\tilde{m},\tilde{n}}^{l,j}$ for $\tilde{m}+\tilde{n}<m+n$ that satisfy the expansion \eqref{expansion ss 0}. Then, as $y\rightarrow 0$, using \eqref{expansion ss 0}, \eqref{def tilde W} and \eqref{tilde W}, $g_{m,n}^{l,j}$ admits the asymptotic expansion of the form
\begin{equation}\label{g mn expansion}
\begin{aligned}
g_{m,n}^{l,j}(y)=\sum\limits_{k\geq 0}\sum_{\ell\geq 0}\alpha_{m,n}^{l,j,k,\ell}(\ln y)^{\ell} y^{\lambda_{+}m+\lambda_{-}n+2k-2},\quad (m,n,l,j)\in \mathcal{S},
\end{aligned}
\end{equation}
where the sum over $\ell$ is finite for each fixed $k$, and $\{\alpha_{m,n}^{l,j,k,\ell}\}_{\ell\geq 1}$ is zero for $k+2l+j\leq K$.

 According to the direct computations
\begin{equation}\label{ss linear ln y}
\begin{aligned}
&\big(-\Delta+\frac2p y^{-1}\partial_y-2pc_\infty^{2p}y^{-2}\big)\big(y^a(\ln y)^\ell\big)\\
&\quad=-y^{a-2}\Big((a-\lambda_+)(a-\lambda_-)(\ln y)^\ell +\ell(2a-\lambda_+-\lambda_-)(\ln y)^{\ell-1}+\ell(\ell-1)(\ln y)^{\ell-2}\Big),\\
&\big(-\Delta+\frac2p y^{-1}\partial_y\big)\big(y^a(\ln y)^\ell\big)\\
&\quad=-y^{a-2}\Big( a(a-\lambda_+-\lambda_-)(\ln y)^\ell+\ell(2a-\lambda_+-\lambda_-)(\ln y)^{\ell-1}+\ell(\ell-1)(\ln y)^{\ell-2}\Big),
\end{aligned}
\end{equation}
and
\begin{equation*}
\big(\frac {y}{2}\partial_y+\mu_{m,n}^{l,j}\big)\big(y^a(\ln y)^\ell\big)=y^a \big(\left(\frac {a}{2}+\mu_{m,n}^{l,j}\right)(\ln y)^\ell
+\frac{\ell}{2}(\ln y)^{\ell-1}\big),
\end{equation*}
the first equation of \eqref{ss system 2} is resonant when $\lambda_{+}m+\lambda_{-}n+2k\in\{\lambda_+,\lambda_-\}$, whereas the second equation of \eqref{ss system 2} is resonant when $\lambda_{+}m+\lambda_{-}n+2k\in\{0,\lambda_{+}+\lambda_{-}\}$. A nonzero resonant forcing polynomial produces one additional power of $\ln y$.

For fixed $(m,n,l,j)\in\mathcal S$, let
 \begin{equation*}%\label{def ss 0 set}
 \mathcal{N}=\Big\{(m,n,k):\lambda_{+}m+\lambda_{-}n+2k\in\{0,\lambda_{+},\lambda_{-},\lambda_{+}+\lambda_{-}\},\,k\geq K\Big\}.
 \end{equation*}
We write
\begin{equation*}%\label{decomp ss 0}
A_{m,n}^{l,j}(y)=\sum_{(m,n,k)\in\mathcal{N}}\sum_{\ell\geq 0}\tilde{d}_{m,n}^{l,j,k,\ell}(\ln y)^{\ell+1}y^{\lambda_{+}m+\lambda_{-}n+2k}+\tilde{A}_{m,n}^{l,j}(y),
\end{equation*}
where $\tilde d_{m,n}^{l,j,k,\ell}=0$ if $(m,n,k)\notin\mathcal N$, and write
\begin{equation*}
\begin{aligned}
\tilde{d}_{m,n}^{l,j,k,\ell}=&\tilde{d}_{m,n}^{l,j,k,\ell,+}+i\tilde{d}_{m,n}^{l,j,k,\ell,-}, \quad d_{m,n}^{l,j,k,\ell}=d_{m,n}^{l,j,k,\ell,+}+id_{m,n}^{l,j,k,\ell,-},\\
&\alpha_{m,n}^{l,j,k,\ell}=\alpha_{m,n}^{l,j,k,\ell,+}+i\alpha_{m,n}^{l,j,k,\ell,-}.
\end{aligned}
\end{equation*}
If $\lambda_{+}m+\lambda_{-}n+2k\in\{\lambda_+,\lambda_-\}$, we choose
\begin{equation*}
\begin{aligned}
\tilde d_{m,n}^{l,j,k,\ell,+}=&-\frac{1}{(\ell+1)(2(\lambda_{+}m+\lambda_{-}n+2k)-\lambda_+-\lambda_-)}\\
&\qquad\times\Big(\alpha_{m,n}^{l,j,k,\ell,+}+\big(\mu_{m,n}^{l,j}+\frac{\lambda_{+}m+\lambda_{-}n+2k-2}{2}\big)d_{m,n}^{l,j,k-1,\ell,-}\\
&\qquad\quad\quad+\frac{\ell+1}{2}d_{m,n}^{l,j,k-1,\ell+1,-}+(\ell+1)(\ell+2)\tilde d_{m,n}^{l,j,k,\ell+1,+}\Big),\qquad \tilde d_{m,n}^{l,j,k,\ell,-}=0.
\end{aligned}
\end{equation*}
If $\lambda_{+}m+\lambda_{-}n+2k\in\{0,\lambda_++\lambda_-\}$, we choose
\begin{equation*}
\begin{aligned}
\tilde d_{m,n}^{l,j,k,\ell,-}=&-\frac{1}{(\ell+1)(2(\lambda_{+}m+\lambda_{-}n+2k)-\lambda_+-\lambda_-)}\\
&\qquad\times\Big(\alpha_{m,n}^{l,j,k,\ell,-}-\big(\mu_{m,n}^{l,j}+\frac{\lambda_{+}m+\lambda_{-}n+2k-2}{2}\big)
 d_{m,n}^{l,j,k-1,\ell,+}\\
 &\qquad\qquad-\frac{\ell+1}{2} d_{m,n}^{l,j,k-1,\ell+1,+}+(\ell+1)(\ell+2)\tilde d_{m,n}^{l,j,k,\ell+1,-}
\Big),\qquad\tilde d_{m,n}^{l,j,k,\ell,+}=0.
\end{aligned}
\end{equation*}
Then, applying \eqref{g mn expansion} and \eqref{ss linear ln y}, $\tilde{A}_{m,n}^{l,j}$ satisfies
\begin{equation}\label{eq tilde A ss 0}
\mathcal L\bigl(\tilde A_{m,n}^{l,j},\overline{\tilde A_{m,n}^{l,j}}\bigr)+i\mu_{m,n}^{l,j}\tilde A_{m,n}^{l,j}=\tilde g_{m,n}^{l,j},
\end{equation}
where
\begin{equation*}
\begin{aligned}
\tilde g_{m,n}^{l,j}(y)=g_{m,n}^{l,j}-\sum_{(m,n,k)\in\mathcal N}\sum_{\ell\geq0}\Big(&\mathcal{L}\big(
 \tilde d_{m,n}^{l,j,k,\ell}y^{\lambda_{+}m+\lambda_{-}n+2k}(\ln y)^{\ell+1},
 \overline{\tilde d_{m,n}^{l,j,k,\ell}} y^{\lambda_{+}m+\lambda_{-}n+2k}(\ln y)^{\ell+1}\big)\\
&+i\mu_{m,n}^{l,j} \tilde d_{m,n}^{l,j,k,\ell}y^{\lambda_{+}m+\lambda_{-}n+2k}(\ln y)^{\ell+1}\Big),
\end{aligned}
\end{equation*}
with the following asymptotic expansion as $y\rightarrow 0$:
\begin{equation*}
\begin{aligned}
\tilde{g}_{m,n}^{l,j}(y)=\sum\limits_{k\geq 0}\sum_{\ell\geq 0}\tilde{\alpha}_{m,n}^{l,j,k,\ell}(\ln y)^{\ell} y^{\lambda_{+}m+\lambda_{-}n+2k-2}.
\end{aligned}
\end{equation*}
%with some coefficients $\tilde{\alpha}_{m,n}^{l,j,k,\ell}$.

Using the fundamental system in Lemma \ref{lem basis of Lss 0}, after writing $\tilde{A}_{m,n}^{l,j}=\tilde{A}_{m,n}^{l,j,+}+i\tilde{A}_{m,n}^{l,j,-}$, the solution of \eqref{eq tilde A ss 0} has the expansion:
\begin{equation*}
\begin{aligned}
\tilde{A}_{m,n}^{l,j,+}(y)=&\sum_{\substack{k\geq0\\ \lambda_{+}m+\lambda_{-}n+2k\notin\{\lambda_+,\lambda_-\}}}\sum_{\ell\geq0}d_{m,n}^{l,j,k,\ell,+}(\ln y)^\ell y^{\lambda_{+}m+\lambda_{-}n+2k}\\
&+\sum_{\substack{k\geq0\\ \lambda_{+}m+\lambda_{-}n+2k\in\{\lambda_+,\lambda_-\}}}
d_{m,n}^{l,j,k,0,+}y^{\lambda_{+}m+\lambda_{-}n+2k},
\end{aligned}
\end{equation*}
and
\begin{equation*}
\begin{aligned}
\tilde{A}_{m,n}^{l,j,-}(y)=&\sum_{\substack{k\geq0\\ \lambda_{+}m+\lambda_{-}n+2k\notin\{0,\lambda_{+}+\lambda_{-}\}}}\sum_{\ell\geq0}d_{m,n}^{l,j,k,\ell,-}(\ln y)^\ell y^{\lambda_{+}m+\lambda_{-}n+2k}\\
&+\sum_{\substack{k\geq0\\ \lambda_{+}m+\lambda_{-}n+2k\in\{0,\lambda_{+}+\lambda_{-}\}}}
d_{m,n}^{l,j,k,0,-}y^{\lambda_{+}m+\lambda_{-}n+2k}.
\end{aligned}
\end{equation*}
The coefficients are determined recursively by the matching conditions \eqref{ss coe 1} and \eqref{ss coe 2}. Each sum over $\ell$ is finite. This completes the induction.

\end{proof}

\bigbreak
We next study the behavior of $A_{m,n}^{l,j}$ as $y \rightarrow +\infty$. To this purpose, we first give the fundamental systems for \eqref{homo ss}.

\begin{lemma}\label{lem basis of Lss infty}
{\sl  For each $(m,n,l,j)\in\mathcal S$, equation \eqref{homo ss} admits four complex valued solutions $f_{h,m,n}^{l,j,\pm}$, $h=1,2$, which form a fundamental system over $\mathbb R$. As $y\to+\infty$, they satisfy
\begin{equation}\label{ss basis infinity}
\begin{aligned}
f_{1,m,n}^{l,j,+}(y)&=y^{-2\mu_{m,n}^{l,j}} (1+O(y^{-2})),\quad f_{2,m,n}^{l,j,+}(y)=e^{\frac{iy^2}{4}} y^{-2s_c+2\mu_{m,n}^{l,j}}(1+O(y^{-2})),\\
f_{1,m,n}^{l,j,-}(y)&=iy^{-2\mu_{m,n}^{l,j}}(1+O(y^{-2})),\quad f_{2,m,n}^{l,j,-}(y)=ie^{\frac{iy^2}{4}} y^{-2s_c+2\mu_{m,n}^{l,j}} (1+O(y^{-2})).
\end{aligned}
\end{equation}
These asymptotics can be differentiated any number of times with respect to $y$. More precisely, for every integer $k\geq0$,
\begin{equation*}
\begin{aligned}
\left|\partial_y^k\left(f_{1,m,n}^{l,j,+}(y)-y^{-2\mu_{m,n}^{l,j}}\right)\right|\leq Cy^{-2\mu_{m,n}^{l,j}-2-k},
\end{aligned}
\end{equation*}
and
\begin{equation*}
\begin{aligned}
\left|\partial_y^k\left(f_{2,m,n}^{l,j,+}(y)-e^{\frac{iy^2}{4}}y^{-2s_c+2\mu_{m,n}^{l,j}}\right)\right|\leq C y^{-2s_c+2\mu_{m,n}^{l,j}+k-2}.
\end{aligned}
\end{equation*}
The same estimates hold for $f_{h,m,n}^{l,j,-}$, $h=1,2$. Furthermore, the Wronskian is defined by
\begin{equation*}
W[f_1^\pm,f_2^\pm](y)=:\det F[f_1^\pm,f_2^\pm](y),
\end{equation*}
where 
\begin{equation}\label{def F}
F[f_1^\pm,f_2^\pm](y)=\begin{pmatrix}
\Re f_1^+ & \Re f_2^+ & \Re f_1^- & \Re f_2^-\\
\partial_y\Re f_1^+ & \partial_y\Re f_2^+
& \partial_y\Re f_1^- & \partial_y\Re f_2^-\\
\Im f_1^+ & \Im f_2^+ & \Im f_1^- & \Im f_2^-\\
\partial_y\Im f_1^+ & \partial_y\Im f_2^+
& \partial_y\Im f_1^- & \partial_y\Im f_2^-
\end{pmatrix}
\end{equation}
and $f_h^\pm=f_{h,m,n}^{l,j,\pm}$, $h=1,2$. Then $W(y)=\frac14\,y^{2-4s_c}$.}
\end{lemma}

\begin{proof} For $(m,n,l,j) \in \mathcal{S}$, the linearized equation \eqref{homo ss} is
\begin{equation}\label{new homo ss}
f''+\frac{2s_c-1}{y}f'-\frac{i}{2}yf'-i\mu_{m,n}^{l,j}f+\frac{pc_{\infty}^{2p}}{y^2}\big(f+\bar{f}\big)=0.
\end{equation}
Then, \eqref{new homo ss} admits basis solutions $f_{1,m,n}^{l,j,\pm}$ exhibiting the following expansion as $y\rightarrow +\infty$
\begin{equation*}\label{ss infinity basis +}
f_{1,m,n}^{l,j,+}(y) = y^{-2\mu_{m,n}^{l,j}}\big(1+\sum_{k\geq 1}i^k\beta_{k}^{+}y^{-2k}\big),
\end{equation*}
and
\begin{equation*}\label{ss infinity basis -}
f_{1,m,n}^{l,j,-}(y) = i y^{-2\mu_{m,n}^{l,j}}\big(1+\sum_{k\geq 1}i^k\beta_{k}^{-}y^{-2k}\big),
\end{equation*}
where the real coefficients $\beta_{k}^{\pm}$ for $k\geq 1$ are uniquely determined by
\begin{equation*}
\beta_{k}^{\pm}=\frac{1}{k}\Big(\big(2\mu_{m,n}^{l,j}+2k-2\big)\big(2\mu_{m,n}^{l,j}+2k-2s_c\big)+pc_{\infty}^{2p}(1\pm(-1)^{k-1})\Big)\beta_{k-1}^{\pm},
\end{equation*}
with $\beta_{0}^{\pm}=1$.

\medskip For the oscillatory branches, according to the terms $\frac{i}{2}yf'$ and $\frac{pc_{\infty}^{2p}}{y^2}\bar{f}$, we seek basis solutions admits the following expansion as $y\rightarrow \infty$,
\begin{equation*}
f_{2,m,n}^{l,j,+}(y)=U_{m,n}^{l,j}(y)+V_{m,n}^{l,j}(y),\quad f_{2,m,n}^{l,j,-}(y)=i\big(U_{m,n}^{l,j}(y)-V_{m,n}^{l,j}(y)\big),
\end{equation*}
with 
\begin{equation*}
\begin{aligned}
U_{m,n}^{l,j}(y)=&e^{\frac{iy^2}{4}}y^{-2s_c+2\mu_{m,n}^{l,j}}\big(1+\sum\limits_{k\geq 1}\alpha_{m,n}^{l,j,k,1}y^{-2k}\big),\\
V_{m,n}^{l,j}(y)=&e^{\frac{-iy^2}{4}}y^{-2s_c-4+2\mu_{m,n}^{l,j}}\big(\alpha_{m,n}^{l,j,0,2}+\sum\limits_{k\geq 1}\alpha_{m,n}^{l,j,k,2}y^{-2k}\big).
\end{aligned}
\end{equation*}
Matching the terms of order $e^{-\frac{iy^2}{4}}y^{-2s_c-2+2\mu_{m,n}^{l,j}}$ in \eqref{new homo ss} yields:
\begin{equation*}
\Big((-\frac{i}{2})^2-\frac{i}{2}(-\frac{i}{2})\Big)\alpha_{m,n}^{l,j,0,2}+pc_{\infty}^{2p}=0,
\end{equation*}
then $\alpha_{m,n}^{l,j,0,2}=2pc_{\infty}^{2p}$. The coefficient $\alpha_{m,n}^{l,j,k,1}$ is uniquely determined by $\alpha_{m,n}^{l,j,k-1,1}$, $\bar{\alpha}_{m,n}^{l,j,k-3,2}$ for $k\geq 1$, and $\alpha_{m,n}^{l,j,k,2}$ is uniquely determined by $\bar{\alpha}_{m,n}^{l,j,k,1}$, $\alpha_{m,n}^{l,j,k-1,2}$ and $\alpha_{m,n}^{l,j,k-2,2}$.

\medskip We next compute the Wronskian for $\{f_{k,m,n}^{l,j,\pm}\}_{k=1,2}$. Writing the solution of \eqref{new homo ss} by $f=u+iv$, the equation is equivalent to
\begin{equation}\label{new homo ss 1}
Y'=A(y)Y,\qquad Y=(u,u',v,v')^{\mathsf T},
\end{equation}
where
\begin{equation*}
\begin{aligned}
A(y)=
\begin{pmatrix}
0&1&0&0\\
-\dfrac{2pc_{\infty}^{2p}}{y^2}&-\dfrac{2s_c-1}{y}&-\mu_{m,n}^{l,j}&-\dfrac {y}{2}\\
0&0&0&1\\
\mu_{m,n}^{l,j}&\dfrac {y}{2}&0&-\dfrac{2s_c-1}{y}
\end{pmatrix}.
\end{aligned}
\end{equation*}

For an integer $N\geq1$, define the truncated functions $\{f_{k,m,n}^{l,j,\pm,(N)}\}_{k=1,2}$
\begin{equation*}
\begin{aligned}
f_{1,m,n}^{l,j,+,(N)}(y) = y^{-2\mu_{m,n}^{l,j}}\Big(1+\sum_{k=1}^{N}i^k\beta_{k}^{+}y^{-2k}\Big),&\quad
f_{1,m,n}^{l,j,-,(N)}(y) = i y^{-2\mu_{m,n}^{l,j}}\Big(1+\sum_{k=1}^{N} i^k\beta_{k}^{-}y^{-2k}\Big),\\
f_{2,m,n}^{l,j,+,(N)}(y)=U_{m,n}^{l,j,(N)}+V_{m,n}^{l,j,(N)},&\quad f_{2,m,n}^{l,j,-,(N)}(y)=i\big(U_{m,n}^{l,j,(N)}-V_{m,n}^{l,j,(N)}\big),
\end{aligned}
\end{equation*}
with
\begin{equation*}
\begin{aligned}
U_{m,n}^{l,j,(N)}=&e^{\frac{iy^2}{4}}y^{-2s_c+2\mu_{m,n}^{l,j}}\Big(1+\sum\limits_{k=1}^{N}\alpha_{m,n}^{l,j,k,1}y^{-2k}\Big)\\
V_{m,n}^{l,j,(N)}=&e^{-\frac{iy^2}{4}}y^{-2s_c-4+2\mu_{m,n}^{l,j}}\Big(\alpha_{m,n}^{l,j,0,2}+\sum\limits_{k=1}^{N}\alpha_{m,n}^{l,j,k,2}y^{-2k}\Big).
\end{aligned}
\end{equation*}
And the corresponding matrix
\begin{equation*}
F_N(y)=\begin{pmatrix}
\Re f_{1,N}^+ & \Re f_{2,N}^+ & \Re f_{1,N}^- & \Re f_{2,N}^-\\
\partial_y\Re  f_{1,N}^+ & \partial_y\Re f_{2,N}^+& \partial_y\Re f_{1,N}^- & \partial_y\Re f_{2,N}^-\\
\Im f_{1,N}^+ & \Im f_{2,N}^+ & \Im f_{1,N}^- & \Im f_{2,N}^-\\
\partial_y\Im  f_{1,N}^+ & \partial_y\Im f_{2,N}^+& \partial_y\Im f_{1,N}^- & \partial_y\Im f_{2,N}^-
\end{pmatrix}
\end{equation*}
where $f_{h,N}^{\pm}=f_{h,m,n}^{l,j,\pm,(N)}$ for $h=1,2$. Then
\begin{equation}\label{WN ss}
W_N(y):=W\big[f_{1,m,n}^{l,j,\pm,(N)},f_{2,m,n}^{l,j,\pm,(N)}\big]=\det F_N=\frac{1}{4}y^{2-4s_c}\big(1+O(y^{-2})\big).
\end{equation}
Thus $F_{N}$ is invertible for sufficiently large $y$. Since it has finite sums of terms, there exists a positive $\alpha$ independent of
$N$, such that, for every $k\geq0$,
\begin{equation}\label{est FN ss}
\|\partial_y^k F_N(y)\|+\|\partial_y^k F_N^{-1}(y)\|\leq Cy^{\alpha+k},
\end{equation}
where the norm is defined by 
\begin{equation*}
\|M\|:=\sup\limits_{\|f\|_{L^\infty}=1}|Mf|,\qquad M\in\mathbb{R}^{4\times 4}.
\end{equation*}

Therefore, substituting $\{f_{h,m,n}^{l,j,(N),\pm}\}_{h=1,2}$ into the equation \eqref{new homo ss 1},
\begin{equation*}%\label{est infinity ss FN}
\left\|\partial_y^k\bigl(F_N'-A F_N\bigr)\right\|\leq C y^{\alpha-2N+k}.
\end{equation*}
Consequently, the matrix
\begin{equation*}
R_N= F_N^{-1} \bigl( F_N'-A F_N\bigr)
\end{equation*}
satisfies
\begin{equation}\label{est RN ss}
\|\partial_y^k R_N(y)\| \leq C y^{2\alpha+k-2N}.
\end{equation}

\medskip Choose $N$ such that $2N>2\alpha+1$, and lift the truncated basis solutions to the exact solutions. Write
\begin{equation}\label{def GN}
 G_N= F_N(I+B_N),
\end{equation}
then $G_N$ satisfies equation \eqref{new homo ss 1} is equivalent to
\begin{equation*}
B_N'=-R_N(I+B_N).
\end{equation*}
We impose $B_N(y)\rightarrow 0$ as $y\rightarrow +\infty$, and solve
\begin{equation}\label{BN basis infinity}
B_N(y)=\int_y^\infty R_N(s)(I+B_N(s))\,\ud s.
\end{equation}
According to \eqref{est RN ss}, for $y_0$ sufficiently large, we have $\int_{y_0}^{\infty}\|R_{N}(s)\|\ud s<\frac{1}{4}$. Thus \eqref{BN basis infinity} defines a contraction on $C_b\big([y_0,\infty),\mathbb{R}^4\times\mathbb{R}^4\big)$. Hence, there exists a unique bounded solution satisfying
$\|B_N\|_{L^\infty(y\geq y_0)}\leq1/3$. Thus \eqref{est RN ss} and \eqref{BN basis infinity} imply
\begin{equation}\label{est BN}
\|B_N(y)\|\leq2\int_y^\infty\|R_N(s)\|\,\ud s\leq C y^{2\alpha+1-2N}.
\end{equation}
In particular, if $I+B_N$ is invertible, then $G_N$ is an exact fundamental matrix. For every $k\geq0$, by  \eqref{est FN ss}, \eqref{def GN} and \eqref{est BN}
\begin{equation}\label{exact basis remainder}
\left\|\partial_y^k(G_N-F_N)\right\|\leq C y^{3\alpha+1-2N+k}.
\end{equation}

\medskip We next prove that the resulting fundamental matrix is independent of $N$. Indeed, for $M\neq N$ satisfying $2M>2\alpha+1$, we assume without loss of generality that $M>N$. By the definitions of $F_M$ and $F_N$ and \eqref{est BN}, we obtain
\begin{equation}\label{est FM FN}
\| F_M(y)- F_N(y)\|\leq C y^{\alpha-2N},\qquad \|B_M\|\leq Cy^{2\alpha+1-2N}.
\end{equation}
Using \eqref{est FN ss}, we obtain for sufficiently large $y$,
\begin{equation}\label{est limit I}
\|F_N^{-1}F_M-I\|=\| F_N^{-1}(F_M-F_N)\|\leq C y^{2\alpha-2N},
\end{equation}
goes to $0$ by $2N>2\alpha+1$. Thus, the estimates \eqref{est FM FN} and \eqref{est limit I} imply
\begin{equation}\label{inverse GN GM}
G_N^{-1}G_M=(I+B_N)^{-1}F_N^{-1}F_M(I+B_M)\rightarrow I,\quad \text{as $y\rightarrow +\infty$.}
\end{equation}
On the other hand, since $G_M$ and $G_N$ both solution of \eqref{new homo ss 1}, we deduce that
\begin{equation*}
\frac{\ud}{\ud y}
( G_N^{-1} G_M)=-G_N^{-1} G_N' G_N^{-1} G_M+G_N^{-1}G_M'=0.
\end{equation*}
Combining this identity with \eqref{inverse GN GM}, we obtain $G_N=G_M$.

\medskip Finally, by the definition \eqref{def F}, for every sufficiently large $N$,
\begin{equation*}
F= F_N(I+B_N), \qquad F'=A(y)F.
\end{equation*}
Applying \eqref{exact basis remainder} with $G_N=F$ and choosing $N$ sufficiently large for each derivative order yields
\eqref{ss basis infinity} and the stated differentiated asymptotics.

Abel's identity gives
\begin{equation*}
W'(y)=\operatorname{tr}A(y) W(y)=-\frac{4s_c-2}{y}W(y).
\end{equation*}
Hence, by \eqref{WN ss}, we obtain $W(y)=\frac{1}{4}y^{2-4s_c}$.
\end{proof}

\bigbreak

\begin{lemma}\label{lem ss infinity}
{\sl For $(m,n,l,j)\in\mathcal{S}$, given coefficients $\lambda_{m,n}^{l,j,1}$ and $\lambda_{m,n}^{l,j,2}$, there exists a unique solution to \eqref{system ss} of the following form
\begin{equation}\label{expansion A infinity}
\begin{aligned}
A_{m,n}^{l,j}(y)=&A_{m,n,1}^{l,j}(y)+A_{m,n,2}^{l,j}(y),\,\,\,(m,n,l,j)\in\tilde{\mathcal{S}},\\
A_{m,n}^{l,j}(y)=&A_{m,n,1}^{l,j}(y)+A_{m,n,2}^{l,j}(y)+A_{m,n,3}^{l,j}(y),\,\,\,(m,n,l,j)\in\mathcal{S}\setminus\tilde{\mathcal S},
\end{aligned}
\end{equation}
where $\{A_{m,n,k}^{l,j}\}_{(m,n,l,j)\in\mathcal{S}}$, $k=1,2$ solve \eqref{homo ss} and have the following asymptotic expansion as $y\rightarrow \infty$:
\begin{equation}\label{expansion tilde 12}
\begin{aligned}
A_{m,n,1}^{l,j}(y)=&y^{-2\mu_{m,n}^{l,j}}\sum\limits_{q\geq 0}a_{m,n,1}^{l,j,q}y^{-2q},\\
A_{m,n,2}^{l,j}(y)=&e^{\frac{i}{4}y^{2}}y^{-2s_c+2\mu_{m,n}^{l,j}}\sum\limits_{q\geq 0}a_{m,n,2}^{l,j,q,+}y^{-2q}+
e^{-\frac{i}{4}y^{2}}y^{-2s_c+2\mu_{m,n}^{l,j}-4}\sum_{q\geq0}a_{m,n,2}^{l,j,q,-}y^{-2q},
\end{aligned}
\end{equation}
with $a_{m,n,1}^{l,j,0}=\lambda_{m,n}^{l,j,1}$, $a_{m,n,2,+}^{l,j,0}=\lambda_{m,n}^{l,j,2}$ and $a_{m,n,2}^{l,j,0,-}=2pc_\infty^{2p}\overline{\lambda_{m,n}^{l,j,2}}$. The interaction part has the form
\begin{equation*}%\label{form A infinity2}
\begin{aligned}
A_{m,n,3}^{l,1}(y)=&\sum\limits_{|k|\leq m+n-1}e^{\frac{i}{4}ky^2}y^{-2s_c|k|-2}A_{m,n,3}^{l,1,k}(y),\\
A_{m,n,3}^{l,j'}(y)=&\sum\limits_{|k|\leq m+n}e^{\frac{i}{4}ky^2}y^{-2s_c|k|-2}A_{m,n,3}^{l,j',k}(y),\,\,\,j'=0,2,
\end{aligned}
\end{equation*}
where $A_{m,n,3}^{l,j,k}$, $j=0,1,2$ have the following asymptotic expansion as $y\rightarrow+\infty$
\begin{equation}\label{expansion tilde 3 1}
\begin{aligned}
A_{m,n,3}^{l,j',0}(y)=&O\Big(y^{-2\mu_{m,n}^{l,j'}}\Big),\\
A_{m,n,3}^{l,j',1}(y)=&O\Big(\sum_{\substack{1\leq s\leq m+n\\s-1\in2\mathbb N_0}}
\sum_{\theta=0}^{2l+j'}
\sum_{\substack{|i|\leq m,\,|j|\leq n\\i\equiv m\,(\text{mod 2}),\ j\equiv n\,(\text{mod 2})\\i+j\leq m+n-2s}}
y^{-2s_c(s-1)-2\mu_{i,j}^{l-\theta,j'}}(1+\ln y)\Big),\\
A_{m,n,3}^{l,j',-1}(y)=&O\Big(\sum_{\substack{1\leq s\leq m+n\\s-1\in2\mathbb N_0}}
\sum_{\theta=0}^{2l+j'}
\sum_{\substack{|i|\leq m,\,|j|\leq n\\i\equiv m\,(\text{mod 2}),\ j\equiv n\,(\text{mod 2})\\i+j\leq m+n-2s}}y^{-2s_c(s-1)-2\mu_{i,j}^{l-\theta,j'}-2}\Big),\\
\end{aligned}
\end{equation}
where the sums over $s$ are restricted to $s\leq m+n-1$ when $j'=1$ and the factor $1+\ln y$ in the $k=1$ estimate can be omitted if $(s-1)s_c+\mu_{m+i,n+j}^{2l-\theta,2j'}\neq-1$. 

For $k\neq -1,0,1$ and $j'=0,\,2$,
\begin{equation}\label{expansion tilde 3 2}
\begin{aligned}
A_{m,n,3}^{l,1,k}(y)=&O\Big(y^{2\mu_{m,n}^{l,1}-2}\Big),\qquad k=\pm(m+n-1),\\
A_{m,n,3}^{l,j',k}(y)=&O\Big(y^{2\mu_{m,n}^{l,j'}-2}\Big),\quad k=\pm(m+n),\\
A_{m,n,3}^{l,1,k}(y)=&O\Big(\sum_{\substack{|k|\leq s\leq m+n-1\\s-k\in2\mathbb N_0\\0\leq \theta\leq 2l+1}}
\sum_{\substack{|i|\leq m,\,|j|\leq n\\i\equiv m\,(\text{mod 2}),\ j\equiv n\,(\text{mod 2})\\i+j\leq m+n-2s}}y^{-2s_c(s-|k|)-2\mu_{i,j}^{l-\theta,1}-2}\Big),\quad |k|\leq m+n-2,\\
A_{m,n,3}^{l,j',k}(y)=&O\Big(\sum_{\substack{|k|\leq s\leq m+n\\s-k\in2\mathbb N_0\\0\leq \theta\leq 2l+j'}}
\sum_{\substack{|i|\leq m,\,|j|\leq n\\i\equiv m\,(\text{mod 2}),\ j\equiv n\,(\text{mod 2})\\i+j\leq m+n-2s}}
y^{-2s_c(s-|k|)-2\mu_{i,j}^{l-\theta,j'}-2}
\Big),\quad |k|\leq m+n-1.
\end{aligned}
\end{equation}
The asymptotic expansions \eqref{expansion tilde 12}, \eqref{expansion tilde 3 1} and \eqref{expansion tilde 3 2} can be differentiated any number of times with respect to $y$.}
\end{lemma}

\medbreak

\begin{proof} For $(m,n,l,j)\in\tilde{\mathcal{S}}$, decomposing $A_{m,n}^{l,j}$ in the basis $f_{1,m,n}^{l,j,\pm}$, $f_{2,m,n}^{l,j,\pm}$ one gets
\begin{equation*}
A_{m,n}^{l,j}(y) = A_{m,n,1}^{l,j}(y)+A_{m,n,2}^{l,j}(y)
\end{equation*} 
where 
\begin{equation*}
A_{m,n,h}^{l,j}(y)=\Re \lambda_{m,n}^{l,j,h}f_{h,m,n}^{l,j,+}(y)+\Im \lambda_{m,n}^{l,j,h} f_{h,m,n}^{l,j,-}(y),
\end{equation*}
with coefficients $\lambda_{m,n}^{l,j,h}$ for $h=1,2$.

Consequently, as $y\rightarrow +\infty$, one has
\begin{equation*}\label{ss infinity fund}
\begin{aligned}
A_{m,n}^{l,j}(y) =& y^{-2\mu_{m,n}^{l,j}}\big(\lambda_{m,n}^{l,j,1}+O(y^{-2})\big) +e^{\frac{iy^2}{4}}y^{-2s_c+2\mu_{m,n}^{l,j}}\big(\lambda_{m,n}^{l,j,2} +O(y^{-2})\big).
\end{aligned}
\end{equation*} 
This asymptotic can be differentiated any number of times with respect to $y$. 

\medskip We next determine $A_{m,n}^{l,j}$ inductively for all $(m,n,l,j)\in\mathcal{S}\setminus\tilde{\mathcal S}$. Assume that we have found the solutions $\{A_{\tilde{m},\tilde{n}}^{l,j'}\}_{(\tilde{m},\tilde{n},l,j')\in\mathcal{S}}$ that satisfy the expansion \eqref{expansion A infinity} for $0\leq \tilde{m}\leq m$ and $0\leq \tilde{n}\leq n$. It follows from the asymptotics \eqref{def tilde W}, \eqref{expansion A infinity} and \eqref{ss basis infinity} that $g_{m,n}^{l,j'}$ in \eqref{system ss} has the following behaviors as $y\rightarrow +\infty$,
\begin{equation}\label{g1 infinty}
\begin{aligned}
g_{m,n}^{l,1}(y)=&\sum_{s=0}^{m+n-1}\sum_{\substack{|k|\leq s\\s-k\in2\mathbb N_0}}e^{\frac{i}{4}ky^2}y^{-2s s_c-2}\\
&\times\sum_{\substack{q\geq0\\0\leq \theta\leq 2l+1}}\sum_{\substack{|i|\leq m,\,|j|\leq n\\i\equiv m\,(\text{mod 2}),\ j\equiv n\,(\text{mod 2})\\i+j\leq m+n-2s}}b_{m,n,i,j}^{l,1,k,s,\theta,q}y^{-2\mu_{i,j}^{l-\theta,1}-2q},\qquad y\to+\infty,
\end{aligned}
\end{equation}
and for $j'\in\{0,2\}$,
\begin{equation}\label{g02 infinity}
\begin{aligned}
g_{m,n}^{l,j'}(y)=&\sum_{s=0}^{m+n}\sum_{\substack{|k|\leq s\\s-k\in2\mathbb N_0}}e^{\frac{i}{4}ky^2}y^{-2s s_c-2}\\
&\times\sum_{\substack{q\geq0\\0\leq \theta\leq 2l+j'}}\sum_{\substack{|i|\leq m,\,|j|\leq n\\i\equiv m\,(\text{mod 2}),\ j\equiv n\,(\text{mod 2})\\i+j\leq m+n-2s}}b_{m,n,i,j}^{l,j',k,s,\theta,q}y^{-2\mu_{i,j}^{l-\theta}-2q,j'},\qquad y\to+\infty,
\end{aligned}
\end{equation}
with coefficients $b_{m,n,i,j}^{l,j',k,s,\theta}$.  Let us write $A_{m,n}^{l,j'}$ as
\begin{equation}\label{expan A mn 0}
\begin{aligned}
A_{m,n}^{l,j'}(y)
=& \sum_{\substack{1\leq s\leq m+n\\s-1\in2\mathbb N_0\\0\leq \theta\leq 2l+j'}}\sum_{\substack{|i|\leq m,\,|j|\leq n\\i\equiv m\,(\text{mod 2}),\ j\equiv n\,(\text{mod 2})\\i+j\leq m+n-2s}}y^{-2\big((s-1)s_c+\mu_{m+i,n+j}^{2l-\theta,2j'}+1\big)}\\
&\times\Big(a_{m,n,i,j}^{l,j',s,\theta,0,+}f_{2,m,n}^{l,j',+}(y)+a_{m,n,i,j}^{l,j',s,\theta,0,-}f_{2,m,n}^{l,j',-}(y)\Big)
+A_{m,n}^{l,j',0}(y).
\end{aligned}
\end{equation}
For $j'=1$, all sums over $s$ below are restricted to $s\leq m+n-1$. Assume that
\begin{equation*}
(s-1)s_c+\mu_{m+i,n+j}^{2l-\theta,2j'}+1\neq0,
\end{equation*}
and choose
\begin{equation*}
\begin{aligned}
a_{m,n,i,j}^{l,j',s,\theta,0}=a_{m,n,i,j}^{l,j',s,\theta,0,+}+i a_{m,n,i,j}^{l,j',s,\theta,0,-}=-\frac{i b_{m,n,i,j}^{l,j',1,s,\theta,0}}{
(s-1)s_c+\mu_{m+i,n+j}^{2l-\theta,2j'}+1}.
\end{aligned}
\end{equation*}
Then $A_{m,n}^{l,j',0}$ solves
\begin{equation*}
\mathcal{L}(A_{m,n}^{l,j',0},\overline{A_{m,n}^{l,j',0}})+i\mu_{m,n}^{l,j'}A_{m,n}^{l,j',0}=g_{m,n}^{l,j',0},\quad j'=0,1,2,
\end{equation*}
with
\begin{equation*}
\begin{aligned}
g_{m,n}^{l,j',0}(y)=&g_{m,n}^{l,j'}(y)-2\sum_{\substack{1\leq s\leq m+n\\s-1\in2\mathbb N_0\\0\leq \theta\leq 2l+j'}}\sum_{\substack{|i|\leq m,\,|j|\leq n\\i\equiv m\,(\text{mod 2}),\ j\equiv n\,(\text{mod 2})\\i+j\leq m+n-2s}}
\big((s-1)s_c+\mu_{m+i,n+j}^{2l-\theta,2j'}+1\big)\\
&\times y^{-2\big((s-1)s_c+\mu_{m+i,n+j}^{2l-\theta,2j'}+1\big)}
\Big(2y^{-1}\partial_y-\big(4+2(s-2)s_c+2\mu_{m+i,n+j}^{2l-\theta,2j'}\big)y^{-2}-\frac i2\Big)\\
&\qquad\qquad\qquad\qquad\qquad\quad\times
\Big(a_{m,n,i,j}^{l,j',s,\theta,0,+}f_{2,m,n}^{l,j',+}(y)+a_{m,n,i,j}^{l,j',s,\theta,0,-}f_{2,m,n}^{l,j',-}(y)\Big).
\end{aligned}
\end{equation*}

\medskip Using $|\lambda_-|>|\lambda_+|$ and $\nu>0$, we define
\begin{equation*}%\label{def constant L}
L:=\left\lfloor\nu\bigl(\alpha_{2m-2,2n}+4j'+8l\bigr)\right\rfloor.
\end{equation*}
Indeed, on the above index range,
\begin{equation*}
-\mu_{m+i,n+j}^{2l-\theta,2j'}=\nu\bigl(\alpha_{m+i,n+j}+4j'+8l-4\theta\bigr)\leq\nu\bigl(\alpha_{2m-2,2n}+4j'+8l\bigr).
\end{equation*}
We repeat the above process to eliminate the terms
\begin{equation*}
\begin{aligned}
e^{\frac i4y^2}\sum_{\substack{1\leq s\leq m+n\\s-1\in2\mathbb N_0\\0\leq \theta\leq 2l+j'}}y^{-2ss_c-2}\sum_{q=0}^{L}\sum_{\substack{|i|\leq m,\,|j|\leq n\\i\equiv m\,(\text{mod 2}),\ j\equiv n\,(\text{mod 2})\\i+j\leq m+n-2s}}b_{m,n,i,j}^{l,j',1,s,\theta,q}y^{-2\mu_{i,j}^{l-\theta,j'}-2q}
\end{aligned}
\end{equation*}
in \eqref{g1 infinty} and \eqref{g02 infinity}. We use the identities
\begin{equation*}
\begin{aligned}
\mathcal L\left(y^{a}f_{2,m,n}^{l,j',\pm}, y^{a}\overline{f_{2,m,n}^{l,j',\pm}}\right)+i\mu_{m,n}^{l,j'}y^{a}f_{2,m,n}^{l,j',\pm}=a y^{a}\Big(-\frac2y\partial_y-\big(a+2s_c-2\big)y^{-2}+\frac i2\Big) f_{2,m,n}^{l,j',\pm},
\end{aligned}
\end{equation*}
and
\begin{equation*}
\begin{aligned}
\mathcal L\left( \ln y f_{2,m,n}^{l,j',\pm},  \ln y \overline{f_{2,m,n}^{l,j',\pm}}\right)+i\mu_{m,n}^{l,j'}\ln y f_{2,m,n}^{l,j',\pm}=-\Big(2y^{-1}\partial_y+(2s_c-2)y^{-2}-\frac{i}{2}\Big)f_{2,m,n}^{l,j',\pm}.
\end{aligned}
\end{equation*}

We rewrite \eqref{expan A mn 0} as
\begin{equation}\label{expan A mn}
\begin{aligned}
A_{m,n}^{l,j'}(y)=&\sum_{\substack{1\leq s\leq m+n\\s-1\in2\mathbb N_0\\0\leq \theta\leq 2l+j'}}
\sum_{\substack{|i|\leq m,\,|j|\leq n\\i\equiv m\,(\text{mod 2}),\ j\equiv n\,(\text{mod 2})\\i+j\leq m+n-2s}}\\
&\times\Big(\sum_{\substack{0\leq q\leq L\\(s-1)s_c+\mu_{m+i,n+j}^{2l-\theta,2j'}+q+1\neq0}}
y^{-2((s-1)s_c+\mu_{m+i,n+j}^{2l-\theta,2j'}+q+1)}\\
&\qquad\qquad\qquad\qquad\qquad\times\Big(a_{m,n,i,j}^{l,j',s,\theta,q,+}f_{2,m,n}^{l,j',+}(y)+a_{m,n,i,j}^{l,j',s,\theta,q,-}f_{2,m,n}^{l,j',-}(y)\Big)\\
&\qquad+\sum_{\substack{0\leq q\leq L\\(s-1)s_c+\mu_{m+i,n+j}^{2l-\theta,2j'}+q+1=0}}\ln y\Big(a_{m,n,i,j}^{l,j',s,\theta,q,+}f_{2,m,n}^{l,j',+}(y)+a_{m,n,i,j}^{l,j',s,\theta,q,-}f_{2,m,n}^{l,j',-}(y)\Big)\Big)\\
&+A_{m,n}^{l,j',L}(y).
\end{aligned}
\end{equation}
The coefficients $a_{m,n,i,j}^{l,j',s,\theta,q,\pm}$ are determined recursively from the source coefficients and the expansions \eqref{ss basis infinity} of the fundamental solutions. Then $A_{m,n}^{l,j',L}$ satisfies
\begin{equation}\label{eq ss A infinity}
\mathcal{L}(A_{m,n}^{l,j',L},\overline{A_{m,n}^{l,j',L}})+i\mu_{m,n}^{l,j'}A_{m,n}^{l,j',L}=g_{m,n}^{l,j',L},\quad j'=0,1,2,
\end{equation}
with 
\begin{equation}\label{eq ss g infinity}
\begin{aligned}
g_{m,n}^{l,1,L}(y)=&\sum_{|k|\leq m+n-1}e^{\frac{i}{4}ky^2}y^{-2s_c|k|-2}g_{m,n,k}^{l,1,L}(y),\\
g_{m,n}^{l,j',L}(y)=&\sum_{|k|\leq m+n}e^{\frac{i}{4}ky^2}y^{-2s_c|k|-2}g_{m,n,k}^{l,j',L}(y),\quad j'=0,2.
\end{aligned}
\end{equation}
Here the contributions with different $s$ and $\theta$ are included in $g_{m,n,k}^{l,j',L}$.

\medskip For every integer $\ell\geq0$, $g_{m,n,\pm 1}^{l,j',L}$ has the following behaviors as $y\rightarrow +\infty$,
\begin{equation*}
\begin{aligned}
\partial_y^{\ell} g_{m,n,1}^{l,j',L}(y)=&O\Big(\sum_{\substack{1\leq s\leq m+n\\s-1\in2\mathbb N_0\\0\leq \theta\leq 2l+j'}}
\sum_{\substack{|i|\leq m,\,|j|\leq n\\i\equiv m\,(\text{mod 2}),\ j\equiv n\,(\text{mod 2})\\i+j\leq m+n-2s}}
y^{-2s_c(s-1)-2\mu_{i,j}^{l-\theta,j'}-2(L+1)-\ell}\Big),\\
\partial_y^{\ell} g_{m,n,-1}^{l,j',L}(y)=&O\Big(\sum_{\substack{1\leq s\leq m+n\\s-1\in2\mathbb N_0\\0\leq \theta\leq 2l+j'}}
\sum_{\substack{|i|\leq m,\,|j|\leq n\\i\equiv m\,(\text{mod 2}),\ j\equiv n\,(\text{mod 2})\\i+j\leq m+n-2s}}y^{-2s_c(s-1)-2\mu_{i,j}^{l-\theta,j'}-\ell}\Big),
\end{aligned}
\end{equation*}
for $j'=0,1,2$. For $j'=1$, all sums over $r$ below are restricted to $s\leq m+n-1$.

Similarly, the $g_{m,n,k}^{l,j',L}$ for $k\neq \pm1$ satisfy
\begin{equation*}
\begin{aligned}
\partial_y^{\ell} g_{m,n,k}^{l,1,L}(y)&= O\Big(y^{2\mu_{m,n}^{l,1}-\ell}\Big),k=\pm(m+n-1),\\
\partial_y^{\ell} g_{m,n,k}^{l,1,L}(y)&=O\Big(\sum_{\substack{|k|\leq s\leq m+n-1\\s-k\in2\mathbb N_0\\0\leq \theta\leq 2l+1}}
\sum_{\substack{|i|\leq m,\,|j|\leq n\\i\equiv m\,(\text{mod 2}),\ j\equiv n\,(\text{mod 2})\\i+j\leq m+n-2s}}y^{-2s_c(s-|k|)-2\mu_{i,j}^{l-\theta,1}-\ell}\Big),&&|k|\leq m+n-2,
\end{aligned}
\end{equation*}
and for $j'=0,\,2$,
\begin{equation*}
\begin{aligned}
\partial_y^{\ell} g_{m,n,k}^{l,j',L}(y)&=O\Big(y^{2\mu_{m,n}^{l,j'}-\ell}\Big),&& k=\pm(m+n),\\
\partial_y^{\ell} g_{m,n,k}^{l,j',L}(y)&=O\Big(\sum_{\substack{|k|\leq s\leq m+n\\s-k\in2\mathbb N_0\\0\leq \theta\leq 2l+j'}}
\sum_{\substack{|i|\leq m,\,|j|\leq n\\i\equiv m\,(\text{mod 2}),\ j\equiv n\,(\text{mod 2})\\i+j\leq m+n-2s}}y^{-2s_c(s-|k|)-2\mu_{i,j}^{l-\theta,j'}-\ell}\Big),&& |k|\leq m+n-1.
\end{aligned}
\end{equation*}

\medskip Integrating \eqref{eq ss A infinity} with \eqref{eq ss g infinity}, one gets
\begin{equation}\label{def ss A1L}
\begin{aligned}
A_{m,n}^{l,j',L}(y)=
\sum_{h=1}^{2}\Big(\Re\lambda_{m,n}^{l,j',h}f_{h,m,n}^{l,j',+}(y)+\Im\lambda_{m,n}^{l,j',h}f_{h,m,n}^{l,j',-}(y)\Big)
+\tilde A_{m,n}^{l,j',L}(y),
\end{aligned}
\end{equation}
where $\lambda_{m,n}^{l,j',h}\in\mathbb C$, $h=1,2$ are constants, and
\begin{equation}\label{expan tilde A ss}
\begin{aligned}
\tilde{A}_{m,n}^{l,1,L}(y)=&\sum\limits_{|k|\leq m+n-1}e^{\frac{i}{4}ky^{2}}y^{-2s_c |k|-2}\tilde{A}_{m,n,k}^{l,1,L}(y),\\
\tilde{A}_{m,n}^{l,j',L}(y)=&\sum\limits_{|k|\leq m+n}e^{\frac{i}{4}ky^{2}}y^{-2s_c |k|-2}\tilde{A}_{m,n,k}^{l,j',L}(y),\quad j'=0,2.
\end{aligned}
\end{equation}

For every integer $\ell\geq0$, the nonoscillatory and negative phase for $j'=0,1,2$ satisfy
\begin{equation*}
\begin{aligned}
\partial_y^{\ell}\tilde A_{m,n,0}^{l,j',L}(y)&=O\Big(y^{-2\mu_{m,n}^{l,j'}-\ell}\Big),\\
\partial_y^{\ell}\tilde A_{m,n,-1}^{l,j',L}(y)&=O\Big(\sum_{\substack{1\leq s\leq m+n\\s-1\in2\mathbb N_0\\0\leq \theta\leq 2l+j'}}
\sum_{\substack{|i|\leq m,\,|j|\leq n\\i\equiv m\,(\text{mod 2}),\ j\equiv n\,(\text{mod 2})\\i+j\leq m+n-2s}}y^{-2s_c(s-1)-2\mu_{i,j}^{l-\theta,j'}-2-\ell}\Big).
\end{aligned}
\end{equation*}

The negative phase response also contributes to the positive phase through $-pc_\infty^{2p}y^{-2}\overline{\tilde A_{m,n}^{l,j',L}}$.
The corresponding response has power
\begin{equation*}
y^{-2s s_c-6-2\mu_{i,j}^{l-\theta,j'}},
\end{equation*}
which coincides with the positive phase homogeneous power $y^{-2s_c+2\mu_{m,n}^{l,j'}}$ precisely when
\begin{equation*}
(s-1)s_c+\mu_{m+i,n+j}^{2l-\theta,2j'}=-3.
\end{equation*}
Hence, if this equality holds for some indices in the above sums, then
\begin{equation*}
\begin{aligned}
\partial_y^{\ell}\tilde A_{m,n,1}^{l,j',L}(y)=O\Big(
&\sum_{\substack{1\leq s\leq m+n\\s-1\in2\mathbb N_0\\0\leq \theta\leq 2l+j'}}
\sum_{\substack{|i|\leq m,\,|j|\leq n\\i\equiv m\,(\text{mod 2}),\ j\equiv n\,(\text{mod 2})\\i+j\leq m+n-2s}}y^{-2s_c(s-1)-2\mu_{i,j}^{l-\theta,j'}-\ell}\\
&\qquad\times\bigl(y^{-2(L+1)}+y^{-4}(1+\ln y)\bigr)\Big),
\qquad j'=0,1,2.
\end{aligned}
\end{equation*}
Otherwise, if
\begin{equation*}
(s-1)s_c+\mu_{m+i,n+j}^{2l-\theta,2j'}\neq-3
\end{equation*}
for every set of indices in the sums, then
\begin{equation*}
\begin{aligned}
\partial_y^{\ell}\tilde A_{m,n,1}^{l,j',L}(y)=O\Big(
&\sum_{\substack{1\leq s\leq m+n\\s-1\in2\mathbb N_0\\0\leq \theta\leq 2l+j'}}
\sum_{\substack{|i|\leq m,\,|j|\leq n\\i\equiv m\,(\text{mod 2}),\ j\equiv n\,(\text{mod 2})\\i+j\leq m+n-2s}}y^{-2s_c(s-1)-2\mu_{i,j}^{l-\theta,j'}-\ell}\\
&\qquad\times
\bigl(y^{-2(L+1)}+y^{-4}\bigr)\Big),\qquad j'=0,1,2.
\end{aligned}
\end{equation*}
For the remaining phases $k\notin\{-1,0,1\}$,
\begin{equation*}
\begin{aligned}
\partial_y^{\ell}\tilde A_{m,n,k}^{l,1,L}(y)&=O\Big(y^{2\mu_{m,n}^{l,1}-2-\ell}\Big),\quad\qquad k=\pm(m+n-1),\\
\partial_y^{\ell}\tilde A_{m,n,k}^{l,1,L}(y)&=O\Big(\sum_{\substack{|k|\leq s\leq m+n-1\\s-k\in2\mathbb N_0\\0\leq \theta\leq 2l+1}}
\sum_{\substack{|i|\leq m,\,|j|\leq n\\i\equiv m\,(\text{mod 2}),\ j\equiv n\,(\text{mod 2})\\i+j\leq m+n-2s}}
y^{-2s_c(s-|k|)-2\mu_{i,j}^{l-\theta,1}-2-\ell}\Big),\quad |k|\leq m+n-2,
\end{aligned}
\end{equation*}
and for $j'=0,\,2$,
\begin{equation*}
\begin{aligned}
\partial_y^{\ell}\tilde A_{m,n,k}^{l,j',L}(y)&=O\bigl(y^{2\mu_{m,n}^{l,j'}-2-\ell}\bigr),\quad\qquad k=\pm(m+n),\\
\partial_y^{\ell}\tilde A_{m,n,k}^{l,j',L}(y)&=O\Big(\sum_{\substack{|k|\leq s\leq m+n\\s-k\in2\mathbb N_0\\0\leq \theta\leq 2l+j'}}
\sum_{\substack{|i|\leq m,\,|j|\leq n\\i\equiv m\,(\text{mod 2}),\ j\equiv n\,(\text{mod 2})\\i+j\leq m+n-2s}}
y^{-2s_c(s-|k|)-2\mu_{i,j}^{l-\theta,j'}-2-\ell}\Big),\quad |k|\leq m+n-1.
\end{aligned}
\end{equation*}
Combining \eqref{expan A mn}, \eqref{def ss A1L},
and \eqref{expan tilde A ss} gives the asymptotic
decomposition, including both oscillatory branches
and the resonant logarithmic terms.

\end{proof}
\bigbreak

\subsection{Estimate of the approximate solution in the  self-similar region}\label{subsec 3.3}
Throughout this subsection, assume that $0<\epsilon_2<\min\big\{\nu,\frac12\big\}$. Under the above notations, for an integer $N_{ss} \geq 2$ satisfy
\begin{equation}\label{cond Nss}
N_{ss}\geq\Big\lfloor\frac{N_{in}-1}{2}\Big\rfloor.
\end{equation}
We set the approximate solution in the self-similar region as
\begin{equation}\label{def psi ss}
 \psi_{ss}^{ap}(t,x)=\lambda^{\frac{1}{p}}(t)u_{ss}^{ap}(t,\rho),\,\,\,\text{where}\,\,\,\,u_{ss}^{ap}(t,\rho)=Q(\rho)+\chi_{ss}^{ap}(t,\rho),\end{equation}
and
\begin{equation}\label{chi ss}
\chi_{ss}^{ap}(t,\rho)=\rho^{-\frac{1}{p}}\big(c_{\infty}+v_{ss}^{ap}(t,t^{\nu}\rho)\big)-Q(\rho),
\end{equation}
has the form
\begin{equation*}
\begin{aligned}
v_{ss}^{ap}(t,t^{\nu}\rho)=&\sum\limits_{\substack{m+n\geq 1\\0\leq m,n\leq N_{ss}}}t^{\nu(|\lambda_{+}|m+|\lambda_{-}|n)}A_{m,n}^{0,0}(t^{\nu}\rho)+\sum\limits_{\substack{1\leq m,n\leq N_{ss}\\ 0\leq l\leq N_{ss}}}t^{\nu(|\lambda_{+}|m+|\lambda_{-}|n+2+4l)}A_{m,n}^{l,1}(t^{\nu}\rho)\\
&+\sum\limits_{\substack{0\leq m,l\leq N_{ss}\\1\leq n\leq N_{ss}}}t^{\nu(|\lambda_{+}|m+|\lambda_{-}|n+4+4l)}A_{m,n}^{l,2}(t^{\nu}\rho).
\end{aligned}
\end{equation*}

\medskip The purpose of this paragraph is firstly to estimate the radial function $\chi_{ss}^{ap}$ defined by \eqref{chi ss} in the self-similar region $\mathcal{S}_{ss}$  and secondly to study the remainder term for $N_{ss}$ sufficiently large.

\begin{lemma}\label{lem L infinity ss}
{\sl There exists a constant $C=C(k,l,N_{ss})>0$ and a time $T=T(N_{ss})>0$ such that, for all $0<t\leq T$,
\begin{equation}\label{L infty est ss 1}
\begin{aligned}
  \|\rho^{-k}\partial_{\rho}^{l}\chi_{ss}^{ap}\|_{L_{\rho}^{\infty}(\mathcal{S}_{ss})}\leq& Ct^{(\nu-\epsilon_1)(\frac{1}{p}+|\lambda_{+}|+k+l)}+Ct^{(\nu+\epsilon_2)(\frac{1}{p}+k+l)+\nu|\lambda_{+}|(1-2\epsilon_2)},\\
  &\qquad\qquad\qquad\qquad\qquad\qquad\qquad 0\leq l<s_c+2\nu|\lambda_+|,
  \end{aligned}
  \end{equation}
  and
  \begin{equation}\label{L infty est ss 2}
\begin{aligned}
  \|\rho^{-k}\partial_{\rho}^{l}\chi_{ss}^{ap}\|_{L_{\rho}^{\infty}(\mathcal{S}_{ss})}\leq& Ct^{(\nu-\epsilon_1)(\frac{1}{p}+|\lambda_{+}|+k+l)}\\&+Ct^{(\nu+\epsilon_2)(\frac{1}{p}+k+l)+\nu|\lambda_{+}|(1+2\epsilon_2)-2\epsilon_2(l-s_c)},\quad l>s_c+2\nu|\lambda_+|.
  \end{aligned}
  \end{equation}
Moreover, the following $L^2$ estimates hold:
  \begin{equation*}\label{L 2 est ss 1}
   \big\|\rho^{-k}\partial_{\rho}^{l}\chi_{ss}^{ap}\big\|_{L_{\rho}^{2}(\mathcal{S}_{ss})}\leq Ct^{(\nu+\epsilon_2)(k+l-s_c)+(1-2\epsilon_2)\nu|\lambda_{+}|},\,\,0\leq k+l< s_c+2-|\lambda_{+}|,
   \end{equation*}
  
     \begin{equation*}\label{L 2 est ss 2}
   \begin{aligned}
   \big\|\rho^{-k}\partial_{\rho}^{l}\chi_{ss}^{ap}\big\|_{L_{\rho}^{2}(\mathcal{S}_{ss})}\leq& Ct^{2\nu}+Ct^{(\nu+\epsilon_2)(k+l-s_c)
+(1-2\epsilon_2)\nu|\lambda_+|},\\
 &\quad\qquad\qquad k+l>s_c+2-|\lambda_+|, \quad l-k\leq s_c+2\nu|\lambda_+|,  
  \end{aligned}
  \end{equation*}
   \begin{equation*}\label{L 2 est ss 3}
   \begin{aligned}
   \big\|\rho^{-k}\partial_{\rho}^{l}\chi_{ss}^{ap}\big\|_{L_{\rho}^{2}(\mathcal{S}_{ss})}\leq& Ct^{2\nu}\Big(1+t^{\epsilon_2(k-l+s_c+2\nu|\lambda_+|)}\Big),\quad l-k>s_c+2\nu|\lambda_+|.  
  \end{aligned}
  \end{equation*}
      }
\end{lemma}

\medskip

Let us now define the remainder associated with the approximate profile
\begin{align*}\label{R ss}
\mathcal{R}_{ss}^{ap}(t,\rho):=%&[\eqref{eq of U in}U_{ss}^{N_{ss}}](t,\rho).
-it^{1+2\nu}\partial_{t}u_{ss}^{ap}+(\frac{1}{2}+\nu)it^{2\nu}\Lambda u_{ss}^{ap}-\Delta u_{ss}^{ap}-|u_{ss}^{ap}|^{2p}u_{ss}^{ap}.
\end{align*}
Using \eqref{def psi ss} and \eqref{chi ss}, it can be represented as 
\begin{align*}%\label{Rss and tilde Rss}
\mathcal{R}_{ss}^{ap}(t,\rho)=t^{2\nu}\rho^{-\frac{1}{p}}\tilde{\mathcal{R}}_{ss}^{ap}(t,t^{\nu}\rho),
\end{align*}
where with $y=t^\nu\rho$,
\begin{align*}%\label{tilde Rss}
\tilde{\mathcal{R}}_{ss}^{ap}(t,y)=%[\eqref{eq of v ss} v_{ss}^{N_{ss}}](t,y).
-it\partial_{t}v_{ss}^{ap}-\Delta v_{ss}^{ap}+\frac{2}{p}y^{-1}\partial_{y}v_{ss}^{ap}+\frac{i}{2}y\partial_{y}v_{ss}^{ap}-pc_{\infty}^{2p}y^{-2}\big(v_{ss}^{ap}+\overline{v_{ss}^{ap}}\big)-\tilde{W}(v_{ss}^{ap}).
\end{align*}
Substituting the asymptotic expansion of $v_{ss}^{ap}$ given in Lemma \ref{lem ss 0} and \ref{lem ss infinity} into system \eqref{system ss}, the terms with $m,n,l \leq N_{ss}$ cancel out. Consequently, the actual error induced by the
approximate solution $v^{ap}$ is given by
\begin{equation*}
\begin{aligned}
\tilde{\mathcal R}_{ss}^{ap}(t,y)
={}&\sum_{\substack{(m,n,l,j)\in\mathcal S\\m,n\leq N_{ss},\ l\leq N_{ss}}}t^{\nu(|\lambda_+|m+|\lambda_-|n+2j+4l)}g_{m,n}^{l,j}(y)-\tilde W(v_{ss}^{ap})(t,y).
\end{aligned}
\end{equation*}
The coefficients cancel by \eqref{system ss}. The remaining terms are the negatives of the nonlinear coefficients generated by the truncated profile \eqref{def psi ss} whose indices satisfy
$$ m>N_{ss}\quad\text{or}\quad n>N_{ss}\quad\text{or}\quad l>N_{ss}.$$
The last alternative applies only to $j=1,2$. As a direct consequence of Lemma \ref{lem ss 0} and \ref{lem ss infinity}, we obtain the following lemma:
\begin{lemma}\label{lem: error chi ss}
{\sl For  every integer $l\geq0$, there exist $C=C(l,N_{ss})>0$ and $T=T(N_{ss})>0$ such that, for $0<t\leq T$,
the following estimate holds
\begin{equation}\label{ss error est for any Nss}
\begin{aligned}
\|\nabla_{\rho}^{l}\mathcal{R}_{ss}^{ap}\|_{L^2}\leq&Ct^{(\nu-\epsilon_1)(|\lambda_{+}|(N_{ss}+1)+2-s_c)}\\
&+Ct^{(1-2\epsilon_2)\nu|\lambda_{+}|(N_{ss}+1)+(\nu+\epsilon_2)(2-s_c)}.
\end{aligned}
\end{equation}
}
\end{lemma}

\bigbreak
Let us end this section by investigating $\psi_{ss}^{ap}-\psi_{in}^{ap}$ in the intersection of the inner and self-similar regions,  namely
$$\mathcal{S}_{in} \cap \mathcal{S}_{ss}=\Big\{  \frac {1} {10} \,t^{\frac{1}{2}+\epsilon_1}  \leq |x|  \leq   10\, t^{\frac{1}{2}+\epsilon_1}\Big\}. $$ 
Substituting \eqref{inner expansion infinity k1} and \eqref{inner expansion infinity k2} into \eqref{def psi in}, and changing the variable $\rho$ to $y=t^{\nu}\rho$,  as in \eqref{inner expansion in the ss region}, we obtain, for integers $2\leq N_{in}\leq K$,
the following asymptotic expansion on the overlap:
\begin{equation*}%\label{expansion inner in interact}
\begin{aligned}
&\psi_{in}^{ap}(t,t^{1/2}y)\\
=&(t^{1/2}y)^{-\frac{1}{p}}\Big(c_{\infty}+\sum\limits_{m+n\geq 1}t^{-\nu(\lambda_{+}m+\lambda_{-}n)}\sum\limits_{0\leq k\leq N_{in}}i^k c_{m,n}^{k,0,0}y^{\lambda_{+}m+\lambda_{-}n+2k}\\
&\qquad\qquad+\sum\limits_{\substack{m,n\geq 1\\0\leq l\leq\lfloor(N_{in}-1)/2\rfloor}}t^{-\nu(\lambda_{+}m+\lambda_{-}n-2-4l)}\sum\limits_{0\leq k\leq N_{in}-2l-1}i^{k+1+2l}c_{m,n}^{k+2l+1,l,1}y^{\lambda_{+}m+\lambda_{-}n+2k}\\
&\qquad\qquad+\sum\limits_{\substack{m\geq 0,n\geq 1\\0\leq l\leq\lfloor(N_{in}-2)/2\rfloor}}t^{-\nu(\lambda_{+}m+\lambda_{-}n-4-4l)}\sum\limits_{0\leq k\leq N_{in}-2l-2}i^{k+2+2l}c_{m,n}^{k+2l+2,l,2}y^{\lambda_{+}m+\lambda_{-}n+2k}
\Big).
\end{aligned}
\end{equation*}
By Lemma \ref{lem ss 0} and the matching conditions,
\begin{equation*}
d_{m,n}^{l,j,k,0}=(-1)^l i^{k+j}c_{m,n}^{k+2l+j,l,j},\qquad d_{m,n}^{l,j,k,\ell}=0, \quad(\ell\geq1),\qquad k+2l+j\leq N_{in}.
\end{equation*}
Therefore, in view of \eqref{def psi ss} and \eqref{chi ss}, for integers $2\leq N_{in}\leq K$ and $N_{ss}\geq2$ satisfying \eqref{cond Nss}, we have
\begin{equation*}
\begin{aligned}
&(\psi_{in}^{ap}-\psi_{ss}^{ap})(t,x)\\
=&r^{-\frac{1}{p}}\Big(-\sum_{\substack{m+n\geq1\\0\leq m,n\leq N_{ss}}}\sum_{k\geq N_{in}+1}\sum_{\ell\geq0}d_{m,n}^{0,0,k,\ell}
\big(\ln(t^{-1/2}r)\big)^\ell\mathcal M_{m,n,k}^{0,0}(t,r)\\
&\qquad+\sum_{\substack{m,n\geq0,\ m+n\geq1\\m>N_{ss}\ \mathrm{or}\ n>N_{ss}}}\sum_{0\leq k\leq N_{in}}i^k c_{m,n}^{k,0,0}
\mathcal M_{m,n,k}^{0,0}(t,r)\Big)\\
&+r^{-\frac{1}{p}}\Big(-\sum_{\substack{1\leq m,n\leq N_{ss}\\0\leq l\leq N_{ss}}}\sum_{k\geq\max\{0,N_{in}-2l\}}
\sum_{\ell\geq0}d_{m,n}^{l,1,k,\ell}\big(\ln(t^{-1/2}r)\big)^{\ell}\mathcal M_{m,n,k}^{1,l}(t,r)\\
&\qquad\qquad+\sum_{\substack{m,n\geq1\\0\leq l\leq\lfloor(N_{in}-1)/2\rfloor\\m>N_{ss}\ \mathrm{or}\ n>N_{ss}}}
\sum_{0\leq k\leq N_{in}-2l-1}i^{k+1+2l}c_{m,n}^{k+2l+1,l,1}\mathcal M_{m,n,k}^{1,l}(t,r)\Big)\\
&+r^{-\frac{1}{p}}\Big(-\sum_{\substack{0\leq m\leq N_{ss}\\1\leq n\leq N_{ss}\\0\leq l\leq N_{ss}}}\sum_{k\geq\max\{0,N_{in}-2l-1\}}
\sum_{\ell\geq0}d_{m,n}^{l,2,k,\ell}\big(\ln(t^{-1/2}r)\big)^\ell\mathcal M_{m,n,k}^{2,l}(t,r)\\
&\qquad\qquad+\sum_{\substack{m\geq0,\ n\geq1\\0\leq l\leq\lfloor(N_{in}-2)/2\rfloor\\m>N_{ss}\ \mathrm{or}\ n>N_{ss}}}
\sum_{0\leq k\leq N_{in}-2l-2}i^{k+2+2l}c_{m,n}^{k+2l+2,l,2}\mathcal M_{m,n,k}^{2,l}(t,r)\Big),
\end{aligned}
\end{equation*}
where $r=|x|$ and
\begin{equation*}
\mathcal M_{m,n,k}^{j,l}(t,r)=t^{-(\frac12+\nu)(\lambda_+m+\lambda_-n)-k+\nu(2j+4l)}r^{\lambda_+m+\lambda_-n+2k}.
\end{equation*}
The coefficients come from the expansions \eqref{inner expansion infinity k1}, \eqref{inner expansion infinity k2} and \eqref{expansion ss 0}. On the overlap,
\begin{equation*}
\big|\mathcal M_{m,n,k}^{j,l}(t,r)\big|
\leq
Ct^{(\nu-\epsilon_1)(|\lambda_+|m+|\lambda_-|n)
+2\epsilon_1k+\nu(2j+4l)}.
\end{equation*}
Thus, we can prove the following estimate.

\medbreak
\begin{lemma}\label{lem interaction of psi in ss}
{\sl For every integer $l\geq0$, there exist $C=C(l,N_{in},N_{ss})>0$ and $T>0$ such that, for $0<t\leq T$,
\begin{equation}\label{est psi in ss any N}
\begin{aligned}
\Big|\nabla_{r}^{l}(\psi_{ss}^{ap}(t,r)-\psi_{in}^{ap}(t,r))\Big|\leq r^{-\frac{1}{p}-l}\Big(&Ct^{(\nu-\epsilon_1)|\lambda_{+}|+2(N_{in}+1)\epsilon_1}\\
&\qquad+Ct^{(\nu-\epsilon_1)|\lambda_+|(N_{ss}+1)}\Big),
\end{aligned}
\end{equation}
for $ \frac {1} {10} \,t^{\frac{1}{2}+\epsilon_1}  \leq r \leq   10\, t^{\frac{1}{2}+\epsilon_1}$.}
\end{lemma}

\bigbreak

 \section{Approximate solution in the remote region}\label{sec 4}
   \subsection{General scheme of the construction of the approximate solution in the remote region}\label{subsec 4.1}
 In the previous section, we constructed an approximate solution $\psi_{ss}^{ap}$ in the self-similar region, which extends the inner approximation $\psi_{in}^{ap}$ developed in Section \ref{sec 2}.
We now define the remote region as
\begin{align*}\label{def remote region}
\mathcal{S}_{out}=\Big\{(t,x)\Big|\,|x|\geq \frac{1}{10}t^{1/2-\epsilon_2}\Big\}. 
\end{align*}
Our objective in this section is to construct an approximate solution in $\mathcal{S}_{out}$ with prescribed asymptotic data when $r\sim 1$,
which is determined by matching its behavior with $\psi_{ss}^{ap}$ in the overlap $\mathcal{S}_{ss} \cap \mathcal{S}_{out}$. 

\medskip Under the definition \eqref{def admissible}, for each $q\geq1$, we further define the $q$-th order admissible index set by
$$ \mathcal{S}_{q} := \left\{ (m,n,l,j)\in\mathcal{S}: j\neq0\ \text{or}\ m+n\geq q \right\}. $$
In particular, \(\mathcal{S}_{1}=\mathcal{S}\).

\medskip Note that in the limit $y\to \infty$, $r\to 0$, the main order terms of the expansion $v_{ss}^{ap}$ constructed in Section \ref{sec 3} is
\begin{equation}\label{main terms in out expansion}
\begin{aligned}
&\sum_{(m,n,l,j)\in\mathcal{S}}t^{\nu(|\lambda_{+}|m+|\lambda_{-}|n+2j+4l)}A_{m,n}^{l,j}(t^{-\frac{1}{2}}r)\\
\sim&\sum_{k\geq 0}\frac{t^k}{r^{2k}}\sum_{(m,n,l,j)\in\mathcal{S}}a_{m,n}^{k,0,l,j}r^{2\nu(\alpha_{m,n}+2j+4l)}\\&+e^{\frac{ir^2}{4t}}\sum_{k\geq 0}\frac{t^{k-s_c}}{r^{2k}}\sum_{(m,n,l,j)\in\mathcal{S}}a_{m,n}^{k,1,l,j}\big(\frac{r}{t}\big)^{-2s_c-2\nu(\alpha_{m,n}+2j+4l)},
\end{aligned}
\end{equation}
where $\alpha_{m,n}=|\lambda_{+}|m+|\lambda_{-}|n$. This imply that in the region $\frac{r}{t}\geq t^{-\frac{1}{2}-\epsilon_2}$, by $\lambda_{-}\leq \lambda_{+}\leq -2$ and $\nu>0$, we have to look for the solution of \eqref{eq of v} as a perturbation of the time independent profile:
\begin{align*}
\sum_{(m,n,l,j)\in\mathcal{S}}\beta_{m,n}^{l,j}r^{2\nu(\alpha_{m,n}+2j+4l)},
\end{align*}
with $\beta_{m,n}^{0}=a_{m,n}^{0,0}$, $\beta_{m,n}^{l,1}=a_{m,n}^{0,0,l,1}$ and $\beta_{m,n}^{l,2}=a_{m,n}^{0,0,l,2}$.

\medskip Let $\Theta\in C_{0}^{\infty}(\mathbb{R})$, $\Theta(\xi)=\left\{\begin{array}{l}
1,\,\,\,|\xi|\leq 1\\
0,\,\,\,|\xi|\geq 2
\end{array}\right.$. For $N\geq 2$ and $\delta>0$ we define
\begin{align}\label{eq f0}
f_0(r)=&\Theta(\frac{r}{\delta})r^{-\frac{1}{p}}\Big(c_{\infty}+\sum_{\substack{(m,n,l,j)\in\mathcal{S}\\m,n,l\leq N}}\beta_{m,n}^{l,j}r^{2\nu(\alpha_{m,n}+2j+4l)}\Big),\quad r=|x|.
\end{align}
Write $\psi(t,r)=f_0(r)+\chi(t,r)$. Then $\chi$ solves
\begin{align}\label{eq of chi out}
i\chi_t=-\Delta\chi-(p+1)|f_0|^{2p}\chi-p|f_0|^{2p-2}f_0^2\bar{\chi}-\Delta f_0-|f_0|^{2p}f_0-N(\chi),
\end{align}
where
\begin{align}\label{nonlinear out}
N(\chi)=|f_0+\chi|^{2p}(f_0+\chi)-|f_0|^{2p}f_0-(p+1)|f_0|^{2p}\chi-p|f_0|^{2p-2}f_0^2\bar{\chi}
\end{align}
contains the terms that are at least quadratic in $\chi$. According to \eqref{main terms in out expansion}, \eqref{eq f0}, and \eqref{eq of chi out}, we look for $\chi$ as 
\begin{equation}\label{ansatz chi out}
\begin{aligned}
\chi(t,r)=&\sum_{k\geq 1}t^kf_k(r)+e^{\frac{ir^2}{4t}}\sum_{k\geq 0}t^{s_c+k}A_{1,k}^{+}(t,r)+e^{-\frac{ir^2}{4t}}\sum_{k\geq 0}t^{s_c+k+2}A_{1,k}^{-}(t,r)\\
%+&e^{\frac{ir^2}{4t}}\sum_{\substack{2\leq q\leq 2p+1\\q-1\in2\mathbb{Z}}}\sum_{\substack{0\leq h\leq q\\ \mathcal{A}_{q,1,h}\neq \emptyset}}\sum_{k\geq 0}t^{qs_c+2h+k}A_{q,1,h,k}(t,r)\\
+&\sum_{ q\geq 2}\sum_{\substack{|s|\leq q\\q-s\in 2\mathcal{Z}}}e^{\frac{isr^2}{4t}}\sum_{0\leq h\leq q}\sum_{k\geq 0}t^{qs_c+2h+k}A_{q,s,h,k}(t,r),
\end{aligned}
\end{equation}
where
\begin{align*}
A_{1,k}^{\pm}(t, r)=\sum_{(m,n,l,j)\in\mathcal{S}}t^{2\nu(\alpha_{m,n}+2j+4l)}g_{m,n}^{l,j,k,\pm}(r),
\end{align*}
and 
\begin{align*}
A_{q,s,h,k}(t,r)=\sum_{(m,n,l,j)\in\mathcal{S}_q}t^{2\nu(\alpha_{m,n}+2j+4l)}g_{q,s,h,m,n}^{l,j,k}(r).
\end{align*}
 We use the index set
\begin{equation}\label{def outer Omega}
\Omega:=\left\{(q,s,h)\in\mathbb Z^3: q\geq2,\quad |s|\leq q,\quad q-s\in2\mathbb Z, \quad 0\leq h\leq q\right\},
\end{equation}
then $g_{q,s,h,m,n}^{l,j,k}=0$ if $(q,s,h)\notin \Omega$. 

\medskip Starting from \eqref{eq of chi out},we decompose the nonlinear remainder \eqref{nonlinear out} according to the number of oscillatory factors as
$$ N(\chi)=N_0^{out}(\chi)+N_1^{out}(\chi)+N_2^{out}(\chi), $$
where $N_0^{out}$ involves only the functions $f_k$, $N_1^{out}$ has formal order $q=1$, and $N_2^{out}$ contains the remaining terms of formal order $q\geq2$:
\begin{equation*}%\label{nonlinear out parts}
\begin{aligned}
N_0^{out}=&F(f_0+\sum_{k\geq 1}t^k f_k)-F(f_0)-(p+1)|f_0|^{2p}\sum_{k\geq 1}t^k f_k-p|f_0|^{2p-2}f_0^2\sum_{k\geq 1}t^k\bar{f_k},\\
N_1^{out}(\chi)=&(p+1)\Big(\big|f_0+\sum_{k\geq 1}t^k f_k\big|^{2p}-|f_0|^{2p}\Big)\Big(e^{\frac{ir^2}{4t}}\sum_{k\geq 0}t^{s_c+k}A_{1,k}^{+}+e^{-\frac{ir^2}{4t}}\sum_{k\geq 0}t^{s_c+k+2}A_{1,k}^{-}\Big)\\
&+p\Big(\big|f_0+\sum_{k\geq 1}t^k f_k\big|^{2p-2}\big(f_0+\sum_{k\geq 1}t^k f_k\big)^2-|f_0|^{2p-2}f_0^2\Big)\\
&\qquad\qquad\times\Big(e^{\frac{-ir^2}{4t}}\sum_{k\geq 0}t^{s_c+k}\overline{A_{1,k}^{+}}+e^{\frac{ir^2}{4t}}\sum_{k\geq 0}t^{s_c+k+2}\overline{A_{1,k}^{-}}\Big),
\end{aligned}
\end{equation*}
\begin{equation*}
\begin{aligned}
N_2^{out}(\chi)=&F\Big(f_0+\chi\Big)-F\Big(f_0+\sum_{k\geq 1}t^k f_k\Big)\\
&-(p+1)\Big|f_0+\sum_{k\geq 1}t^k f_k\Big|^{2p}\Big(e^{\frac{ir^2}{4t}}\sum_{k\geq 0}t^{s_c+k}A_{1,k}^{+}+e^{-\frac{ir^2}{4t}}\sum_{k\geq 0}t^{s_c+k+2}A_{1,k}^{-}\Big)\\
&-p\Big|f_0+\sum_{k\geq 1}t^k f_k\Big|^{2p-2}\Big(f_0+\sum_{k\geq 1}t^k f_k\Big)^2 \Big(e^{\frac{-ir^2}{4t}}\sum_{k\geq 0}t^{s_c+k}\overline{A_{1,k}^{+}}+e^{\frac{ir^2}{4t}}\sum_{k\geq 0}t^{s_c+k+2}\overline{A_{1,k}^{-}}\Big)\\
&-(p+1)|f_0|^{2p}\sum_{q\geq2}\sum_{\substack{|s|\leq q\\q-s\in 2\mathbb{Z}}}e^{\frac{isr^2}{4t}}\sum_{\substack{0\leq h\leq q\\ (q,s,h)\in \Omega}}\sum_{k\geq 0}t^{qs_c+2h+k}A_{q,s,h,k}\\
&-p|f_0|^{2p-2}f_0^2\sum_{q\geq2}\sum_{\substack{|s|\leq q\\q-s\in 2\mathbb{Z}}}e^{-\frac{isr^2}{4t}}\sum_{\substack{0\leq h\leq q\\(q,s,h)\in \Omega}}\sum_{k\geq 0}t^{qs_c+2h+k}\overline{A_{q,s,h,k}},
\end{aligned}
\end{equation*}
and $F(u)=|u|^{2p}u$. Therefore, taking \eqref{ansatz chi out} into these expressions, we deduce that
\begin{equation}\label{nonlinear out 1}
\begin{aligned}
N_0^{out}(t,r)=&\sum_{k\geq 2}t^{k}N_{0,k}(r),\\
N_{1}^{out}(t,r)=&\sum_{s=\pm1}e^{\frac{isr^2}{4t}}\sum_{k\geq 1}t^{s_c+k}\sum_{(m,n,l,j)\in\mathcal{S}}t^{2\nu(\alpha_{m,n}+2j+4l)}
N_{s,m,n}^{l,j,k}(r),
\end{aligned}
\end{equation}
and 
\begin{equation}\label{nonlinear out 2}
\begin{aligned}
N_2^{out}(t,r)=&\sum_{q\geq 2}\sum_{\substack{|s|\leq q\\q-s\in2\mathbb{Z}}}e^{\frac{isr^2}{4t}}\sum_{\substack{0\leq h\leq q\\k\geq 0}}t^{qs_c+2h+k}\sum_{(m,n,l,j)\in\mathcal{S}_q}t^{2\nu(\alpha_{m,n}+2j+4l)}N_{q,s,h,m,n}^{l,j,k}(r),
\end{aligned}
\end{equation}
where $N_{0,k}$, $N_{\pm1,m,n}^{l,j,k}$ and $N_{q,s,h,m,n}^{l,j,k}$ given by the above expressions of $N_0^{out}$ and $N_j^{out}(\chi)$, $j=1,2$.

\medskip We then substitute the ansatz  \eqref{ansatz chi out} into \eqref{eq of chi out} and use the following computations,
\begin{equation}\label{linear out 1}
\begin{aligned}
&\big(-i\partial_t-\Delta\big)\sum_{k\geq 1}t^{k}f_k(r)=-if_1-\sum_{k\geq 1}t^k\big(i(k+1)f_{k+1}+\Delta f_k\big),\\
&\big(-i\partial_t-\Delta\big)\Big(e^{\frac{ir^2}{4t}}\sum_{k\geq 0}t^{\alpha+k}g_k(r)\Big)\\
=&-e^{\frac{ir^2}{4t}}\Big[it^{\alpha-1}\big(\alpha+\frac{d}{2}+r\partial_r\big)g_0+\sum_{k\geq 0}t^{\alpha+k}\big(i(\alpha+k+1+\frac{d}{2}+r\partial_r)g_{k+1}+\Delta g_k\big)\Big],\\
&\big(-i\partial_t-\Delta\big)\Big(e^{-\frac{ir^2}{4t}}\sum_{k\geq 0}t^{\alpha+k+2}g_k(r)\Big)\\
=&e^{-\frac{ir^2}{4t}}\Big[t^{\alpha}\frac{r^2}{2}g_0+t^{\alpha+1}\Big(\frac{r^2}{2}g_1-i\big(\alpha+2-\frac{d}{2}-r\partial_r\big)g_0\Big)\\
&\qquad\quad+\sum_{k\geq 0}t^{\alpha+k+2}\Big(\frac{r^2}{2}g_{k+2}-i\big(\alpha+k+3-\frac{d}{2}-r\partial_r\big)g_{k+1}-\Delta g_k\Big)\Big],
\end{aligned}
\end{equation}
and 
\begin{equation}\label{linear out 2}
\begin{aligned}
&\big(-i\partial_t-\Delta\big)\Big(e^{\frac{isr^2}{4t}}\sum_{k\geq 0}t^{\alpha+k}g_k(r)\Big)\\
=&e^{\frac{isr^2}{4t}}\Big[t^{\alpha-2}\frac{s(s-1)r^2}{4}g_0+t^{\alpha-1}\Big(-i\big(\alpha+\frac{sd}{2}+sr\partial_r\big)g_0+\frac{s(s-1)r^2}{4}g_1\Big)\\
&\qquad+\sum_{k\geq 0}t^{\alpha+k}\Big(\frac{s(s-1)r^2}{4}g_{k+2}-i\big(\alpha+k+1+\frac{sd}{2}+sr\partial_r\big)g_{k+1}-\Delta g_k\Big)\Big].
\end{aligned}
\end{equation}

Thus, taking \eqref{nonlinear out 1} and \eqref{linear out 1} into \eqref{eq of chi out},  the equation of $f_k$ is 
\begin{equation}\label{out system 1}
\left\{
\begin{aligned}
&if_1+\Delta f_0+|f_0|^{2p}f_0=0,\\
&ikf_{k}+\Delta f_{k-1}+(p+1)|f_0|^{2p}f_{k-1}+p|f_0|^{2p-2}f_0^2\overline{f_{k-1}}+N_{0,k-1}=0,\,\,\,\,k\geq 2.
\end{aligned}
\right.
\end{equation}
with $N_{0,1}=0$. For any admissible pairs $(m,n,l,j)\in\mathcal{S}$, the equation of $g_{m,n}^{l,j,k,\pm}$ is equivelent to
\begin{equation}\label{out system 3}
\left\{
\begin{aligned}
&\frac{r^2}{2}g_{m,n}^{l,j,0,-}-p|f_0|^{2p-2}f_0^2 \overline{g_{m,n}^{l,j,0,+}}=0\\
&\frac{r^2}{2}g_{m,n}^{l,j,1,-}-i\big(s_c+2\nu(\alpha_{m,n}+2j+4l)+2-\frac{d}{2}-r\partial_r\big)g_{m,n}^{l,j,0,-}\\
&\qquad\qquad\qquad\qquad\qquad-p|f_0|^{2p-2}f_0^2\overline{g_{m,n}^{l,j,1,+}}-N_{-1,m,n}^{l,j,1}=0\\
&\frac{r^2}{2}g_{m,n}^{l,j,k,-}-i\big(s_c+2\nu(\alpha_{m,n}+2j+4l)+k+1-\frac{d}{2}-r\partial_r\big)g_{m,n}^{l,j,k-1,-}\\
&-\Delta g_{m,n}^{l,j,k-2,-}-(p+1)|f_0|^{2p}g_{m,n}^{l,j,k-2,-}-p|f_0|^{2p-2}f_0^2\overline{g_{m,n}^{l,j,k,+}}-N_{-1,m,n}^{l,j,k}=0,\,\,\,\,k\geq 2.
\end{aligned}
\right.
\end{equation}
and
 \begin{equation}\label{out system 2}
 \left\{
\begin{aligned}
&\big(s_c+2\nu(\alpha_{m,n}+2j+4l)+\frac{d}{2}+r\partial_r \big)g_{m,n}^{l,j,0,+}=0\\
&i\big(s_c+2\nu(\alpha_{m,n}+2j+4l)+\frac{d}{2}+k+r\partial_r \big)g_{m,n}^{l,j,k,+}+\Delta g_{m,n}^{l,j,k-1,+}\\
&\qquad\qquad\qquad\qquad+(p+1)|f_0|^{2p}g_{m,n}^{l, j,k-1,+}+N_{1,m,n}^{l,j,k-1}=0,\,\,\,\,\,k=1,2\\
&i\big(s_c+2\nu(\alpha_{m,n}+2j+4l)+\frac{d}{2}+k+r\partial_r\big)g_{m,n}^{l,j,k,+}+\Delta g_{m,n}^{l,j,k-1,+}\\
&\qquad+(p+1)|f_0|^{2p}g_{m,n}^{l,j,k-1,+}+p|f_0|^{2p-2}f_0^2\overline{g_{m,n}^{l,j,k-3,-}}+N_{1,m,n}^{l,j,k-1}=0,\,\,\,\,k\geq 3,
\end{aligned}
\right.
\end{equation}
with $N_{1,m,n}^{l,j,0}=0$. Similarly, for the higher order oscillation, for $q\geq 2$ and $(m,n,l,j)\in\mathcal{S}_q$, substituting \eqref{linear out 2} and \eqref{nonlinear out 2} into \eqref{eq of chi out}, we obtain the equation of $g_{q,s,h,m,n}^{l,j,k}$
\begin{equation}\label{out system 4}
\left\{
\begin{aligned}
&\frac{s(s-1)}{4}r^2g_{q,s,h,m,n}^{l,j,0}=0,\\
&\frac{s(s-1)}{4}r^2g_{q,s,h,m,n}^{l,j,1}-i\big(qs_c+2\nu(\alpha_{m,n}+2j+4l)+2h+\frac{sd}{2}+sr\partial_r \big)g_{q,s,h,m,n}^{l,j,0}=0\\
&\frac{s(s-1)}{4}r^2g_{q,s,h,m,n}^{l,j,k}-i\big(qs_c+2\nu(\alpha_{m,n}+2j+4l)+2h+k-1+\frac{sd}{2}+sr\partial_r \big)g_{q,s,h,m,n}^{l,j,k-1}\\
&-\Delta g_{q,s,h,m,n}^{l,j,k-2}-(p+1)|f_0|^{2p}g_{q,s,h,m,n}^{l,j,k-2}-p|f_0|^{2p-2}f_0^2\overline{g_{q,-s,h,m,n}^{l,j,k-2}}-N_{q,s,h,m,n}^{l,j,k-2}=0,\,\,\,\,k\geq 2.
\end{aligned}
\right.
\end{equation}
For $s=1$, the first equation of the system gives
\begin{equation*}
\begin{aligned}
g_{q,1,h,m,n}^{l,j,0}(r)={}&g_{q,1,h,m,n}^{l,j,0}(\delta)\left(\frac r\delta\right)^{-\frac1p-(q+1)s_c-2h-2\nu(\alpha_{m,n}+2j+4l)}.
\end{aligned}
\end{equation*}
The value $g_{q,1,h,m,n}^{l,j,0}(\delta)$ is fixed by matching with the self-similar region.

\medskip

\subsection{Analysis of the functions.}\label{subsec 4.2}
We solve the systems \eqref{out system 1}--\eqref{out system 4} with the homogeneous coefficients prescribed by matching. To formulate the result, we introduce the following function classes. We denote by $\mathcal A_{\mathcal S}$ the space of continuous functions $a\in C(\mathbb{R}_{+};\mathbb C)\cap C^\infty(\mathbb{R}_{+}^{\ast};\mathbb C)$ satisfying:
\begin{itemize}
\item[(i)] $\operatorname{supp}a\subset[0,2\delta]$.

\item[(ii)]
For $0\leq r<\delta$, there exist an integer $L$ such that the function $a$ admits an absolutely convergent
expansion of the form
\begin{equation}\label{expansion a}
a(r)=a_0%+\sum_{(m,n,l,j)\in\mathcal{S}}a_{m,n}^{l,j}r^{2\nu(\alpha_{m,n}+2j+4l)},
+\sum_{\substack{(m,n,l,j)\in\mathcal S\cup\{(0,0,0,0)\}\\qs_c+h+2\nu(\alpha_{m,n}+2j+4l)>0}} \sum_{\ell =0}^{L} a_{m,n}^{l,j,\ell} r^{qs_c+h+2\nu(\alpha_{m,n}+2j+4l)} (\ln r)^{\ell},
\end{equation}
with some coefficients $a_{m,n}^{l,j,q}$ and $q,h\in\mathbb N_0$.
\end{itemize}

\medskip For each integer $k\geq1$, we denote by $\mathcal B_{k,\mathcal S}$ the space of continuous functions $b\in C(\mathbb{R}_{+};\mathbb C)\cap C^\infty(\mathbb{R}_{+}^{\ast};\mathbb C)$ satisfying:
\begin{itemize}
\item[(i)]
For $0\leq r<\delta$, there exist an integer $L$ such that the function $b$ admits an absolutely convergent expansion of the form
\begin{equation}\label{expansion b}
b(r)=b_k+\sum_{\substack{(m,n,l,j)\in\mathcal S\cup\{(0,0,0,0)\}\\qs_c+h+2\nu(\alpha_{m,n}+2j+4l)>0}} \sum_{\ell =0}^{L} b_{m,n}^{l,j,k,q,h,\ell}
r^{qs_c+h+2\nu(\alpha_{m,n}+2j+4l)}(\ln r)^{\ell},
\end{equation}
with some coefficients $b_{m,n}^{l,j,k,q,h,\ell}$ and $q,h\in\mathbb N_0$.

\item[(ii)]
For $r\geq2\delta$, 
\begin{equation*}\label{expansion b infinity}
b(r)=\sum_{a=0}^{k-1}\sum_{\ell=0}^{L}b_{q,h,a,\ell}^{k,\infty}r^{qs_c+2h+a}(\ln r)^\ell,\qquad q,h\in\mathbb N_0,
\end{equation*}
where only finitely many coefficients $b_{q,h,a,\ell}^{k,\infty}$ are nonzero.
%$b$ is a finite sum of terms $r^{q s_c+2h}$, $q,h\in\mathbb N_0$ multiplied by polynomials of degree at most $k-1$.
\end{itemize}

\medskip
\begin{remark}
The exponents in the expansions \eqref{expansion a} and \eqref{expansion b} at the origin are
\begin{equation*}
    qs_c+h+2\nu(\alpha_{m,n}+2j+4l)>0, \qquad q,h\in\mathbb N_0,\quad (m,n,l,j)\in\mathcal S\cup\{(0,0,0,0)\}.
\end{equation*}
The terms $2\nu(\alpha_{m,n}+2j+4l)$ originate from matching with the self-similar profiles, while the additional shifts arise in the outer recursion. Logarithmic terms may be inherited from the matching data or generated by resonances. Indeed, the equation
\begin{equation*}
    (r\partial_r-k)f(r)=r^\beta
\end{equation*}
admits the particular solution
\begin{equation*}
    f(r)=\frac{1}{\beta-k}r^\beta, \qquad\text{for $\beta\neq k$}.
\end{equation*}
For $\beta=k$, one obtains
\begin{equation*}
    f(r)=r^k\ln r.
\end{equation*}
More generally,
\begin{equation*}
(r\partial_r-k)\left(\frac{r^k(\ln r)^{\ell+1}}{\ell+1}\right)=r^k(\ln r)^\ell,\qquad \ell\in\mathbb N_0.
\end{equation*}
Thus a resonant forcing term increases the logarithmic degree by one. Since the outer systems \eqref{out system 1}-\eqref{out system 4} are solved inductively, logarithmic terms produced at one order enter the inhomogeneous terms $N_{s,m,n}^{l,j,k}$ and $N_{q,s,h,m,n}^{l,j,k}$ determining the higher order coefficients, and repeated resonances may generate higher powers of $\ln r$. 
\end{remark}

\medskip By the closure of the index set $\mathcal S$, for any $k$, one has
\begin{equation*}\label{note 2}
\begin{aligned}
r\partial_{r}\mathcal{A}_{\mathcal{S}}\subset\mathcal{A}_{\mathcal{S}}, \,\,\, r\partial_{r}\mathcal{B}_{k,\mathcal{S}}\subset\mathcal{B}_{k, \mathcal{S}},\,\,\,\mathcal{A}_{\mathcal{S}}\mathcal{B}_{k, \mathcal{S}}\subset\mathcal{A}_{\mathcal{S}}.
\end{aligned}
\end{equation*}
Moreover, \eqref{eq f0}, \eqref{out system 2}, and \eqref{out system 4} imply
\begin{equation*}
    f_0\in r^{-\frac1p}\mathcal A_{\mathcal S},\qquad g_{m,n}^{l,j,0,+}(r)
    \sim r^{-\frac1p-2s_c-2\nu(\alpha_{m,n}+2j+4l)}\quad\text{as $r\rightarrow 0$},
\end{equation*}
and
\begin{equation*}
    g_{q,s,h,m,n}^{l,j,k}=0,\qquad k=0,1, \quad s\notin\{0,1\}.
\end{equation*}

\medskip We are now in a position to prove the following result.
\begin{lemma}\label{lem remote expansion}
{\sl Let $d\geq11$ and $p\in\mathbb{N}^{\ast}$ satisfy $2p+1>p_{JL}$. There exist unique solutions $\{f_k\}_{k\geq 1}$, $\{g_{m,n}^{l,j,k,\pm}\}_{k\geq 0, (m,n,l,j)\in\mathcal{S}}$ and 
$\{g_{q,s,h,m,n}^{l,j,k}\}_{\substack{k\geq 0,(q,s,h)\in\Omega,\\
(m,n,l,j)\in\mathcal{S}}}$ of systems \eqref{out system 1} to \eqref{out system 4} verifying
\begin{equation}\label{out expansion}
\begin{aligned}
&f_k\in r^{-\frac{1}{p}-2k}\mathcal{A}_{\mathcal{S}},\, g_{m,n}^{l,j,k,-}\in r^{-\frac{1}{p}-2s_c-2\nu(\alpha_{m,n}+2j+4l)-2(k+2)}\mathcal{A}_{\mathcal{S}},\\
&g_{q,0,h,m,n}^{l,j,k}\in r^{-\frac{1}{p}-2qs_c-4h-2\nu(\alpha_{m,n}+2j+4l)-2k}\mathcal{A}_{\mathcal{S}},\,\,k\geq 1\\
&g_{q,s,h,m,n}^{l,j,k}\in r^{-\frac{1}{p}-2qs_c-4h-2\nu(\alpha_{m,n}+2j+4l)-2k}\mathcal{A}_{\mathcal{S}},\,\,s\neq 0, 1,\,\,k\geq 2\\
&g_{m,n}^{l,j,k,+}\in r^{-\frac{1}{p}-2s_c-2\nu(\alpha_{m,n}+2j+4l)-2k}\mathcal{B}_{k+1,\mathcal{S}},\,\,k\geq 1,\\
&g_{q,1,h,m,n}^{l,j,k}\in r^{-\frac{1}{p}-2q s_c-4h-2\nu(\alpha_{m,n}+2j+4l)-2k}\mathcal{B}_{k+1,\mathcal{S}},\,\,k\geq 0,
\end{aligned}
\end{equation}
where
\begin{equation*}
g_{q,0,h,m,n}^{l,j,0}=0,\qquad g_{q,s,h,m,n}^{l,j,0}=g_{q,s,h,m,n}^{l,j,1}=0 \quad\text{for }s\notin\{0,1\},
\end{equation*}
and
\begin{equation*}
\begin{aligned}
g_{m,n}^{l,j,0,+}(r) &\sim r^{-\frac1p-2s_c-2\nu(\alpha_{m,n}+2j+4l)},\\
g_{q,1,h,m,n}^{l,j,0}(r) &\sim r^{-\frac1p-(q+1)s_c-2h-2\nu(\alpha_{m,n}+2j+4l)}.
\end{aligned}
\end{equation*}}
\end{lemma}

\begin{proof}[Proof of Lemma \ref{lem remote expansion}]

\noindent{\bf Step 1. Reduction of system \eqref{out system 4} according to the value of $s$.}
According to \eqref{out system 4}, for $s\neq 0,1$, $g_{q,s,h,m,n}^{l,j,k}=0$, $k=0,1$. Then the system \eqref{out system 4} is equivalent to 
\begin{equation}\label{out system 4 neq 01}
\begin{aligned}
&\frac{s(s-1)r^2}{4}g_{q,s,h,m,n}^{l,j,2}=N_{q,s,h,m,n}^{l,j,0},\\
&\frac{s(s-1)r^2}{4}g_{q,s,h,m,n}^{l,j,3}-i\Big(qs_c+2\nu\big(\alpha_{m,n}+2j+4l\big)+2h+2\\
&\qquad\qquad\qquad\qquad\qquad\quad+\frac{sd}{2}+sr\partial_{r}\Big)g_{q,s,h,m,n}^{l,j,2}=N_{q,s,h,m,n}^{l,j,1},\\
&\frac{s(s-1)r^2}{4}g_{q,s,h,m,n}^{l,j,k}-i\Big(qs_c+2\nu\big(\alpha_{m,n}+2j+4l\big)+2h+k-1\\
&\qquad\qquad\qquad\qquad\qquad\quad+\frac{sd}{2}+sr\partial_{r}\Big)g_{q,s,h,m,n}^{l,j,k-1}\\
&=\Delta g_{q,s,h,m,n}^{l,j,k-2}+(p+1)|f_0|^{2p}g_{q,s,h,m,n}^{l,j,k-2}+p|f_0|^{2p-2}f_0^2 \overline{g_{q,-s,h,m,n}^{l,j,k-2}}+N_{q,s,h,m,n}^{l,j,k-2},\,\,\,\,k\geq 4.
\end{aligned}
\end{equation}

For $s=0$, the first two lines of system \eqref{out system 4} becomes
\begin{equation*}
\big(qs_c+2\nu(\alpha_{m,n}+2j+4l)+2h\big)g_{q,0,h,m,n}^{l,j,0}=0.
\end{equation*}
By the definition of $\mathcal{S}$ and $\Omega$, $qs_c+2\nu(\alpha_{m,n}+2j+4l)+2h>0$, we have $g_{1,0,h,m,n}^{j,0,l}=0$. Hence the system \eqref{out system 4} is equivalent to 
\begin{equation}\label{out system 4 0}
\begin{aligned}
\big(qs_c+2\nu(\alpha_{m,n}+2j+4l)+2h+1\big)g_{q,0,h,m,n}^{l,j,1}=&iN_{q,0,h,m,n}^{l,j,0},\\
\big(qs_c+2\nu(\alpha_{m,n}+2j+4l)+2h+k\big)g_{q,0,h,m,n}^{l,j,k}=&i\Big(\Delta g_{q,0,h,m,n}^{l,j,k-1}+(p+1)|f_0|^{2p}g_{q,0,h,m,n}^{l,j,k-1}\\
&+p|f_0|^{2p-2}f_0^2\overline{g_{q,0,h,m,n}^{l,j,k-1}}+N_{q,0,h,m,n}^{l,j,k-1}\Big),\,\,\,k\geq 2.
\end{aligned}
\end{equation}
 For $s=1$, the system \eqref{out system 4} becomes
\begin{equation}\label{out system 4 1}
\begin{aligned}
&\big(qs_c+2\nu(\alpha_{m,n}+2j+4l)+2h+\frac{d}{2}+r\partial_{r}\big)g_{q,1,h,m,n}^{l,j,0}=0\\
&\big(qs_c+2\nu(\alpha_{m,n}+2j+4l)+2h+k+\frac{d}{2}+r\partial_{r}\big)g_{q,1,h,m,n}^{l,j,k}=i\Big(\Delta g_{q,1,h,m,n}^{l,j,k-1}\\
&\qquad\qquad+(p+1)|f_0|^{2p}g_{q,1,h,m,n}^{l,j,k-1}+p|f_0|^{2p-2}f_0^2\overline{g_{q,-1,h,m,n}^{l,j,k-1}}+N_{q,1,h,m,n}^{l,j,k-1}\Big),\,\,\,k\geq 1.
\end{aligned}
\end{equation}

\medskip
\noindent{\bf Step 2. The lowest functions on $k$.}
For $k=0$ in systems \eqref{out system 3}, \eqref{out system 4 1}, $k=1$ in system \eqref{out system 1}, \eqref{out system 2}, \eqref{out system 4 0}, and $k=2$ in system \eqref{out system 4 neq 01} give for any $(m,n,l,j)\in\mathcal{S}$,
\begin{align}
&f_1=i\big(\Delta f_0+|f_0|^{2p}f_0\big),\label{out k11}\\
&\frac{r^2}{2}g_{m,n}^{l,j,0,-}=p|f_0|^{2p-2}f_0^2\overline{g_{m,n}^{l,j,0,+}},\label{out k12}\\
&\frac{s(s-1)r^2}{4}g_{q,s,h,m,n}^{l,j,2}=N_{q,s,h,m,n}^{l,j,0},\,\,\,\,s\neq 0,1, \label{out k13}\\
&\big(qs_c+2\nu(\alpha_{m,n}+2j+4l)+2h+1\big)g_{q,0,h,m,n}^{l,j,1}=iN_{q,0,h,m,n}^{l,j,0},\label{out k14}\\
&\big(\frac{1}{p}+2s_c+2\nu(\alpha_{m,n}+2j+4l)+1+r\partial_{r}\big)g_{m,n}^{l,j,1,+}=D_{1,m,n}^{l,j,1,+}(r),\label{out k16}\\
&\big(\frac{1}{p}+(q+1)s_c+2\nu(\alpha_{m,n}+2j+4l)+2h+1+r\partial_{r}\big)g_{q,1,h,m,n}^{l,j,1}=D_{q,1,h,m,n}^{l,j,1},\label{out k15}
\end{align}
where 
\begin{equation}\label{def D1,m,n,l,j,1,+}
\begin{aligned}
D_{1,m,n}^{l,j,1,+}(r)=&i\big(\Delta g_{m,n}^{l,j,0,+}+(p+1)|f_0|^{2p}g_{m,n}^{l,j,0,+}\big), \\
D_{q,1,h,m,n}^{l,j,1}=&i\big(\Delta g_{q,1,h,m,n}^{l,j,0}+(p+1)|f_0|^{2p}g_{q,1,h,m,n}^{l,j,0}+p|f_0|^{2p-2}f_0^2\overline{g_{q,-1,h,m,n}^{l,j,0}}+N_{q,1,h,m,n}^{l,j,0}\big).
\end{aligned}
\end{equation}
Here we have used the fact that $s_c+\frac{d}{2}=\frac{1}{p}+2s_c$ in \eqref{out system 2} and \eqref{out system 4 1}.

\medskip Recall that $f_0\in r^{-\frac{1}{p}}\mathcal{A}_{\mathcal{S}}$, $g_{m,n}^{l,j,0,+}\sim r^{-\frac{1}{p}-2s_c-2\nu(\alpha_{m,n}+2j+4l)}$. By \eqref{out k11} and \eqref{out k12} that
\begin{equation*}
f_1\in r^{-\frac{1}{p}-2}\mathcal{A}_{\mathcal{S}},\,\,\,\,\,\,\,g_{m,n}^{j,0,l,-}\in r^{-\frac{1}{p}-2s_c-2\nu(\alpha_{m,n}+2j+4l)-4}\mathcal{A}_{\mathcal{S}}.
\end{equation*}
According to \eqref{nonlinear out 2}, we find that $N_{q,s,h,m,n}^{l,j,0}$ depend on $f_0$ and $g_{m,n}^{l,j,0,\pm}$. Hence $$N_{q,s,h,m,n}^{l,j,0}\in r^{-\frac{1}{p}-2qs_c-4h-2-2\nu(\alpha_{m,n}+2j+4l)}\mathcal{A}_{\mathcal{S}}.$$
 Therefore, we get from \eqref{out k13} and \eqref{out k14} that
\begin{equation*}
g_{q,0,h,m,n}^{l,j,1}\in r^{-\frac{1}{p}-2qs_c-4h-2-2\nu(\alpha_{m,n}+2j+4l)}\mathcal{A}_{\mathcal{S}},
\end{equation*}
and 
\begin{equation*}
g_{q,s,h,m,n}^{l,j,2}\in r^{-\frac{1}{p}-2qs_c-4h-4-2\nu(\alpha_{m,n}+2j+4l)}\mathcal{A}_{\mathcal{S}},\,\,\,s\neq 0,1.
\end{equation*}

\medskip It remains to consider \eqref{out k16} and \eqref{out k15}. Write 
\begin{equation*}
g_{m,n}^{l,j,1,+}(r)=r^{-\frac{1}{p}-2s_c-2\nu(\alpha_{m,n}+2j+4l)-1}\tilde{g}_{m,n}^{l,j,1,+}(r),
\end{equation*} 
then  \eqref{out k16} is equivalent to
\begin{equation}\label{eq g1+}
\partial_r \tilde{g}_{m,n}^{l,j,1,+}(r)=r^{-2}\tilde{D}_{1,m,n}^{l,j,1,+}(r),
\end{equation}
where 
\begin{equation}\label{def D1 out}
\tilde{D}_{1,m,n}^{l,j,1,+}(r)=r^{\frac{1}{p}+2s_c+2\nu(\alpha_{m,n}+2j+4l)+2}D_{1,m,n}^{l,j,1,+}(r).
\end{equation}
 Substituting the expansions of $f_0$ and $g_{m,n}^{l,j,0,+}$ into \eqref{def D1,m,n,l,j,1,+}, we obtain 
 $$D_{1,m,n}^{l,j,1,+}\in r^{-\frac{1}{p}-2s_c-2\nu(\alpha_{m,n}+2j+4l)-2}\mathcal{B}_{1,\mathcal{S}},$$
  then \eqref{def D1 out} implies
 \begin{itemize}
\item For $0\leq r<\delta$, $\tilde{D}_{1,m,n}^{l,j,1,+}$ admits an absolutely convergent expansion of the form
\begin{equation*}
\tilde{D}_{1,m,n}^{l,j,1,+}(r)=d_{m,n}^{l,j,1,+,0}+\sum_{(m',n',l',j')\in\mathcal{S}}d_{m',n'}^{l',j',1}r^{2\nu(\alpha_{m',n'}+2j'+4l')}.
\end{equation*}
\item For $r\geq 2\delta$, $\tilde{D}_{1,m,n}^{l,j,1,+}$ is a constant.
\end{itemize}
Clearly, there exists a unique solution $\tilde{g}_{m,n}^{l,j,1,+}$ of \eqref{eq g1+} such that 
$$g_{m,n}^{l,j,1,+}\in r^{-\frac{1}{p}-2s_c-2\nu(\alpha_{m,n}+2j+4l)-2}\mathcal{B}_{2,\mathcal{S}}.$$
 For every $(m',n',l',j')\in\mathcal{S}$, $\tilde{g}_{m,n}^{l,j,1,+}$ is given by 
\begin{equation}\label{compu g 1+}
\begin{aligned}
\tilde{g}_{m,n}^{l,j,1,+}(r)=&C_{m,n}^{l,j,1,+}+\sum_{2\nu(\alpha_{m',n'}+2j'+4l')=1}d_{m',n'}^{l',j',1}\,\ln r\\
&+\int_{0}^{r} s^{-2}\Big(\tilde{D}_{m,n}^{l,j,1,+}(s)-d_{m,n}^{l,j,1,+,0}-\sum\limits_{\substack{(m',n',l',j')\in\mathcal{S}\\2\nu(\alpha_{m',n'}+2j'+4l')\leq 1}}d_{m',n'}^{l',j',1}s^{2\nu(\alpha_{m',n'}+2j'+4l')}\Big)\ud s\\
&-\int_{r}^{\infty}s^{-2}\Big(d_{m,n}^{l,j,1,+,0}+\sum\limits_{\substack{(m',n',l',j')\in\mathcal{S}\\2\nu(\alpha_{m',n'}+2j'+4l')< 1}}d_{m',n'}^{l',j',1}s^{2\nu(\alpha_{m',n'}+2j'+4l')}\Big)\ud s.
\end{aligned}
\end{equation}
The matching condition from self-similar region fixes $C_{m,n}^{l,j,1,+}$ uniquely.

\medskip Similarly, for the equation \eqref{out k15} with \eqref{def D1,m,n,l,j,1,+}, write 
\begin{equation*}
g_{q,1,h,m,n}^{l,j,1}(r)=r^{-\frac{1}{p}-2q s_c-2\nu(\alpha_{m,n}+2j+4l)-4h-2}\tilde{g}_{q,1,h,m,n}^{l,j,1}(r),
\end{equation*}
and 
\begin{equation*}
D_{q,1,h,m,n}^{l,j,1}=r^{-\frac{1}{p}-2q s_c-2\nu(\alpha_{m,n}+2j+4l)-4h-2}\tilde{D}_{q,1,h,m,n}^{l,j,1}(r),
\end{equation*}
then,
$$\big(r\partial_{r}-(q-1)s_c-2h-1\big)\tilde{g}_{q,1,h,m,n}^{l,j,1}=\tilde{D}_{q,1,h,m,n}^{l,j,1},$$
 and $\tilde{D}_{q,1,h,m,n}^{l,j,1}\in\mathcal{B}_{1,\mathcal{S}}$. Similar to the calculation of $g_{m,n}^{l,j,1,+}$, with logarithmic powers increasing only at $q's_c+h'+2\nu(\alpha_{m',n'}+2j'+4l')=(q-1)s_c+2h+1$ for $(m',n',l',j')\in\mathcal{S}$. We deduce that the equation \eqref{out k15} has a unique solution 
$$g_{q,1,h,m,n}^{l,j,1}\in r^{-\frac{1}{p}-2q s_c-4h-2-2\nu(\alpha_{m,n}+2j+4l)}\mathcal{B}_{2,\mathcal{S}}.$$

\medskip

\noindent{\bf Step 3. Induction on the expansion of functions on $k$.}
We next proceed by induction. Suppose we have solved \eqref{out system 1}-\eqref{out system 4} with $\ell=1,2,\cdots k-1$, $k\geq 4$, and have found $\{f_\ell\}_{1\leq \ell\leq k-1}$, $\{g_{m,n}^{l,j,\ell,\pm}\}_{\substack{(m,n,l,j)\in\mathcal{S}\\1\leq \ell\leq k-1}}$ and $\{g_{q,s,h,m,n}^{l,j,\ell}\}_{\substack{(m,n,l,j)\in\mathcal{S},(q,s,h)\in\Omega\\1\leq \ell\leq k-1}}$ verifying \eqref{out expansion}. Consider $\ell=k$. It follows from systems \eqref{out system 1}--\eqref{out system 3} and \eqref{out system 4 neq 01}--\eqref{out system 4 0} that we have to solve the following equations:
\begin{align}
&f_k=\frac{i}{k}\big(\Delta f_{k-1}+(p+1)|f_0|^{2p}f_{k-1}+p|f_0|^{2p-2}f_0^2\overline{f_{k-1}}+N_{0,k-1}\big),\label{eq fk}\\
&\frac{r^2}{2}g_{m,n}^{l,j,k,-}(r)=i\big(-\frac{1}{p}+2\nu(\alpha_{m,n}+2j+4l)+k+1-r\partial_{r}\big)g_{m,n}^{l,j,k-1,-}\nonumber\\
&\qquad\qquad\quad+\Delta g_{m,n}^{l,j,k-2,-}+(p+1)|f_0|^{2p}g_{m,n}^{l,j,k-2,-}+p|f_0|^{2p-2}f_0^2\overline{g_{m,n}^{l,j,k,+}}+N_{-1,m,n}^{l,j,k},\label{eq g- k}
\end{align}
and 
\begin{equation}\label{eq g s neq 01}
\begin{aligned}
\frac{s(s-1)r^2}{4}g_{q,s,h,m,n}^{l,j,k}=&i\Big(qs_c+2\nu\big(\alpha_{m,n}+2j+4l\big)+2h+k-1+\frac{sd}{2}\\
&\quad+sr\partial_{r}\Big)g_{q,s,h,m,n}^{l,j,k-1}+\Delta g_{q,s,h,m,n}^{l,j,k-2}+(p+1)|f_0|^{2p}g_{q,s,h,m,n}^{l,j,k-2}\\
&+p|f_0|^{2p-2}f_0^2 \overline{g_{q,-s,h,m,n}^{l,j,k-2}}+N_{q,s,h,m,n}^{l,j,k-2},
\end{aligned}
\end{equation}
for $s\in\mathbb{Z}$, $s\neq 1$. By \eqref{nonlinear out 1}, \eqref{nonlinear out 2} and \eqref{out expansion}, the right handside of \eqref{eq fk}-\eqref{eq g s neq 01} satisfy
\begin{equation*}
\begin{aligned}
&N_{0,k-1}\in r^{-\frac{1}{p}-2k}\mathcal{A}_{\mathcal{S}}, \,\,\,N_{-1,m,n}^{l,j,k}\in r^{-\frac{1}{p}-2s_c-2\nu(\alpha_{m,n}+2j+4l)-2k-2}\mathcal{A}_{\mathcal{S}},\\
&N_{q,s,h,m,n}^{l,j,k-2}\in r^{-\frac{1}{p}-2q s_c-4h-2(k-2)-2\nu(\alpha_{m,n}+2j+4l)-2}\mathcal{A}_{\mathcal{S}},\quad s\neq 0,1,\\
&N_{q,0,h,m,n}^{l,j,k-1}\in r^{-\frac{1}{p}-2q s_c-4h-2(k-1)-2\nu(\alpha_{m,n}+2j+4l)-2}\mathcal{A}_{\mathcal{S}}.
\end{aligned}
\end{equation*}
Thus, with the prescribed matching normalization,
\begin{equation*}
\begin{aligned}
&f_k\in r^{-\frac{1}{p}-2k}\mathcal{A}_{\mathcal{S}},\, g_{m,n}^{l,j,k,-}\in r^{-\frac{1}{p}-2s_c-2\nu(\alpha_{m,n}+2j+4l)-2(k+2)}\mathcal{A}_{\mathcal{S}},\\
&g_{q,s,h,m,n}^{l,j,k}\in r^{-\frac{1}{p}-2 qs_c-4h-2k-2\nu(\alpha_{m,n}+2j+4l)}\mathcal{A}_{\mathcal{S}},\quad\text{for $s\neq 1$}.
\end{aligned}
\end{equation*}

\medskip We finally consider the equation of $g_{m,n}^{j,k,l,+}$ and $g_{q,1,h,m,n}^{j,k,l}$  from the systems \eqref{out system 2}, \eqref{out system 4 1}
\begin{equation}\label{eq s=1}
\begin{aligned}
&\big(\frac{1}{p}+2s_c+2\nu(\alpha_{m,n}+2j+4l)+k+r\partial_{r}\big)g_{m,n}^{l,j,k,+}=D_{m,n}^{l,j,k,+},\\
&\big(\frac{1}{p}+(q+1)s_c+2\nu(\alpha_{m,n}+2j+4l)+2h+k+r\partial_{r}\big)g_{q,1,h,m,n}^{l,j,k}=D_{q,1,h,m,n}^{l,j,k},
\end{aligned}
\end{equation}
where $D_{m,n}^{l,j,k,+}$ and $D_{q,1,h,m,n}^{l,j,k}$ are given by:
\begin{equation*}
\begin{aligned}
D_{m,n}^{l,j,k,+}=&i\Big(\Delta g_{m,n}^{l,j,k-1,+}+(p+1)|f_0|^{2p}g_{m,n}^{l,j,k-1,+}+p|f_0|^{2p-2}f_0^2\overline{g_{m,n}^{l,j,k-3,-}}+N_{1,m,n}^{l,j,k-1}\Big),\\
D_{q,1,h,m,n}^{l,j,k}=&i\Big(\Delta g_{q,1,h,m,n}^{l,j,k-1}+(p+1)|f_0|^{2p}g_{q,1,h,m,n}^{l,j,k-1}+p|f_0|^{2p-2}f_0^2\overline{g_{q,-1,h,m,n}^{l,j,k-1}}+N_{q,1,h,m,n}^{l,j,k-1}\Big).
\end{aligned}
\end{equation*}
Setting
\begin{equation*}
\begin{aligned}
\tilde{D}_{m,n}^{l,j,k,+}(r):=&r^{\frac{1}{p}+2s_c+2\nu(\alpha_{m,n}+2j+4l)+2k}D_{m,n}^{l,j,k,+}(r),\\
\tilde{D}_{q,1,h,m,n}^{l,j,k}(r):=& r^{\frac{1}{p}+2q s_c+4h+2\nu(\alpha_{m,n}+2j+4l)+2k}D_{q,1,h,m,n}^{l,j,k}(r),
\end{aligned}
\end{equation*}
by \eqref{nonlinear out 1}, \eqref{nonlinear out 2} and \eqref{out expansion},  $\tilde{D}_{m,n}^{l,j,k,+}\in\mathcal{B}_{k,\mathcal{S}}$ and $\tilde{D}_{q,1,h,m,n}^{l,j,k}\in \mathcal{B}_{k,\mathcal{S}}$.

\medskip We only treat the first equation of the system \eqref{eq s=1}, since the second equation can be handled in the same way. Write
\begin{equation*}%\label{def of gk+}
g_{m,n}^{l,j,k,+}(r)=r^{-\frac{1}{p}-2s_c-2\nu(\alpha_{m,n}+2j+4l)-k}\tilde{g}_{m,n}^{l,j,k,+}(r),
\end{equation*}
then
\begin{equation}\label{eq tilde g1k+}
\partial_r\widetilde g_{m,n}^{l,j,k,+}(r) = r^{-k-1} \widetilde D_{m,n}^{l,j,k,+}(r).
\end{equation}
\medskip Hence, for $0< r<\delta$, by the definition of $\mathcal{B}_{k,\mathcal{S}}$,
\begin{equation*}
\tilde{D}_{m,n}^{l,j,k+}(r)=d_{m,n}^{l,j,k,+,0}+\sum\limits_{(m',n',l',j')\in\mathcal{S}}\sum_{0\leq a\leq L}d_{m',n'}^{l',j'k',+,a}r^{2\nu(\alpha_{m',n'}+2j'+4l')}(\ln r)^{a}.
\end{equation*}
Following the procedure in \eqref{compu g 1+} for the case $k=1$, we integrate \eqref{eq tilde g1k+} to obtain
\begin{equation*}
\tilde g_{m,n}^{l,j,k,+}(r)=C_{m,n}^{l,j,k,+}-\frac{d_{m,n}^{l,j,k,+,0}}{k}r^{-k}+\sum_{(m',n',l',j')\in\mathcal S}\sum_{0\leq a\leq L}\tilde d_{m',n'}^{l',j',k,+,a}r^{2\nu(\alpha_{m',n'}+2j'+4l')-k}(\ln r)^a,
\end{equation*}
where the increase by one occurs only when $(m',n',l',j')\in \mathcal{S}$ satisfying
\begin{equation*}
2\nu(\alpha_{m',n'}+2j'+4l')=k.
\end{equation*}
Therefore,
\begin{equation*}
\tilde g_{m,n}^{l,j,k,+}\in r^{-k}\mathcal B_{k+1,\mathcal S},
\end{equation*}
and consequently
\begin{equation*}
g_{m,n}^{l,j,k,+}\in r^{-\frac1p-2s_c-2\nu(\alpha_{m,n}+2j+4l)-2k} \mathcal B_{k+1,\mathcal S}.
\end{equation*}

\end{proof}

\bigbreak

\subsection{Estimate of the approximate solution in the remote region}\label{subsec 4.3}
Under the above notations, set for any integer $N_{out}\geq 2$,
\begin{equation}\label{def psi out}
\begin{aligned}
\psi_{out}^{ap}(t,x)=&f_0(r)+\sum_{1\leq k\leq N_{out}}t^kf_k(r)+e^{\frac{ir^2}{4t}}\sum_{\substack{0\leq k\leq N_{out}\\(m,n,l,j)\in\mathcal{S}}}t^{s_c+2\nu(\alpha_{m,n}+2j+4l)+k}g_{m,n}^{l,j,k,+}(r)\\
&+e^{-\frac{ir^2}{4t}}\sum_{\substack{0\leq k\leq N_{out}\\(m,n,l,j)\in\mathcal{S}}}t^{s_c+2\nu(\alpha_{m,n}+2j+4l)+k+2}g_{m,n}^{l,j,k,-}(r)\\
&+\sum_{\substack{(q,s,h)\in \Omega\\2\leq q\leq N_{out}}}e^{\frac{isr^2}{4t}}\sum_{\substack{0\leq k\leq N_{out}\\(m,n,l,j)\in\mathcal{S}_q}}t^{qs_c+2\nu(\alpha_{m,n}+2j+4l)+2h+k}g_{q,s,h,m,n}^{l,j,k}(r),
\end{aligned}
\end{equation}
and 
\begin{align*}\label{def chi out}
u_{out}^{ap}(t,\rho)=\lambda^{-\frac{1}{p}}(t)\psi_{out}^{ap}(t,t^{\frac{1}{2}+\nu}\rho).
\end{align*}

As a direct consequence of the previous analysis, we get
\begin{lemma}\label {lem Uout est}
{\sl There exists $T(N_{out},\delta)>0$ such that for $0< t\leq T(N_{out},\delta)$, the following holds.

(1). The profile $\psi_{out}^{ap}$ verifies
  \begin{equation}\label {L infinity out}
  \begin{aligned}
   \left \| r^{-k}\partial_{r}^l \big(\psi_{out}^{ap}-f_0\big)\right\|_{L^{\infty}(\mathcal{S}_{out})}\leq t^{-(\frac{1}{2}-\epsilon_2)(\frac{1}{p}+k+l)}\big(&Ct^{2\epsilon_2}+C t^{2\epsilon_2(s_c-l)+\nu(1+2\epsilon_2)|\lambda_{+}|}\big),   \\ 
   &l-k\leq \frac{1}{p}+2s_c+2\nu|\lambda_{+}|.
   \end{aligned}
   \end{equation}
   \begin{align}
&\left\| \partial_{r}^l \big( \psi_{out}^{ap}-f_0\big)\right\|_{L^{2}(\mathcal{S}_{out})}\leq C\delta^{s_c-l}+C t^{(\frac{1}{2}+\epsilon_2)(2\nu|\lambda_{+}|+s_c-l)},\quad 0\leq l\leq s_c,\label {L2 out small} \\ 
&\left\|\partial_{r}^l \big(\psi_{out}^{ap}-f_0\big)\right\|_{L^{2}(\mathcal{S}_{out})}\leq C\delta^{s_c-l}+Ct^{(\frac{1}{2}-\epsilon_2)(s_c-l)}, \quad s_c<l<s_c+2\nu|\lambda_{+}|.\label {L2 out large}
\end{align}
(2). For any $k\geq 0$, $0\leq l\leq s_c+2\nu|\lambda_{+}|$ and $\frac{1}{10}t^{\frac{1}{2}-\epsilon_2}\leq r\leq 10 t^{\frac{1}{2}-\epsilon_2}$, one has
\begin{align}\label{match est out}
\left|\partial_{r}^{l}\big(\psi_{ss}^{ap}-\psi_{out}^{ap}\big)\right|\leq r^{-\frac{1}{p}-l}\big(C t^{2\epsilon_2(N_{out}+1)}+Ct^{\nu(1-2\epsilon_2)|\lambda_{+}|(N_{ss}+1)}\big).
\end{align}}
\end{lemma}

Let us now consider  the remainder
  \begin{equation*}
  \begin{aligned}
  \mathcal{R}_{out}^ {ap}(t,\rho)=&% [ \eqref{eq of U in} U_{out}^ {N_{out}} ](t,\rho)\, . 
  -it^{1+2\nu}\partial_{t}u_{out}^{ap}+(\frac{1}{2}+\nu)it^{2\nu}\Lambda u_{out}^{ap}-\Delta u_{out}^{ap}-|u_{out}^{ap}|^{2p}u_{out}^{ap},\\
  R_{out}^{ap}(t,r)=&-i\partial_{t}\psi_{out}^{ap}-\Delta \psi_{out}^{ap}-|\psi_{out}^{ap}|^{2p}\psi_{out}^{ap}.
  \end{aligned}
  \end{equation*}
We have
$$\mathcal{R}_{out}^ {ap}(t,\rho) =\lambda^{-\frac{1}{p}-2}(t)R_{out}^{ap} \Big(t,\frac{\rho}{\lambda(t)}\Big) .$$

It follows from the expression of $\psi_{out}^{ap}$ that for $ l\geq 0$,
\begin{equation*}
\left\| \partial_{r}^{l}   R_{out}^ {ap}(t,r)\right\|_{L^2(\mathcal{S}_{out})} \leq C t^{2\epsilon_2(N_{out}+1)-1+(1/2-\epsilon_2)(s_c-l)}\big(1+t^{2\epsilon_2(s_c-l)+\nu(1+2\epsilon_2)|\lambda_+|}\big),
\end{equation*}
provided $N_{out}$ is sufficiently large. Then, we obtain the following lemma.
\medbreak

  \begin{lemma}\label {lem Rout}
{\sl For every integer $l\geq 0$,  there exist $C>0$ and $T>0$ such that, for $0<t\leq T$,
\begin{equation}\label {R out est} 
\left\| \partial_{\rho}^{l}   \mathcal{R}_{out}^ {ap}(t, \rho)\right\|_{L_{\rho}^2(\mathcal{S}_{out})} \leq Ct^{2\nu+2\epsilon_2(N_{out}+1)-(\nu+\epsilon_2)(s_c-l)}\big(1+t^{2\epsilon_2(s_c-l)+\nu(1+2\epsilon_2)|\lambda_+|}\big),
\end{equation}
provided that $N_{out}$ is sufficiently large. In particular, if $l<s_c+2\nu|\lambda_+|$, then
\begin{equation*}
\left\| \partial_{\rho}^{l}   \mathcal{R}_{out}^ {ap}(t, \rho)\right\|_{L_{\rho}^2(\mathcal{S}_{out})} \leq Ct^{2\nu+2\epsilon_2(N_{out}+1)-(\nu+\epsilon_2)(s_c-l)}.
\end{equation*}}
\end{lemma}

\bigbreak

   \section{Approximate solution in the whole space}\label{sec 5}
We are now in a position to conclude the proof of Proposition \ref{main prop}.
Fix $\epsilon_2$ such that $0<\epsilon_2<\min\big\{\nu,\frac{1}{2}\big\}$ and consider the radial approximate solution $\psi^{ap}(t,x)$:
\begin{equation}\label {def rad psi ap}
\begin{aligned}
 \psi^{ap}(t,x):=& \Theta \big( t^{-\frac{1}{2}- \epsilon_1}r\big)\psi_{in}^ {ap}(t,x) +\big(1-\Theta \big( t^{-\frac{1}{2}- \epsilon_1}r\big)\big)\Theta \big( t^{-\frac{1}{2}+ \epsilon_2}r\big)\psi_{ss}^{ap}(t,x)\\
&+\big(1-\Theta \big( t^{-\frac{1}{2}+ \epsilon_2}r\big)\big)\psi_{out}^{ap}(t,x),
\end{aligned}
\end{equation}
where $\psi_{in}^{ap}$, $\psi_{ss}^{ap}$ and $\psi_{out}^{ap}$ are defined in \eqref{def psi in}, \eqref{def psi ss} and \eqref{def psi out}. Write $\psi^{ap}$ as
\begin{equation*}\label{def chi ap}
\psi^{ap}(t,x)=\lambda^{\frac{1}{p}}\big(Q(\rho)+\chi^{ap}(t,\rho)\big),\quad \rho=\lambda(t) |x|.
\end{equation*}

\medskip For $\rho\geq \frac{1}{10}t^{-(\nu+\epsilon_2)}$, according to the estimate \eqref{L infinity out}, we write the $\chi^{ap}$ as
\begin{equation*}
\chi^{ap}(t,\rho)=\lambda^{-\frac{1}{p}}\big(\psi^{ap}(t,x)-f_0(x)\big)+\big(\lambda^{-\frac{1}{p}}f_0(x)-Q(\rho)\big).
\end{equation*}
By the expansions \eqref{expansion Q} and \eqref{eq f0}, we have
\begin{equation*}
\Big|\lambda^{-\frac{1}{p}}f_0(\frac{\rho}{\lambda})-Q(\rho)\Big|\leq C\big|\Theta(\frac{\rho}{\lambda\delta})-1\big|\rho^{-\frac{1}{p}}+C\Theta(\frac{\rho}{\lambda \delta})\lambda^{-2\nu|\lambda_{+}|}\rho^{-\frac{1}{p}+2\nu|\lambda_{+}|}+C\rho^{-\frac{1}{p}-|\lambda_{+}|}.
\end{equation*}
This implies that for $l\geq 0$, and $0<\epsilon_2<\min\{\nu,\frac12\}$,
\begin{equation}\label{L infty f0}
\begin{aligned}
\Big\|\nabla^l\big(\lambda^{-\frac{1}{p}}f_0\big(\frac{\rho}{\lambda}\big)-Q(\rho)\big)\Big\|_{L^{\infty}(\rho\geq \frac{1}{10}t^{-\nu-\epsilon_2)}}\leq& t^{(\frac{1}{2}+\nu)(\frac{1}{p}+l)},\quad  l\leq 2\nu|\lambda_+|-\frac{1}{p},\\
\Big\|\nabla^l\big(\lambda^{-\frac{1}{p}}f_0\big(\frac{\rho}{\lambda}\big)-Q(\rho)\big)\Big\|_{L^{\infty}(\rho\geq \frac{1}{10}t^{-\nu-\epsilon_2)}}\leq& C t^{(\nu+\epsilon_2)(\frac1p+l)+(1-2\epsilon_2)\nu|\lambda_+|},\quad  l> 2\nu|\lambda_+|-\frac{1}{p}.
\end{aligned}
\end{equation}

\medskip By Lemma \ref{lem: sol chi in L infinity}, Lemma \ref {lem L infinity ss}, and Lemma \ref {lem Uout est}, there exists a positive parameter $\delta_0(N_{out})$ such that, for any $\delta \leq \delta_0(N_{out})$, there exists a positive time $T=T(\delta)$ for which the approximate solution $\psi^{ap}$ defined by \eqref{def rad psi ap} satisfies the following $L^{\infty}$ estimates:
\begin{equation} \label{L infinity whole}
\big \| \langle \rho\rangle^{-k} \partial_{\rho}^{l} \chi^ {ap} (t)\big\|_{L^{\infty}(\R^d)} \leq  C\big( t^{\frac{\nu}{p}}+t^{(\nu-\epsilon_1)(\frac{1}{p}+|\lambda_{+}|)}\big).
  \end{equation} 

\medbreak

 Along the same lines, using Lemma \ref {lem: sol chi in L 2}, \eqref{L2 out small}, \eqref{L2 out large} and \eqref{L infty f0}, we obtain
 \begin{equation}\label{whole L2 est 0}
 \begin{aligned}
 &\big\|\nabla^{l}\chi^{ap}(t)\big\|_{L^2(\rho\leq 10 t^{-\nu-\epsilon_2})}\leq %C t^{(1-2\epsilon_2)\nu|\lambda_{+}|-(\nu+\epsilon_2)(s_c-l)}, &&0<l<s_c,\\
 C t^{-(\nu+\epsilon_2)(s_c-l)}, &&0<l<s_c,\\
 &\big\|\nabla^{l}\chi^{ap}(t)\big\|_{L^2(\rho\leq 10 t^{-\nu-\epsilon_2})}\leq Ct^{2\nu}+Ct^{(1-2\epsilon_2)\nu|\lambda_{+}|}, &&s_c<l<s_c+2\nu|\lambda_{+}|,\\
 &\big\|\nabla^{l}\big(\chi^{ap}(t)+\big(1-\Theta(\frac{\rho}{\lambda(t)\delta})\big)Q\big)\big\|_{L^2(\rho\geq \frac{1}{10} t^{-\nu-\epsilon_2})}\leq Ct^{-(\frac{1}{2}+\nu)(s_c-l)}, &&0<l<s_c,\\
  &\big\|\nabla^{l}\chi^{ap}(t)\big\|_{L^2(\rho\geq \frac{1}{10} t^{-\nu-\epsilon_2})}\leq Ct^{(\nu+\epsilon_2)(l-s_c)},&&s_c<l<s_c+2\nu|\lambda_{+}|.
 \end{aligned}
 \end{equation}
Here $C$ depends on $\delta$, $l$, and the approximation orders, and $T=T(\delta)>0$ is chosen sufficiently small
so that $10t^{\frac12-\epsilon_2}<\delta$ for $0<t\leq T$.

\medskip Inequalities \eqref{whole L2 est 0} imply, in particular
\begin{equation} \label{whole L2 est}
\begin{aligned}
  &\big\|\nabla^{l}\big(\chi^{ap}(t)+\big(1-\Theta(\frac{\rho}{\lambda(t)\delta})\big)Q\big)\big\|_{L^2(\mathbb{R}^d)}\leq C t^{-(\frac{1}{2}+\nu)(s_c-l)},&&0<l<s_c,\\
  &\big\|\nabla^{l}\chi^{ap}(t)\big\|_{L^2(\mathbb{R}^d)}\leq Ct^{2\nu}+C t^{(\nu+\epsilon_2)(l-s_c)},&&s_c<l<s_c+2\nu|\lambda_{+}| .\end{aligned}
  \end{equation}

\medskip
Finally, we consider the error associated with the global approximate profile:
\begin{align}\label{def of global error}
R^{ap}=%[\eqref{NLS}\psi^{ap}(t,r)].
-i\partial_{t}\psi^{ap}-\Delta \psi^{ap}-|\psi^{ap}|^{2p}\psi^{ap}.
\end{align}
Since $0<\epsilon_2<\min\{\nu,\frac12\}$ and $\epsilon_1>0$, for sufficiently small $t>0$, the support of $\Theta(t^{-\frac{1}{2}-\epsilon_1}r)$ is contained in the region where $\Theta(t^{-\frac{1}{2}+\epsilon_2}r)$ is identically one. Then for all multi-indices $\alpha$, $\beta$ with $|\alpha|, |\beta|\geq 1$, there holds
\begin{equation*}
\begin{aligned}
&\big(\partial^{\alpha}\Theta(t^{-\frac{1}{2}-\epsilon_1}r)\big)\Theta(t^{-\frac{1}{2}+\epsilon_2}r)=\partial^{\alpha}\Theta(t^{-\frac{1}{2}-\epsilon_1}r),\\
&\big(1-\Theta(t^{-\frac{1}{2}-\epsilon_1}r)\big)\partial^{\alpha}\Theta(t^{-\frac{1}{2}+\epsilon_2}r)=\partial^{\alpha}\Theta(t^{-\frac{1}{2}+\epsilon_2}r),\\
&\partial^{\alpha}\Theta(t^{-\frac{1}{2}-\epsilon_1}r)\cdot\partial^{\beta}\Theta(t^{-\frac{1}{2}+\epsilon_2}r)=0.
\end{aligned}
\end{equation*}
Therefore, substituting \eqref{def rad psi ap} into \eqref{def of global error}, we decompose the error into the following four terms:
\begin{align*}
R^{ap}=E_1+E_2+E_3+E_4,
\end{align*}
where  $\tilde{\Theta}(\xi)=\xi \cdot \nabla \Theta(\xi)$ and 
\begin{equation*}
\begin{aligned}
E_1= & (\frac{1}{2}+\epsilon_1)i t^{-1}\tilde{\Theta}(t^{-\frac{1}{2}-\epsilon_1}r)\big(\psi_{in}^{ap}-\psi_{ss}^{ap}\big)-2 t^{-(\frac{1}{2}+\epsilon_1)}\nabla{\Theta}(t^{-\frac{1}{2}-\epsilon_1}r)\cdot\nabla\big(\psi_{in}^{ap}-\psi_{ss}^{ap}\big)\\
&-t^{-2(\frac{1}{2}+\epsilon_1)}\Delta\Theta(t^{-\frac{1}{2}-\epsilon_1}r)\big(\psi_{in}^{ap}-\psi_{ss}^{ap}\big),\\
%\label{E1}\\
E_2= &(\frac{1}{2}-\epsilon_2)it^{-1}\tilde{\Theta}(t^{-\frac{1}{2}+\epsilon_2}r)\big(\psi_{ss}^{ap}-\psi_{out}^{ap}\big)  -2 t^{-(\frac{1}{2}-\epsilon_2)}\nabla{\Theta}(t^{-\frac{1}{2}+\epsilon_2}r)\cdot\nabla\big(\psi_{ss}^{ap}-\psi_{out}^{ap}\big)\\
&-t^{-2(\frac{1}{2}-\epsilon_2)}\Delta\Theta(t^{-\frac{1}{2}+\epsilon_2}r)\big(\psi_{ss}^{ap}-\psi_{out}^{ap}\big), \\%\label{E2}\\
E_3=&\Theta\big(t^{-\frac{1}{2}-\epsilon_1}r\big)R_{in}^{ap}+\big(1-\Theta\big(t^{-\frac{1}{2}-\epsilon_1}r\big)\big)\Theta\big(t^{-\frac{1}{2}+\epsilon_2}r\big)R_{ss}^{ap}+\big(1-\Theta\big(t^{-\frac{1}{2}+\epsilon_2}r\big)\big)R_{out}^{ap},\\%\label{E3}\\
E_4=&\Theta\big(t^{-\frac{1}{2}-\epsilon_1}r\big)\big(|\psi_{in}^{ap}|^{2p}\psi_{in}^{ap}-|\psi^{ap}|^{2p}\psi^{ap}\big)\\
&+\big(1-\Theta\big(t^{-\frac{1}{2}-\epsilon_1}r\big)\big)\Theta\big(t^{-\frac{1}{2}+\epsilon_2}r\big)\big(|\psi_{ss}^{ap}|^{2p}\psi_{ss}^{ap}-|\psi^{ap}|^{2p}\psi^{ap}\big)\\
&+\big(1-\Theta\big(t^{-\frac{1}{2}+\epsilon_2}r\big)\big)\big(|\psi_{out}^{ap}|^{2p}\psi_{out}^{ap}-|\psi^{ap}|^{2p}\psi^{ap}\big).%\label{E4}
\end{aligned}
\end{equation*}
Here
\begin{equation*}
R_{in}^{ap}(t,x)=t^{-(\frac{1}{2}+\nu)(\frac{1}{p}+2)}\mathcal{R}_{in}^{ap}(t, t^{-\frac{1}{2}-\nu}r),\quad R_{ss}^{ap}(t,x)=t^{-(\frac{1}{2}+\nu)(\frac{1}{p}+2)}\mathcal{R}_{ss}^{ap}(t, t^{-\frac{1}{2}-\nu}r).
\end{equation*}

\medbreak 
According to the estimates \eqref{L infinity whole} and \eqref{whole L2 est}, we set
\begin{equation*}
\frac d2<s<s_c+2\nu|\lambda_+|,\qquad0<\epsilon_1<\nu.
\end{equation*}
First, we address $E_1$. By \eqref{est psi in ss any N}, 
\begin{equation}\label{est E1 whole}
\begin{aligned}
\|E_1(t)\|_{\dot H^1\cap\dot H^s}\leq C\Big(&t^{2K\epsilon_1+(\nu-\epsilon_1)|\lambda_+|-1+(\frac12+\epsilon_1)(s_c-s)}\\
&+t^{(\nu-\epsilon_1)|\lambda_+|(N_{ss}+1)+(\frac12+\epsilon_1)(s_c-s-2)}\Big).
\end{aligned}
\end{equation}
Similarly, using \eqref{match est out}, we get for $E_2$ 
\begin{equation}\label{est E2 whole}
\begin{aligned}
\|E_2(t)\|_{\dot H^1\cap\dot H^s}\leq C\Big(&t^{2\epsilon_2(N_{out}+1)-1+(\frac12-\epsilon_2)s_c-(\frac12+\epsilon_2)s}\\
&+t^{(1-2\epsilon_2)\nu|\lambda_+|(N_{ss}+1)-1+(\frac12-\epsilon_2)s_c-(\frac12+\epsilon_2)s}\Big).
\end{aligned}
\end{equation}
Next, we consider $E_3$. By \eqref{error inner 1}, \eqref{ss error est for any Nss}, and \eqref{R out est},  it is apparent that
\begin{equation}\label{est E3 whole}
\begin{aligned}
\|E_3(t)\|_{\dot H^1\cap\dot H^s}\leq C\Big(&t^{2\nu(K+1)-(\frac12+\nu)(s+2-s_c)}\\
&+t^{2K\epsilon_1+(\nu-\epsilon_1)|\lambda_+|-1+(\frac12+\epsilon_1)(s_c-s)}\\
&+t^{(\nu-\epsilon_1)|\lambda_+|(N_{ss}+1)+(\frac12+\epsilon_1)(s_c-2)-(\frac12+\nu)s}\\
&+t^{(1-2\epsilon_2)\nu|\lambda_+|(N_{ss}+1)+(\frac12-\epsilon_2)(s_c-2)-(\frac12+\nu)s}\\
&+t^{2\epsilon_2(N_{out}+1)-1+(\frac12-\epsilon_2)s_c-(\frac12+\epsilon_2)s}\Big).
\end{aligned}
\end{equation}
Finally, applying the estimates \eqref{inner L infinity}, \eqref{L infty est ss 1}, \eqref{L infty est ss 2}, \eqref{L infinity out}, and \eqref{est psi in ss any N}, \eqref{match est out}, we can check that
\begin{equation}\label{est E4 whole}
\begin{aligned}
\|E_4(t)\|_{\dot H^1\cap\dot H^s}
\leq C\Big(&t^{2(\nu-\epsilon_1)|\lambda_+|+4K\epsilon_1+2\epsilon_1-1+(\frac12+\epsilon_1)(s_c-s)}\\
&+t^{2(\nu-\epsilon_1)|\lambda_+|(N_{ss}+1)+(\frac12+\epsilon_1)(s_c-s-2)}\\
&+t^{4\epsilon_2(N_{out}+1)-1+2\epsilon_2 +(\frac12-\epsilon_2)s_c-(\frac12+\epsilon_2)s}\\
&+t^{2(1-2\epsilon_2)\nu|\lambda_+|(N_{ss}+1)-1+2\epsilon_2+(\frac12-\epsilon_2)s_c-(\frac12+\epsilon_2)s}\Big).
\end{aligned}
\end{equation}

Fix $0<\epsilon_2<\min\{\nu,\frac12\}$ and choose $N_{ss}$, $N_{out}$ sufficiently large. Combining
\eqref{est E1 whole}--\eqref{est E4 whole}, we obtain
\begin{equation*}\label{whole error est}
\begin{aligned}
\|R^{ap}(t)\|_{\dot H^1\cap\dot H^s(\mathbb R^d)}\leq C\Big(&t^{2\nu(K+1)-(\frac12+\nu)(s+2-s_c)}\\
&+t^{2K\epsilon_1+(\nu-\epsilon_1)|\lambda_+|-1+(\frac12+\epsilon_1)(s_c-s)}\Big),\qquad 0<t\leq T.
\end{aligned}
\end{equation*}

 \bigbreak

 \section{Exact solution}\label{sec 6}
 Throughout this section, assume $(\mathrm{H1})$ and set 
  \begin{equation*}
\nu_0(d,p)=\frac{\lfloor d/2\rfloor+1-s_c}{2|\lambda_+|}.
\end{equation*}
Fix $\nu>\nu_0$ and an integer $s$ satisfying
\begin{equation*}
\frac{d}{2}<s<s_c+\min\Big\{2\nu|\lambda_+|,\frac{4\nu K}{1+2\nu}\Big\}.
\end{equation*}
We seek a radial blowup solution $\psi\in \dot{H}^1(\mathbb{R}^{d})\cap\dot{H}^{s}(\mathbb{R}^{d})$ to \eqref{NLS} of the form 
\begin{align*}\label{def of psi}
\psi(t,x)=\lambda^{\frac{1}{p}}(t)u(\tau,y),\quad y=\lambda(t)x,\,\,\,\tau=\frac{1}{2\nu}t^{-2\nu},
\end{align*}
where $ \lambda(t)=t^{-\frac{1}{2}-\nu}$ as defined previously. We further decompose $u$ as 
\begin{equation}\label{def u ap}
\begin{aligned}
u(\tau,y)=u^{ap}(\tau,y)+f(\tau,y),\,\,\,\,u^{ap}(\tau,y)=\lambda^{-\frac{1}{p}}(t)\psi^{ap}(t,x),
\end{aligned}
\end{equation}
where the profile $\psi^{ap}$ is the approximate solution constructed in Section \ref{sec 2}-- Section \ref{sec 5}. We get the following equation for the remainder $f$
\begin{equation}\label{eq f}
i \vec f_{\tau} =-H\vec  f-\frac{1+2\nu}{4\nu\tau}i\Lambda\vec  f+\mathcal F(f) +r,\quad \vec f={f\choose \bar f},
\end{equation}
where
\begin{equation}\label{def of remainder}
\begin{split}
&H =-\Delta\sigma_3-(p+1)Q^{2p}\sigma_3-pQ^{2p}\sigma_3\sigma_1,\,\,\,\Lambda =\frac{1}{p}+y\cdot \nabla_y\\
&\mathcal F(f)=\mathcal{F}_1(f)+\mathcal{F}_2(f),\quad \mathcal{F}_i(f)=\left(\begin{array}{l}
F_i(f)\\
-\overline{F_i(f)}
\end{array}\right),\quad i=1,2,\\
&F_1(f)=\mathcal V_1(\tau)f+\mathcal V_2(\tau) \overline{f},\\
&\mathcal V_1(\tau)=(p+1)\big(|u^{ap}|^{2p}-Q^{2p}\big),\quad 
\mathcal V_2(\tau)=p\big( |u^{ap}|^{2p-2}(u^{ap})^2 - Q^{2p}\big),\\
&F_2(f)=|u^{ap} +f|^{2p} (u^{ap} +f) -|u^{ap}|^{2p} u^{ap} -
(p+1) |u^{ap}|^{2p} f -p (u^{ap})^2 |u^{ap}|^{2p-2} \overline{f},\\
&r = \left( \begin{array}{cc}
\texttt{r}\\
-\overline{\texttt{r}}
\end{array} \right), \quad \texttt{r}(\tau,y) = -t^{(\frac{1}{2}+\nu)(\frac{1}{p}+2)}R^{ap}(t,x),
\end{split}
\end{equation}
and $R^{ap}$ being the error given by the approximate solution. Note that by Proposition \ref{main prop}, for the above choice of $s$,
\begin{align}
&\|u^{ap}(\tau)\|_{W^{s,\infty}(\mathbb{R}^d)}\leq C,\label{non est}\\
&\|\mathcal{V}_j\|_{L^{\infty}}+\sum_{1\leq|\alpha|\leq s}
\big\||y|\partial_y^\alpha\mathcal V_j(\tau)\big\|_{L^\infty}\leq C\tau^{-1},\quad j=1,2\label{1 non est}\\
&\|r\|_{\dot{H}^{l}}\leq C\big(t^{2\nu(K+1)}+t^{2\nu+(\nu-\epsilon_1)(|\lambda_{+}|+l-s_c)+2K\epsilon_1}\big),\quad 1\leq l\leq s.\label{error est}
\end{align}

\medskip We first consider the linearized equation $i \vec f_{\tau} =-H\vec  f-\frac{1+2\nu}{4\nu\tau}i\Lambda\vec  f+\mathcal F_1(f)$. Denote $U(\tau,s)$ as the propagator associated to this equation. In subsection \ref{subsec 6.1}, we prove the following linear estimate.

\begin{proposition}\label{prop linear est}
{\sl Assume that $d\geq11$ and $p\in\mathbb{N}^{\ast}$ satisfies $2p+1>p_{JL}$. For every $\sigma$ satisfying
$$\frac{d}{2}<\sigma<s_c+2\nu|\lambda_{+}|,$$
and every sufficiently small $\epsilon>0$, there exist constants $C_\sigma>0$ and $\tau_0(\epsilon)> 1$ such that
\begin{equation}\label{linear est}
\|U(\tau,s)f\|_{\dot H^{1}\cap\dot H^{\sigma}} \leq C_\sigma\left(\frac{s}{\tau}\right)^{\frac{1+2\nu}{4\nu}(\sigma-s_c)+\epsilon}\|f\|_{\dot H^{1}\cap\dot H^{\sigma}},
\end{equation}
for all $\tau_0\leq\tau\leq s$ and all $f\in\dot H^1_{\mathrm{rad}}(\mathbb R^d;\mathbb C^2) \cap\dot H^\sigma_{\mathrm{rad}}(\mathbb R^d;\mathbb C^2)$.}
\end{proposition}

Before proving this, we establish the coercivity of the linearized operator $H$. This is the basic consequence of the Joseph--Lundgren condition and will also be used a lot in the proof of energy estimates.

\begin{lemma}[{\cite{PA}}]\label{lem coercivity}
{\sl Let $d\geq11$ and $p\in\mathbb{N}^{\ast}$ satisfies $2p+1>p_{JL}$. There exists a positive constant $C(p, d)$ such that
\begin{align}\label{coecivity}
\langle L_{\pm} f, f\rangle \geq C(p, d)\|\nabla f\|_{L^2}^2,\,\,\,\,C( p, d)=1-\frac{4(2p+1)c_{\infty}^{2p}}{(d-2)^2},
\end{align}
for all $f\in \dot{H}^1(\mathbb{R}^d)$.}
\end{lemma}

By scaling invariance, $L_+\Lambda Q=0$. The coercivity of $L_+$ shows that $\Lambda Q(r)$ cannot vanish at any finite $r$. Since $\Lambda Q(0)=\frac1pQ(0)>0$, we conclude that $\Lambda Q>0$.

\begin{lemma}\label{lem position Lambda Q}    
{\sl Let $d\geq11$ and $p\in\mathbb{N}^{\ast}$ satisfy $2p+1>p_{JL}$. Then the function $\Lambda Q$ is positive.}
\end{lemma}
    
\begin{proof}
By differentiating $\Delta Q+Q^{2p+1}=0$ with respect to the scalling parameter, one obtains
\begin{align*}
L_{+}\Lambda Q=0,\,\,\,\,\Lambda Q=(\frac{1}{p}+r\partial_{r})Q(r).
\end{align*}
Moreover,
\begin{equation*}
\Lambda Q(0)=\frac{1}{p}Q(0)=\frac{1}{p}>0.
\end{equation*}

\medskip We prove by contradiction. If $\Lambda Q$ is not positive, then $\Lambda Q$ vanishes at some positive radius. Let
\begin{equation*}
R=\inf \{r>0:\Lambda Q(r)=0\}.
\end{equation*}
Then $\Lambda Q(r)>0$ for $0\leq r<R$ and $\Lambda Q(R)=0$. Hence $\Lambda Q$ is a positive solution to
 \begin{equation*}
\begin{split}
\left\{
\begin{aligned}
&L_{+}\Lambda Q=0,\quad \text{in}\,\, B_R, \\
&\Lambda Q|_{\partial B_R}=0,
\end{aligned}
\right.
\end{split}
\end{equation*}
and $0$ is the first eigenvalue of $L_{+}$ on $B_R$. Since $\Lambda Q>0$ in $B_R$, the eigenvalue corresponding to $\Lambda Q$ is the principal eigenvalue. Thus $\lambda_1(R)=0$.

\medskip Fix $R'>R$. Extending $\Lambda Q$ by zero from $B_R$ to $B_{R'}$, and denote by $\tilde{\Lambda Q}$, which belongs to $H_{0}^1(B_{R'})$. By the strict domain monotonicity of the principal Dirichlet eigenvalue, 
\begin{equation*}
\lambda_{1}(R')<\lambda_{1}(R)=0.
\end{equation*}
Therefore, there exists a nonzero radial function $\psi\in H_{0}^{1}(B_{R'})$ such that
\begin{equation*}
\langle L_{+}\psi,\psi\rangle_{B_{R'}}<0.
\end{equation*}

\medskip Extending $\psi$ by zero outside $B_{R'}$, and approximate it in $H^{1}$ by test functions, we obtain a radial $g\in C_{c}^{\infty}(\mathbb{R}^d)$ satisfies
\begin{equation*}
\langle L_{+}g,g\rangle <0
\end{equation*}
which contradicts the positive of the operator $L_{+}$ in Lemma \ref{lem coercivity}.
\end{proof}

\bigskip

\subsection{Proof of Proposition \ref{prop linear est}. }\label{subsec 6.1}
Consider the equation
\begin{equation}\label{linear flow}
\left\{
\begin{aligned}
&i\partial_\tau\psi=-H\psi-\frac{1+2\nu}{4\nu\tau}i\Lambda\psi+\mathcal{F}_1(\psi),\\
&\psi|_{\tau=s}=f,
\end{aligned}
\right.
\end{equation}
where
\begin{equation*}
H=-\Delta\sigma_3-(p+1)Q^{2p}\sigma_3-pQ^{2p}\sigma_3\sigma_1.
\end{equation*}

\medskip Let $\chi\in C_{c}^{\infty}$ be a cut off function satisfy
\begin{equation*}
\begin{aligned}
\chi(r)=1,\quad \text{for}\quad 0\leq r\leq 1,\qquad \chi(r)=0,\quad \text{for}\quad r\geq 2,
\end{aligned}
\end{equation*}
and for every $A>0$, $\chi_A(r)=\chi(\frac{r}{A})$. Set 
\begin{equation*}
M=(p+1)I+p\sigma_1.
\end{equation*}

\medskip Before proving Proposition \ref{prop linear est}, we introduce several key operators used throughout this subsection. Taking into account the expression \eqref{def u ap} and \eqref{def of remainder}, we decompose $\mathcal{F}_1$ as
\begin{align}\label{decomposition F1}
\mathcal{F}_{1}(\psi)=\tilde{\mathcal{F}}_1(\psi)+\hat{\mathcal{F}}_1(\psi),
\end{align}
with
\begin{align}\label{def B(f)}
\tilde{\mathcal{F}}_1(\psi)=t^{2\nu}\mathbf{B}(\psi),\,\,\,\,\,\mathbf{B}(\psi)=\left(\begin{array}{l}
B(\psi)\\
-\overline{B(\psi)}
\end{array}\right),
\end{align}
where $B(\psi)$ is determined by the first order corrections in the inner region as follows
\begin{equation*}
\begin{aligned}
B(\psi)=&p(p+1)Q^{2p-1}\big(\chi_{1}+\overline{\chi_{1}}\big)\psi+pQ^{2p-1}\big((p+1)\chi_{1}+(p-1)\overline{\chi_{1}}\big)\bar{\psi}=2pQ^{2p-1}\chi_{1}\bar{\psi}.
\end{aligned}
\end{equation*}
The second equality follows from the fact that $\chi_1$ is purely imaginary and $\hat{\mathcal{F}}_1(\psi)$ is the remainder term of $\mathcal{F}_{1}(\psi)$ defined in \eqref{def of remainder}. Therefore, we find that
\begin{equation}\label{B self adjoint}
(\sigma_3\mathbf{B})^{\ast}=\sigma_3\mathbf{B}.
\end{equation}
To track its evolution, we denote its derivative with respect to $\tau$ by 
$\mathbf{C}(\psi)=\left(\begin{array}{l}
C(\psi)\\
-\overline{C(\psi)}
\end{array}\right)$. In view of \eqref{linear flow}, $C(\psi):=(B(\psi))_{\tau}$ is given by
\begin{equation}\label{def C(f)}
\begin{aligned}
C(\psi)=2piQ^{2p-1}\chi_{1}\Big(\big(\Delta+(p+1)Q^{2p}\big)\bar{\psi}+pQ^{2p}\psi+\frac{1+2\nu}{4\nu\tau}i\overline{\Lambda\psi}+\overline{F_1(\psi)}\Big).
\end{aligned}
\end{equation}
We also introduce the higher order derivatives of $f$ associated with the linearized operator $H$, defined by
\begin{align}\label{def of psi k}
\psi_{k}=H^{k}\psi,\quad \text{for $k \geq 0$}.
\end{align}
\medskip

\subsubsection{\bf Energy control}
We first establish the estimate $\dot{H}^{1}$. Define the energy functional 
\begin{equation}\label{def J0}
J_0(\tau)=\langle H\psi,\sigma_3\psi\rangle -\frac{1}{2\nu\tau}\Re{\langle \psi,\sigma_3 \mathbf{B}(\psi)\rangle}-\frac{\beta}{\tau}\Im{\langle \nabla \phi_{A}\nabla\psi,\sigma_3\psi\rangle},
\end{equation}
where $\beta>0$ will be fixed below and $\phi_A\in C_{\mathrm{rad}}^1(\mathbb R^d)\cap C^\infty(\mathbb R^d\setminus\{0\})$ is defined by
\begin{equation*}\label{def phi}
\phi_{A}'(r)=\chi_{A}(r)r^{1-\delta}, \qquad 0<\delta<1.
\end{equation*}

\medskip Clearly, by coercivity estimate \eqref{coecivity}, we obtain that
\begin{equation}\label{bdd energy 0}
\big(1-\frac{4(2p+1)c_{\infty}^{2p}}{(d-2)^2}\big)\|\nabla\psi\|_{L^2}^2\leq \langle H\psi,\sigma_3\psi\rangle\leq \|\nabla \psi\|_{L^2}^2.
\end{equation}
Moreover, Hardy's inequality and the compact support of $\nabla\phi_{A}$ imply
\begin{equation}\label{bdd of morawetz}
\big|\langle\nabla\phi_{A}\nabla\psi,\sigma_3\psi\rangle\big|\leq C_{A}\|\nabla\psi\|_{L^2}^2.
\end{equation}
Thus, by $\tau=\frac{1}{2\nu}t^{-2\nu}$, for any $\tau_0\gg 1$, we deduce
\begin{equation}\label{bdd correct energy 0}
c_{1}\|\nabla\psi\|_{L^2}^2\leq J_0(\tau)\leq C_1\|\nabla \psi\|_{L^2}^2,\qquad \tau\geq \tau_0.
\end{equation}

\medskip We next compute the derivative $\frac{\ud}{\ud \tau}J_0(\tau)$. The contribution of the linearized operator $H$ in \eqref{linear flow} vanishes by the self-adjointness of $\sigma_3H$ and integration by parts gives
\begin{equation}\label{H Lambda identity proof}
\begin{aligned}
&\big\langle H\Lambda\psi,\sigma_3\psi\big\rangle+\big\langle H\psi,\sigma_3\Lambda\psi\big\rangle\\
&\qquad=2(1-s_c)\big\langle H\psi,\sigma_3\psi\big\rangle+2p\big\langle Q^{2p-1}\Lambda Q\,M\psi,\psi\big\rangle,
\end{aligned}
\end{equation}
then
\begin{equation*}
\frac{\ud}{\ud\tau}\langle H\psi,\sigma_3\psi\rangle=-\frac{1+2\nu}{2\nu\tau}\Big((1-s_c)\big\langle H\psi,\sigma_3\psi\big\rangle+p\big\langle Q^{2p-1}\Lambda Q\,M\psi,\psi\big\rangle\Big)-2\Im\big\langle H\psi,\sigma_3\mathcal{F}_1(\psi)\big\rangle.
\end{equation*}
For the second term in \eqref{def J0}, by \eqref{B self adjoint} and \eqref{def C(f)},
\begin{equation*}
\big\langle \psi,\sigma_3\mathbf{C}\big\rangle=\big\langle \psi,\sigma_3\mathbf{B}(\partial_{\tau}\psi)\big\rangle=\overline{\big\langle\partial_{\tau}\psi,\sigma_3\mathbf{B}(\psi)\big\rangle}.
\end{equation*}
Taking real parts, we obtain
\begin{equation*}
\Re{\big\langle \psi,\sigma_3\mathbf{C}(\psi)\big\rangle}=\Re{\big\langle \partial_{\tau}\psi,\sigma_3\mathbf{B}(\psi)\big\rangle},
\end{equation*}
and
\begin{equation*}
\frac{\ud}{\ud\tau}\Big(\frac{1}{2\nu\tau}\Re{\langle \psi,\sigma_3 \mathbf{B}(\psi)\rangle}\Big)=-\frac{1}{2\nu\tau^2}\Re{\big\langle \psi,\sigma_3\mathbf{B}(\psi)\big\rangle}+\frac{1}{\nu\tau}\Re{\big\langle \partial_{\tau}\psi,\sigma_3\mathbf{B}(\psi)\big\rangle}.
\end{equation*}
Since $t^{2\nu}=\frac{1}{2\nu\tau}$ and $\tilde{\mathcal{F}}_1(\psi)=t^{2\nu}\mathbf{B}(\psi)$, this implies
\begin{equation*}
\begin{aligned}
\frac{\ud}{\ud\tau}\Big(\frac{1}{2\nu\tau}\Re{\langle \psi,\sigma_3 \mathbf{B}(\psi)\rangle}\Big)=&-2\Im{\langle H\psi,\sigma_3\tilde{\mathcal{F}}_1(\psi)\big\rangle}+\frac{1}{\nu\tau}\Im{\big\langle\mathcal{F}_1(\psi),\sigma_3\mathbf{B}(\psi)\big\rangle}\\
&-\frac{1}{2\nu\tau^2}\Big(\Re{\big\langle \psi,\sigma_3\mathbf{B}(\psi)\big\rangle}+\frac{1+2\nu}{2\nu}\Re{\big\langle\Lambda\psi,\sigma_3\mathbf{B}(\psi)\big\rangle}\Big).
%\Re{\big\langle (iH\psi-\frac{1+2\nu}{4\nu\tau}\Lambda\psi-i\mathcal{F}_1(f)),\sigma_3\mathbf{B}(\psi)\big\rangle}
\end{aligned}
\end{equation*}
The third term in \eqref{def J0} gives
\begin{equation*}
\begin{aligned}
\frac{\ud}{\ud\tau}\Big(\frac{\beta}{\tau}\Im{\langle \nabla \phi_{A}\nabla\psi,\sigma_3\psi\rangle}\Big)=&-\frac{\beta}{\tau}\Re{\big\langle H\psi,\sigma_3\big(\Delta\phi_A \psi+2\nabla\phi_A \nabla\psi\big)\big\rangle}\\
&+\frac{\beta}{\tau}\Re{\big\langle\mathcal F_1(\psi),\sigma_3(\Delta\phi_A+2\nabla\phi_A\cdot\nabla)\psi\big\rangle}\\
&+\frac{\beta}{\tau^2}\Big(\frac{1+2\nu}{4\nu}\Im\big\langle \Lambda \psi,\sigma_3\big(\Delta\phi_A \psi+2\nabla\phi_A\nabla\psi\big)\big\rangle\\
&\qquad\quad-\Im\big\langle\nabla\phi_A \nabla \psi,\sigma_3\psi\big\rangle\Big).
\end{aligned}
\end{equation*}

Writing $\psi=\left(\begin{array}{l}
\psi_{+}+i\psi_-\\
\psi_{+}-i\psi_-
\end{array}\right)$ and using $\mathcal{F}_1(\psi)=(2\nu\tau)^{-1}\mathbf B(\psi)+\hat{\mathcal F}_1(\psi)$, we obatin that
\begin{equation*}%\label{J1 derivative}
\frac{\ud}{\ud \tau}J_0(\tau)=E_1+E_2+E_3+E_4,
\end{equation*}
where 
\begin{equation}\label{J1 derivative}
\begin{aligned}
E_1=&\frac{(1+2\nu)(s_c-1)}{2\nu\tau}\big\langle H\psi,\sigma_3\psi\big\rangle,\\
E_2=&\frac{1}{\tau}\Big(4\beta K_{A}(\psi)-\frac{(1+2\nu)p}{2\nu}\langle Q^{2p-1}\Lambda Q M\psi,\psi\rangle\Big),\\
E_3=&-2\Im{\big\langle H\psi,\sigma_3\hat{\mathcal{F}}_1(\psi)\big\rangle}-\frac{\beta}{\tau}\Re\big\langle\hat{\mathcal F}_1(\psi),
\sigma_3(\Delta\phi_A+2\nabla\phi_A\cdot\nabla)\psi\big\rangle,\\
E_4=&-\frac{1}{\nu\tau}\Im{\big\langle\mathcal{F}_1(\psi),\sigma_3\mathbf{B}(\psi)\big\rangle}\\
&-\frac{\beta}{\tau^2}\Big(\frac{1+2\nu}{4\nu}\Im\big\langle \Lambda \psi,\sigma_3\big(\Delta\phi_A+2\nabla\phi_A\cdot \nabla\big)\psi\big\rangle-\Im\big\langle\nabla\phi_A \nabla \psi,\sigma_3\psi\big\rangle\Big)\\
&+\frac{1}{2\nu\tau^2}\Big(\Re{\big\langle \psi,\sigma_3\mathbf{B}(\psi)\big\rangle}+\frac{1+2\nu}{2\nu}\Re{\big\langle\Lambda\psi,\sigma_3\mathbf{B}(\psi)\big\rangle}\\
&\qquad\qquad-\beta\Re\big\langle\mathbf B(\psi),\sigma_3(\Delta\phi_A+2\nabla\phi_A\cdot\nabla)\psi\big\rangle\Big),
\end{aligned}
\end{equation}
and
\begin{equation}\label{def KA}
K_{A}(\psi)=\Big\langle L_{+}\psi_{+},\big(\frac{\Delta \phi_A}{2}+\nabla\phi_A\cdot\nabla\Big)\psi_{+}\big\rangle+\Big\langle L_{-}\psi_{-},\big(\frac{\Delta \phi_A}{2}+\nabla\phi_A\cdot\nabla\Big)\psi_{-}\big\rangle.
\end{equation}

\medskip  It follows from the asymptotic expansion of $Q$ that
\begin{equation*}
\Lambda Q(r)=c_{1,0}\lambda_{+}r^{-\frac{1}{p}+\lambda_{+}}+o(r^{-\frac{1}{p}+\lambda_{+}}),\quad \lambda_{+}<-2,
\end{equation*}
and $Q$, $\Lambda Q$ are positive, there exists some $0<\delta<1$ that
\begin{equation*}
0\leq Q^{2p-1}\Lambda Q\leq \frac{C}{|y|^2(1+|y|^{\delta})}.
\end{equation*}
Fix such a $\delta$ in the definition of $\phi_{A}$ and let $A\gg 1$. Choosing $\beta>0$ sufficiently large, independently of $A$,
the \eqref{clam KA bdd} implies
\begin{equation}\label{J1 ques problem}
4\beta K_{A}(\psi)-\frac{(1+ 2\nu)p}{2\nu}\big\langle Q^{2p-1}\Lambda Q M\psi,\psi\big\rangle \geq -\frac{C_{\beta}}{A^{\delta}}\big\langle H\psi,\sigma_3\psi\big\rangle,
\end{equation}
then
\begin{equation}\label{bdd E2}
E_2 \geq -\frac{C_{\beta}A^{-\delta}}{\tau}\big\langle H\psi,\sigma_3\psi\big\rangle.
\end{equation}

Substituting \eqref{def of remainder} and \eqref{def B(f)} into \eqref{decomposition F1}, we obtain
\begin{equation*}
\begin{aligned}
\hat{F}_1(\psi)=&(p+1)\Big(|u^{ap}|^{2p}-Q^{2p}-p\,t^{2\nu}Q^{2p-1}\big(\chi_{1}+\bar{\chi}_{1}\big)\Big)\psi\\
&+p\Big(|u^{ap}|^{2p-2}(u^{ap})^2-Q^{2p}-t^{2\nu}Q^{2p-1}\big((p+1)\chi_{1}+(p-1)\bar{\chi}_{1}\big)\Big)\overline{\psi},
\end{aligned}
\end{equation*}
where $\chi_{ss}^{ap}$ and $\chi_{out}^{ap}$ the profiles defined in Sections \ref{sec 3} and \ref{sec 4}. Taking \eqref{def rad psi ap} and \eqref{def u ap} into account and using the fact that $\chi_1$ is purely imaginary and 
$$\Theta(t^{\nu-\epsilon_1}y)\Theta(t^{\nu+\epsilon_2}y)=\Theta(t^{\nu-\epsilon_1}y),$$
 we find
\begin{equation}\label{expression hat F1}
\begin{aligned}
\hat{F}_1(\psi)=&(p+1)\Big[pQ^{2p-1}(y)\Big(\Theta(t^{\nu-\epsilon_1}y)\sum\limits_{2\leq k\leq K}t^{2\nu k}\big(\chi_{k}(y)+\bar{\chi}_{k}(y)\big)\\
&\qquad\qquad+\big(1-\Theta(t^{\nu-\epsilon_1}y)\big)\Theta(t^{\nu+\epsilon_2}y)\big(\chi_{ss}^{ap}(t,y)+\overline{\chi_{ss}^{ap}(t,y)}\big)\\
&\qquad\qquad+\big(1-\Theta(t^{\nu+\epsilon_2}y)\big)\big(\chi_{out}^{ap}(t,y)+\overline{\chi_{out}^{ap}(t,y)}\big)\Big)+N_1(\chi^2)\Big]\psi\\
&+p\Big[Q^{2p-1}(y)\Big(%2\big(\Theta(t^{\nu-\epsilon_1}y)-1\big)t^{2\nu}\chi_{1}(y)\\
\Theta(t^{\nu-\epsilon_1}y)\sum\limits_{2\leq k\leq K}t^{2\nu k}\big((p+1)\chi_{k}(y)+(p-1)\overline{\chi_{k}(y)}\big)\\
&\qquad+\big(1-\Theta(t^{\nu-\epsilon_1}y)\big)\Theta(t^{\nu+\epsilon_2}y)\big((p+1)\chi_{ss}^{ap}(t,y)\\
&\qquad\qquad\qquad\qquad\qquad\qquad+(p-1)\overline{\chi_{ss}^{ap}(t,y)}-2t^{2\nu}\chi_1(y)\big)\\
&\qquad+\big(1-\Theta(t^{\nu+\epsilon_2}y)\big)\big((p+1)\chi_{out}^{ap}(t,y)\\
&\qquad\qquad\qquad\qquad\qquad+(p-1)\overline{\chi_{out}^{ap}(t,y)}-2t^{2\nu}\chi_1(y)\big)\Big)+N_2(\chi^2)\Big]\overline{\psi},
\end{aligned}
\end{equation}
where $N_i(\chi^2)$, $i=1,2$, are second order terms with respect to $u^{ap}-Q$ and the behaviors of $\chi_k(y)$, $1\leq k\leq K$ are given in Lemma \ref{lem: inner region}. According to the expansion \eqref{inner expansion infinity k1},
\begin{align}\label{est chi 1}
\big|\nabla_y^q(Q^{2p-1}\chi_1)\big|\leq C_q\langle y\rangle^{-|\lambda_+|-q}, \qquad q\geq0,
\end{align}
together with \eqref{inner L infinity}, \eqref{L infty est ss 1}, \eqref{L infty est ss 2} and \eqref{L infinity out},  and using $\lambda_{+}<-2$, we obtain
\begin{equation}\label{hat F1 est}
\begin{aligned}
&\|\mathcal{V}_1\|_{L^\infty}+\big\|\mathcal{V}_2-2pt^{2\nu}Q^{2p-1}\chi_1\big\|_{L^\infty}\\
&\quad+\big\||y|\nabla_y\mathcal{V}_1\big\|_{L^\infty}+\big\||y|\nabla_y \big(\mathcal{V}_2-2pt^{2\nu}Q^{2p-1}\chi_1\big) \big\|_{L^\infty}\\
&\qquad\leq C(1+|\ln t|)^L\Big(t^{4\nu}+t^{1+2\nu}+t^{(\nu-\epsilon_1)(2+|\lambda_+|(N_{ss}+1))}\Big).
\end{aligned}
\end{equation}
Choose  $N_{ss}$ sufficiently large, integrate by parts, and use Hardy's inequality to give
\begin{equation}\label{main in E3}
\Big|\big\langle H\psi, \sigma_3\hat{\mathcal{F}}_1(\psi)\big\rangle\Big|\leq C\big(\tau^{-1-\frac{1}{2\nu}}+\tau^{-2}\big)\|\nabla \psi\|_{L^2}^2.
\end{equation}
Since $\nabla\phi_A$ and $\Delta\phi_A$ are supported in $|y|\leq2A$, Hardy's inequality and \eqref{hat F1 est} implies
\begin{equation*}
\begin{aligned}
&\Big| \Re\big\langle\hat{\mathcal F}_1(\psi), \sigma_3(\Delta\phi_A+2\nabla\phi_A\cdot\nabla)\psi \big\rangle\Big|\leq C_A\big(\tau^{-2}+\tau^{-1-\frac1{2\nu}}\big)\|\nabla\psi\|_{L^2}^2.
\end{aligned}
\end{equation*}
Combining this with \eqref{main in E3}, we obtain
\begin{equation}\label{bdd E3}
\big|E_3\big|\leq C\big(\tau^{-1-\frac{1}{2\nu}}+\tau^{-2}\big)\|\nabla \psi\|_{L^2}^2.
\end{equation}

Finally, by the expression of $\mathcal{F}_1(\psi)$, $\mathbf{B}(\psi)$ and \eqref{1 non est}, 
\begin{equation*}
\begin{aligned}
\Big|\Im{\big\langle\mathcal{F}_1(\psi),\sigma_3\mathbf{B}(\psi)\big\rangle}\Big|\leq& C\big(\|\mathcal{V}_1\|_{L^{\infty}}+\|\mathcal{V}_2\|_{L^{\infty}}\big)\int_{\mathbb{R}^d} Q^{2p-1}|\chi_1||\psi|^2\ud y\\
\leq&\frac{C}{\tau}\|\nabla\psi\|_{L^2}^2.
\end{aligned}
\end{equation*}
The last inequality follows from Hardy's inequality and the fact that $|\lambda_+|>2$. For the additional Morawetz contribution, we similarly have
\begin{equation*}
\Big|\Re\big\langle\mathbf B(\psi),\sigma_3(\Delta\phi_A+2\nabla\phi_A\cdot\nabla)\psi\big\rangle\Big|\leq C_A\|\nabla\psi\|_{L^2}^2.
\end{equation*}
Thus
\begin{equation}\label{bdd E4}
|E4|\leq \frac{C}{\tau^2}\|\nabla \psi\|_{L^2}^2.
\end{equation}

\medskip Therefore, the estimates \eqref{bdd energy 0}, \eqref{bdd E2}, \eqref{bdd E3} and \eqref{bdd E4} imply that
\begin{equation*}\label{derivative of J1}
\begin{aligned}
\frac{\ud}{\ud\tau}J_0(\tau)\geq \frac{1}{\tau}\big(\frac{(1+2\nu)(s_c-1)}{2\nu}-\frac{C_{\beta}}{A^{\delta}}-C_{A,\beta}\tau^{-1}-C_{A,\beta}\tau^{-\frac1{2\nu}}\big)\langle H\psi,\sigma_3\psi\rangle,
\end{aligned}
\end{equation*}
For every sufficiently small $\epsilon>0$, taking $A$ sufficiently large and $\tau_0>0$ sufficiently large, therefore \eqref{bdd correct energy 0} implies
\begin{equation*}
\frac{\ud}{\ud \tau}J_0(\tau)\geq \frac{1}{\tau}\big(\frac{(1+2\nu)(s_c-1)}{2\nu}-\epsilon\big)J_0(\tau),\qquad \tau\geq \tau_0.
\end{equation*}
Integrating this inequality and using \eqref{bdd correct energy 0}, \eqref{linear flow} and $\psi(s)=f$, we conclude
\begin{equation}\label{H1 bdd}
\|U(\tau,s)f\|_{\dot{H}^{1}}\leq C\left(\frac{\tau}{s}\right)^{\frac{(1+2\nu)}{4\nu}(s_c-1)-\epsilon}\|f\|_{\dot{H}^1},\qquad \tau_0\leq \tau\leq s.
\end{equation}

\bigbreak

\subsubsection{\bf Control of the higher regularity}
\bigskip We next control the higher regularity. A direct computation gives
\begin{equation*}\label{commutator H Lambda}
[H,\Lambda]=2H+2pQ^{2p-1}\Lambda Q\,\sigma_3M.
\end{equation*}
Consequently,
\begin{equation*}
[H^k,\Lambda]=2kH^k+2p\sum_{\ell=0}^{k-1}H^\ell\big(Q^{2p-1}\Lambda Q\,\sigma_3M\big)H^{k-1-\ell}.
\end{equation*}

Since $\mathbf{B}=2pQ^{2p-1}\chi_1\sigma_1$ and $\sigma_3\sigma_1=-\sigma_1\sigma_3$, then
\begin{equation}\label{note B1}
(H\mathbf{B}+\mathbf{B}H)\psi=-2p\sigma_3\sigma_1\Big(2\nabla(Q^{2p-1}\chi_1)\cdot\nabla\psi+\Delta(Q^{2p-1}\chi_1)\psi\Big).
\end{equation}
Meanwhile, for every $k\geq 1$,
\begin{equation}\label{note B2}
H^k\mathbf{B}-(-1)^k\mathbf{B}H^k=\sum_{\ell=0}^{k-1} (-1)^\ell H^{k-1-\ell}(H\mathbf{B}+\mathbf{B}H)H^\ell
\end{equation}
has differential order at most $2k-1$. Thus, according to \eqref{decomposition F1}, \eqref{note B1} and \eqref{note B2}, we decompose $H^{k}\mathcal{F}_1(\psi)$ into three parts
\begin{equation*}
H^{k}\mathcal{F}_1(\psi)=\frac{(-1)^k}{2\nu\tau}\mathbf{B}(\psi_k)+\frac{1}{2\nu\tau}\Big(H^k\mathbf{B}(\psi)-(-1)^k\mathbf{B}(\psi_k)\Big)+H^{k}\hat{\mathcal{F}}_1(\psi).
\end{equation*}

\medskip Applying $H^k$ to \eqref{linear flow}, we obtain the equation for $H^{k}\psi$ is
\begin{equation}\label{equation psik}
\begin{aligned}
i\partial_\tau\psi_k=&-H\psi_k-\frac{1+2\nu}{4\nu\tau}i(\Lambda+2k)\psi_k+\frac{(-1)^k}{2\nu\tau}\mathbf{B}(\psi_k)\\
&-\frac{p(1+2\nu)}{2\nu\tau}i\sum_{\ell=0}^{k-1}H^\ell\big(Q^{2p-1}\Lambda Q\,\sigma_3M\psi_{k-1-\ell}\big)+H^{k}\hat{\mathcal{F}}_1(\psi)\\
&-\frac p{\nu\tau}\sum_{\ell=0}^{k-1}(-1)^{k-1-\ell}H^\ell\Big[\sigma_3\sigma_1\Big(2\nabla(Q^{2p-1}\chi_1)\cdot\nabla\psi_{k-1-\ell}
+\Delta(Q^{2p-1}\chi_1)\psi_{k-1-\ell}\Big)\Big],
\end{aligned}
\end{equation}
where $\psi_k$ have been defined in \eqref{def of psi k}.

\medskip Define the energy functional with the same Morawetz correction,
\begin{equation}\label{def energy k}
J_{k}(\tau)=\tau^{\frac{(1+2\nu)k}{\nu}}\Big(\langle H\psi_{k},\sigma_3\psi_k\rangle -\frac{\beta}{\tau}\Im\langle \nabla\phi_A\,\nabla\psi_k,\sigma_3\psi_k\rangle+C_k J_0(\tau)\Big),
\end{equation}
where $C_k>0$ will be chosen below. Applying \eqref{coecivity} and \eqref{bdd of morawetz} to $\psi$ gives
\begin{equation}\label{bdd energy k}
\big(1-\frac{4(2p+1)c_{\infty}^{2p}}{(d-2)^2}\big)\|\nabla\psi_k\|_{L^2}^2\leq \langle H\psi_k,\sigma_3\psi_k\rangle\leq \|\nabla \psi_k\|_{L^2}^2,
\end{equation}
and 
\begin{equation}\label{bdd of morawetz k}
\big|\langle\nabla \phi_A\,\nabla\psi_k,\psi_k\rangle\big|\leq C_A\|\nabla\psi_k\|_{L^2}^2.
\end{equation}
By Leibniz's rule and interpolation,
\begin{equation}\label{interpolation}
\|\psi\|_{\dot{H}^{m}}\leq C\|\psi\|_{\dot{H}^1}^{1-\theta}\|\psi\|_{\dot{H}^{2k+1}}^{\theta},\quad \theta=\frac{m-1}{2k},\quad 1\leq m\leq 2k+1.
\end{equation}
Meanwhile, since $H=-\Delta\sigma_3-Q^{2p}\sigma_3M$, and for every $m\geq 0$,
\begin{equation*}
\big|\partial_r^m Q^{2p}(r)\big|\leq \frac{C_m}{(1+r)^{2+m}},
\end{equation*}
then using Hardy's inequality and induction on $k$ gives
\begin{equation}\label{bdd for phi k}
\|H^k \psi\|_{\dot{H}^1}^2+\|\psi\|_{\dot{H}^1}^2\sim\|\psi\|_{\dot{H}^1\cap\dot{H}^{2k+1}}^2.
\end{equation}
Taking \eqref{bdd correct energy 0}, \eqref{bdd energy k}, \eqref{bdd of morawetz k} and \eqref{bdd for phi k} into \eqref{def energy k}, we deduce that 
\begin{equation}\label{bdd correct energy k}
J_{k}(\tau)\sim\tau^{\frac{(1+2\nu)k}{\nu}}\|\psi\|_{\dot{H}^1\cap\dot{H}^{2k+1}}^2.
\end{equation}

\medskip We next compute the derivative of $J_k$. Differentiating the first term in \eqref{def energy k}, by \eqref{H Lambda identity proof} and \eqref{equation psik}, we obtain
\begin{equation}\label{derivative energy k}
\begin{aligned}
\frac{\ud}{\ud\tau}\langle H\psi_k,\sigma_3\psi_k\rangle=\frac{1+2\nu}{2\nu\tau}\Big(&(s_c-2k-1)\big\langle H\psi_k,\sigma_3\psi_k\big\rangle-p\big\langle Q^{2p-1}\Lambda Q M\psi_k,\psi_k\big\rangle\\
&-\frac{2(-1)^k}{1+2\nu}\Im{\langle H\psi_k,\sigma_3\mathbf B(\psi_k)\rangle}+\mathcal{R}_k(\tau)\Big)-2\Im{\big\langle
H\psi_k,\sigma_3H^{k}\hat{\mathcal F}_1(\psi)\big\rangle},
\end{aligned}
\end{equation}
where 
\begin{align*}
\mathcal R_k(\tau)
=&-2p\sum_{\ell=0}^{k-1}\Re{\Big\langle H^{\ell+1}\big(Q^{2p-1}\Lambda Q\,\sigma_3M\psi_{k-1-\ell}\big),\sigma_3\psi_{k}\Big\rangle}\\
&-\frac{4p}{1+2\nu}\sum_{\ell=0}^{k-1}(-1)^{k-1-\ell}\Im \Big\langle H^{\ell+1}\Big(\sigma_3\sigma_1\big(2\nabla(Q^{2p-1}\chi_1)\cdot\nabla+\Delta(Q^{2p-1}\chi_1)\big)\psi_{k-1-\ell}\Big),\sigma_3\psi_k\Big\rangle.
\end{align*}
Applying the same computation of Morawetz correction to $\psi_k$, we have
\begin{equation}\label{derivative morawetz k}
\begin{aligned}
&\frac{\ud}{\ud\tau}\Big(\frac{\beta}{\tau}\Im\big\langle\nabla\phi_A\,\nabla\psi_k,\sigma_3\psi_k \big\rangle\Big)\\
=&-\frac{4\beta}{\tau}K_A(\psi_k)+\frac{\beta}{\tau}\Re{\big\langle H^k\hat{\mathcal F}_1(\psi),\sigma_3(\Delta\phi_A+2\nabla\phi_A\cdot\nabla)\psi_k\big\rangle}-\frac{\beta}{\tau^2}\mathcal{R}_{\phi_A}(\tau),
\end{aligned}
\end{equation}
where

\begin{equation*}
\begin{aligned}
\mathcal{R}_{\phi_A}(\tau)=&\Im\big \langle \nabla\phi_A\cdot\nabla\psi_k,\sigma_3\psi_k\big\rangle-\frac{1+2\nu}{4\nu}\Im\Big\langle (\Lambda +2k)\psi_k,\sigma_3\big(\Delta\phi_A+2\nabla\phi_A\cdot\nabla\big)\psi_k\Big\rangle\\
&-\frac{(-1)^k}{2\nu}\Re{\big\langle\mathbf{B}(\psi_k),\sigma_3(\Delta\phi_A+2\nabla\phi_A\cdot\nabla)\psi_k\big\rangle}\\
&-\frac{(1+2\nu)p}{2\nu}\sum_{\ell=0}^{k-1}\Im{\Big\langle H^{\ell}\big(Q^{2p-1}\Lambda Q\,\sigma_3M\psi_{k-1-\ell}\big),\sigma_3(\Delta\phi_A+2\nabla\phi_A\cdot\nabla)\psi_k\Big\rangle}\\
&+\frac p\nu
\sum_{\ell=0}^{k-1}(-1)^{k-1-\ell}\Re\Big\langle H^\ell\Big(\sigma_3\sigma_1\big(2\nabla(Q^{2p-1}\chi_1)\cdot\nabla+\Delta(Q^{2p-1}\chi_1)\big)\psi_{k-1-\ell}\Big),\\
&\qquad\qquad\qquad\qquad\qquad\qquad\qquad\qquad\qquad\qquad\qquad\sigma_3(\Delta\phi_A+2\nabla\phi_A\cdot\nabla)\psi_k\Big\rangle.
\end{aligned}
\end{equation*}

\medskip Thus, using \eqref{J1 derivative}, \eqref{derivative energy k} and \eqref{derivative morawetz k}, the derivative of $J_k$ can be separated into four parts,
\begin{equation}\label{derivative of Jk}
\begin{aligned}
\frac{\ud}{\ud\tau}J_k(\tau)=\tau^{\frac{(1+2\nu)k}{\nu}}\Big(E_{k,1}+E_{k,2}+E_{k,3}+E_{k,4}\Big),
\end{aligned}
\end{equation}
where
\begin{equation*}
\begin{aligned}
E_{k,1}=&\frac{1+2\nu}{2\nu\tau}\Big((s_c-1)\langle H\psi_k,\sigma_3\psi_k\rangle+2kC_kJ_0(\tau)\Big)+C_k E_1,\\
E_{k,2}=&\frac{1}{\tau}\Big(4\beta K_A(\psi_k)-\frac{(1+2\nu)p}{2\nu}\big\langle Q^{2p-1}\Lambda Q\,M\psi_k,\psi_k\big\rangle-\frac{(-1)^k}{\nu}\Im{\big\langle H\psi_k,\sigma_3\mathbf{B}(\psi_k)}\big\rangle\\
&\qquad+\frac{1+2\nu}{2\nu}\mathcal{R}_k(\tau)\Big)+C_k E_2,\\
E_{k,3}=&-2\Im{\big\langle H\psi_k,\sigma_3 H^{k}\hat{\mathcal F}_1(\psi)\big\rangle}-\frac{\beta}{\tau}\Re{\big\langle H^{k}\hat{\mathcal{F}}_1(\psi),\sigma_3\big(\Delta\phi_A+2\nabla\phi_A\cdot\nabla\big)\psi_{k} \big\rangle}+C_k E_3,\\
E_{k,4}=&\frac{\beta}{\tau^2}\Big(-\frac{(1+2\nu)k}{\nu}\Im{\big\langle\nabla\phi_A\nabla\psi_k,\sigma_3\psi_k\big\rangle}+\mathcal{R}_{\phi_A}(\tau)\Big)+C_k E_4.
\end{aligned}
\end{equation*}

\medskip We first estimate $E_{k,2}$. By \eqref{J1 ques problem}, for fixed $\beta>0$,
\begin{equation}\label{Jk morawetz est 1}
4\beta K_{A}(\psi_k)-\frac{(1+2\nu)p}{2\nu}\big\langle Q^{2p-1}\Lambda Q M\psi_k,\psi_k\big\rangle \geq -\frac{C_{\beta}}{A^{\delta}}\langle H\psi_k,\sigma_3\psi_k\rangle.
\end{equation}
According to the estimate \eqref{clam KA bdd} for $\psi_k$ and 
\begin{equation*}
\Big|\Im{\big\langle H\psi_k,\sigma_3\mathbf{B}(\psi_k)\big\rangle}\Big|\leq C\int_{|y|\leq A}\frac{1}{1+|y|^{\delta}}\big(|\nabla\psi_k|^2+\frac{|\psi_k|^2}{|y|^2}\big)\ud y+\frac{C}{A^{|\lambda_{+}|}}\big\langle H\psi_k,\sigma_3\psi_k\big\rangle,
\end{equation*}
we obtain that
\begin{equation}\label{Jk morawetz est 2}
\Big|\Im{\big\langle H\psi_k,\sigma_3\mathbf{B}(\psi_k)\big\rangle}\Big|\leq CK_{A}(\psi_k)+\frac{C}{A^{\delta}}\big\langle H\psi_k,\sigma_3\psi_k\big\rangle.
\end{equation}
Moreover,  \eqref{clam KA bdd} and \eqref{bdd energy k} gives
\begin{equation*}
K_A(\psi_k)\geq -\frac{C}{A^{\delta}}\|\nabla\psi_k\|_{L^2}^2\geq -\frac{C}{A^{\delta}}\big\langle H\psi_k,\sigma_3\psi_k\big\rangle,
\end{equation*}
where the constant $C$ changes from term to term. Then, increasing $\beta$ such that $\beta\geq \frac{C}{\nu}$, the estimates \eqref{Jk morawetz est 1} and \eqref{Jk morawetz est 2} imply
\begin{equation}\label{Jk morawetz est}
\begin{aligned}
4\beta K_A(\psi_k)-\frac{(1+2\nu)p}{2\nu}\big\langle Q^{2p-1}\Lambda Q\,M\psi_k,\psi_k\big\rangle-\frac{(-1)^k}{\nu}&\Im{\big\langle H\psi_k,\sigma_3\mathbf{B}(\psi_k)}\big\rangle\\
\geq& -\frac{C_{\beta}}{A^{\delta}}\big\langle H\psi_k,\sigma_3\psi_k\big\rangle.
\end{aligned}
\end{equation}
For the $\mathcal{R}_k$ term, the asymptotic expansion of $Q$ implies that for every $m\geq 0$
\begin{equation}\label{est chi 2}
\big|\partial_{r}^{m}Q^{2p}(r)\big|+\big|\partial_{r}^{m}(Q^{2p-1}\Lambda Q)(r)\big|\leq \frac{C_m}{(1+r)^{2+m}}.
\end{equation}
Using this bound together with \eqref{bdd for phi k}, Hardy's inequality and \eqref{interpolation},
one obtains, for $\eta>0$, 
\begin{equation}\label{bdd error energy k}
\big|\mathcal{R}_k\big|\leq \eta\langle H\psi_k,\sigma_3\psi_k\rangle+C_{k,\eta}J_0(\tau),
\end{equation}
which follows from Young's inequality. After fixing $\eta$ and $A$, choose $C_k$ sufficiently
large that
\begin{equation*}
\frac{1+2\nu}{2\nu}C_{k,\eta}\leq \frac{C_k}{A^\delta}.
\end{equation*}
Therefore, using \eqref{Jk morawetz est}, \eqref{bdd error energy k} and \eqref{bdd E2} in the derivative of $J_0$, we deduce that
\begin{equation}\label{bdd Ek2}
E_{k,2}\geq-\frac{1}{\tau}\big(\frac{C_{\beta}}{A^{\delta}}+\frac{(1+2\nu)\eta}{2\nu}\big)\big\langle H\psi_k,\sigma_3\psi_k\big\rangle-\frac{1}{\tau}\frac{C_k C_{\beta,\eta}}{A^{\delta}}J_0(\tau).
\end{equation}

\medskip We next estimate $E_{k,3}$. Integration by parts and Hardy's inequality give
\begin{equation}\label{Ek3 1}
\Big|\big\langle H\psi_k,\sigma_3H^k\hat{\mathcal F}_1(\psi)\big\rangle\Big|\leq C\|\nabla\psi_k\|_2 \|H^k\hat{\mathcal F}_1(\psi)\|_{\dot H^1}.
\end{equation}
Since the function $\Delta\phi_A+2\nabla\phi_A\cdot\nabla$ is supported in $|y|\leq2A$, we have
\begin{equation}\label{Ek3 2}
\Big|\big\langle H^k\hat{\mathcal F}_1(\psi), \sigma_3(\Delta\phi_A+2\nabla\phi_A\cdot\nabla)\psi_k\big\rangle\Big|\leq
C_{A}\|H^k\hat{\mathcal F}_1(\psi)\|_{\dot H^1}\|\nabla\psi_k\|_{L^2},
\end{equation}
by the Hardy inequality. From the expression \eqref{expression hat F1} of $\hat{\mathcal{F}}_1(\psi)$ and \eqref{hat F1 est}, we deduce that for $N_{ss}$ sufficiently large,
\begin{equation*}
\|\hat{\mathcal{F}}_1(\psi)\|_{\dot{H}^1}\leq C\big(\tau^{-\frac{3}{2}}+\tau^{-1-\frac{1}{4\nu}}\big)\|\nabla\psi\|_{L^2}.
\end{equation*}
We now prove the bound of $\|H^k\hat{\mathcal F}_1(\psi)\|_{\dot H^1}$ for $k\geq 1$ with $2k+1\leq\sigma+1$. For $1\leq j\leq 2k+1$, Leibniz's rule gives
\begin{equation*}
\big\|\nabla^{j}(\mathcal{V}_1\psi)\big\|_{L^2}\leq C_j\|\mathcal{V}_1\|_{L^{\infty}}\|\nabla^j \psi\|_{L^2}+C_j\sum_{q=1}^{j}\big\||y|\nabla^{q}\mathcal{V}_1\big\|_{L^{\infty}}\big\||y|^{-1}\nabla^{j-q}\psi\big\|_{L^2}.
\end{equation*}
Using the differentiated profile expansions in the argument
leading to \eqref{hat F1 est}, we obtain
\begin{equation*}
\begin{aligned}
&\|\mathcal V_1\|_{L^\infty}+\sum_{q=1}^{2k+1}\||y|\nabla_y^q\mathcal V_1\|_{L^\infty}\leq
C_k(1+|\ln t|)^L\Big(t^{4\nu}+t^{1+2\nu}+t^{(\nu-\epsilon_1)(2+|\lambda_+|(N_{ss}+1))}\Big).
\end{aligned}
\end{equation*}
Here $\epsilon_2<\nu$ prevents an additional time loss when differentiating the oscillatory terms in the self-similar region, while
$q-1-2s_c-2\nu|\lambda_+|\leq\sigma-2s_c-2\nu|\lambda_+|<-s_c$ ensures the required spatial decay of the differentiated oscillatory terms in the outer region. Choosing $N_{ss}$ sufficiently large that $(\nu-\epsilon_1)\bigl(2+|\lambda_+|(N_{ss}+1)\bigr)>\max\{4\nu,1+2\nu\}$, we obtain
\begin{equation*}
\|\mathcal V_1\|_{L^\infty}+\sum_{q=1}^{2k+1}\||y|\nabla_y^q\mathcal V_1\|_{L^\infty}
\leq C_k\big(\tau^{-\frac32}+\tau^{-1-\frac1{4\nu}}\big).
\end{equation*}
Combining this bound with the Hardy inequality and \eqref{hat F1 est}, we deduce that
\begin{equation*}
\big\|\nabla^j(\mathcal{V}_1\psi)\big\|_{L^2}\leq C\big(\tau^{-\frac{3}{2}}+\tau^{-1-\frac{1}{4\nu}}\big)\sum_{q=1}^{j}\|\nabla^{q}\psi\|_{L^2}.
\end{equation*}
An analogous calculation, using \eqref{expression hat F1} and \eqref{est chi 1}, yields the same bound for $\big(\mathcal V_2-2pt^{2\nu}Q^{2p-1}\chi_1\big)\bar\psi$. Then, by interpolation and the Hardy inequality, we obtain that
\begin{equation}\label{Ek3 21}
\|H^k\hat{\mathcal F}_1(\psi)\|_{\dot H^1}\leq C\sum_{j=1}^{2k+1}\big\|\nabla^{j}\hat{\mathcal{F}}_1(\psi)\big\|_{L^2}\leq C\big(\tau^{-\frac{3}{2}}+\tau^{-1-\frac{1}{4\nu}}\big)\big(\|\nabla\psi\|_{L^2}+\|\nabla^{2k+1}\psi\|_{L^2}\big).
\end{equation}
Therefore, substituting \eqref{Ek3 21} into \eqref{Ek3 2} and using \eqref{Ek3 1}, the estimate \eqref{bdd E3} in the bound for $J_0$, and the coercivity bound \eqref{bdd energy k}, we prove that
\begin{equation}\label{bdd Ek3}
\Big|E_{k,3}\Big|\leq C_{k,A,\beta}\big(\tau^{-\frac{3}{2}}+\tau^{-1-\frac{1}{4\nu}}\big)\big(\big\langle H\psi_k,\sigma_3\psi_k\big\rangle+C_k J_0(\tau)\big).
\end{equation}

By the compact support of $\nabla\phi_A$, Hardy's inequality, and \eqref{est chi 2}, \eqref{bdd for phi k}, we obtain
\begin{equation*}
\big|\Im\langle\nabla\phi_A\cdot\nabla\psi_k,\sigma_3\psi_k\rangle\big|+|\mathcal R_{\phi_A}(\tau)|\leq C_{k,A}\big(
\|\nabla\psi_k\|_{L^2}^2+\|\nabla\psi\|_{L^2}^2\big).
\end{equation*}
Here the two sums in $\mathcal R_{\phi_A}$ contain derivatives of $\psi$ of order at most $2k-1$. The pairings with $(\Delta\phi_A+2\nabla\phi_A\cdot\nabla)\psi_k$ are controlled by interpolation and Young's inequality. Together with \eqref{bdd E4}, this gives
\begin{equation}\label{bdd Ek4}
\big|E_{k,4}\big|\leq \frac{C_{k,A,\beta}}{\tau^2} \Big(\langle H\psi_k,\sigma_3\psi_k\rangle+C_kJ_0(\tau) \Big).
\end{equation}

\medskip We substitute \eqref{bdd Ek2}, \eqref{bdd Ek3} and \eqref{bdd Ek4} into \eqref{derivative of Jk}. Choose $\eta$ sufficiently small, then fix $A$ sufficiently large and choose $C_k$ to satisfy the stated conditions. Finally, increase $\tau_0$ if necessary. We obtain
\begin{equation*}
\begin{aligned}
\frac{\ud}{\ud\tau}J_k(\tau)\geq\tau^{\frac{(1+2\nu)k}{\nu}-1}\Big[
&\big(\frac{(1+2\nu)(s_c-1)}{2\nu}-\epsilon\big)\big\langle H\psi_k,\sigma_3\psi_k\big\rangle\\
&+C_k\big(\frac{(1+2\nu)(s_c-1+2k)}{2\nu}-\varepsilon\big)J_0-C_{k,\epsilon}\big\langle H\psi,\sigma_3\psi\big\rangle\Big],\quad \tau\geq \tau_0.
\end{aligned}
\end{equation*}
For $k\geq1$, choose $C_k$ sufficiently large that
\begin{equation*}
C_k\frac{(1+2\nu)k}{\nu}\geq2C_{k,\varepsilon}.
\end{equation*}
Since $\langle H\psi,\sigma_3\psi\rangle\leq2J_0$, this estimate implies
\begin{equation*}
\begin{aligned}
\frac{\ud}{\ud\tau}J_k(\tau)\geq\tau^{\frac{(1+2\nu)k}{\nu}-1}\Big[
&\big(\frac{(1+2\nu)(s_c-1)}{2\nu}-\epsilon-\frac{C}{\tau}\big)\big\langle H\psi_k,\sigma_3\psi_k\big\rangle\\
&+C_k\big(\frac{(1+2\nu)(s_c-1)}{2\nu}-\epsilon\big)J_0\Big].
\end{aligned}
\end{equation*}
According to this estimate, \eqref{bdd of morawetz k} and
\begin{equation*}
\left|\tau^{-\frac{(1+2\nu)k}{\nu}}J_k-\langle H\psi_k,\sigma_3\psi_k\rangle-C_kJ_0\right|\leq\frac{C}{\tau}\langle H\psi_k,\sigma_3\psi_k\rangle,
\end{equation*}
we obtain
\begin{equation*}
J_k'(\tau)\geq\frac{1}{\tau}\big(\frac{(1+2\nu)(s_c-1)}{2\nu}-2\epsilon\big)J_k(\tau), \qquad \tau\geq\tau_0.
\end{equation*}
Combining this bound with \eqref{bdd correct energy k} and
using $\psi(s)=f$, we conclude that
\begin{equation}\label{Hk bdd}
\|U(\tau,s)f\|_{\dot H^1\cap\dot H^{2k+1}}
\leq C_{k,\epsilon}
\left(\frac{\tau}{s}\right)^{
\frac{1+2\nu}{4\nu}(s_c-2k-1)-\epsilon}
\|f\|_{\dot H^1\cap\dot H^{2k+1}},
\qquad \tau_0\leq\tau\leq s.
\end{equation}

\medskip Finally, interpolation between \eqref{H1 bdd} and \eqref{Hk bdd} yields \eqref{linear est} for $\frac{d}{2}<\sigma<s_c+2\nu|\lambda_+|$. This completes the proof.

 \medskip 
 
\subsection{Proof of Theorem \ref{main thm}.} \label{subsec 6.2}

\medskip We now transform \eqref{eq f} into a fixed-point problem. Rewrite \eqref{eq f} in the following integral form
\begin{equation}\label{integral form sol}
f(\tau)=J(f)(\tau),
\end{equation}
where
\begin{equation}\label{integral eq}
J(f)(\tau)=i\int_{\tau}^{\infty}U(\tau,s)\big(\mathcal{F}_2(f)(s)+r(s)\big)\ud s.
\end{equation}
Slightly abusing notation we identify in \eqref{integral form sol} $\mathbb{C}^2$ vectors of the form 
$\left(\begin{array}{l}
f\\
\bar{f}
\end{array}\right)$
with their first component $f$.

By \eqref{def K}, we have $K+1\geq s_c-1$. The condition $d\geq 11$ and $p\in\mathbb{N}^{\ast}$ that satisfying $2p+1>p_{JL}$ gives
$s_c>2+\sqrt{d-1}>5$. Moreover, \eqref{def lambda pm} implies 
$|\lambda_+|<s_c-1$, while
\begin{equation*}
\lfloor d/2\rfloor+1-s_c\leq 2.
\end{equation*}
Consequently,
\begin{equation}\label{nu est 1}
2K\geq2s_c-4>s_c+1>|\lambda_+|+\lfloor d/2\rfloor+1-s_c.
\end{equation}

Set
\begin{equation*}
\nu_0(d,p)
=\frac{\lfloor d/2\rfloor+1-s_c}{2|\lambda_+|}.
\end{equation*}
Then, for every $\nu>\nu_0$,
\begin{equation*}
\lfloor d/2\rfloor+1-s_c<2\nu|\lambda_+|,
\end{equation*}
combining this estimate with \eqref{nu est 1} yield
\begin{equation*}
\frac{1+2\nu}{4\nu}\bigl(\lfloor d/2\rfloor+1-s_c\bigr)<\frac{|\lambda_+|+\lfloor d/2\rfloor+1-s_c}{2}<K.
\end{equation*}
Thus the interval
\begin{equation*}
\Big(\frac {d}{2},\,s_c+\min\Big\{2\nu|\lambda_+|,\frac{4\nu K}{1+2\nu}\Big\}\Big)
\end{equation*}
contains at least one integer, with $\lfloor d/2\rfloor+1$ being an admissible choice. We then fix an integer $\sigma$ satisfying
\begin{equation*}
\frac{d}{2}<\sigma<s_c+\min\Big\{2\nu|\lambda_+|,\frac{4\nu K}{1+2\nu}\Big\}.
\end{equation*}
Choose $\epsilon>0$ in \eqref{linear est} sufficiently small that
\begin{equation*}
0<\epsilon<
K-\frac{1+2\nu}{4\nu}(\sigma-s_c).
\end{equation*}
The upper endpoint is positive by the choice of $\sigma$.
Hence the exponent in \eqref{linear est} satisfies
\begin{equation*}
\frac{1+2\nu}{4\nu}(\sigma-s_c)+\epsilon<K.
\end{equation*}

\medskip
Since $K>1$, we may now choose $\gamma$ such that
\begin{equation}\label{def beta}
\max\Big\{1,\frac{1+2\nu}{4\nu}(\sigma-s_c)+\epsilon\Big\}<\gamma<K,
\end{equation}
and $\epsilon_1\in(0,\nu)$ satisfying
\begin{equation}\label{range of epsilon1}
\frac{\nu(2\gamma+s_c-1-|\lambda_{+}|)}{2K+s_c-1-|\lambda_{+}|}<\epsilon_1<\nu.
\end{equation}
The lower endpoint lies in $(0,\nu)$, since $s_c-1>|\lambda_+|$ and $1<\gamma<K$. In the following proof, we fix 
\begin{equation}\label{def epsilon1}
\epsilon_1=\frac{\nu(K+\gamma+s_c-1-|\lambda_{+}|)}{2K+s_c-1-|\lambda_{+}|}.\footnote{This choice has no specific meaning here. To simplify the exponent in the error estimate \eqref{error est}, we choose this $\epsilon_1\in (\frac{\nu(2\gamma+s_c-1-|\lambda_{+}|)}{2K+s_c-1-|\lambda_{+}|},\nu)$. Then $(\nu-\epsilon_1)(|\lambda_{+}|+l-s_c)+2K\epsilon_1=\nu(K+\gamma)+(\nu-\epsilon_1)(l-1)\geq\nu(K+\gamma)$ for $l\geq1$, with equality at $l=1$.} 
\end{equation} 

 \medskip Consider the Banach space
\begin{equation*}
X:=\Big\{f\in C([\tau_0,\infty),\,\dot{H}^{1}(\mathbb{R}^d)\cap\dot{H}^{\sigma}(\mathbb{R}^d)): \|f\|_{X}<\infty\Big\},
\end{equation*}
equipped with the norm
\begin{equation*}
\|f\|_{X}:=\sup_{\tau\geq\tau_0}\tau^{\gamma}\|f(\tau)\|_{\dot{H}^1\cap\dot{H}^{\sigma}}.
\end{equation*}
Our intention is to prove that $J$ defined in \eqref{integral eq} is a contraction of the unit ball
\begin{equation*}
\mathbb{B}:=\left\{f\in X:\|f\|_{X}\leq1\right\}
\end{equation*}
into itself, provided $\tau_0$ is sufficiently large.
 
\medskip Let $f\in\mathbb{B}$. By the estimate \eqref{non est}, 
\begin{equation*}
\|\mathcal{F}_2(f)\|_{\dot{H}^1\cap\dot{H}^{\sigma}}\leq C\|f\|_{\dot{H}^1\cap\dot{H}^{\sigma}}^2\leq C\tau^{-2\gamma}\|f\|_{X}^2.
\end{equation*}
Using \eqref{linear est}, we obtain
\begin{equation}\label{contraction nonlinear}
\begin{aligned}
\sup\limits_{\tau\geq \tau_0}\tau^{\gamma}\Big\|&\int_{\tau}^{\infty}U(\tau,s)\mathcal{F}_2(f)(s)\ud s\Big\|_{\dot{H}^1\cap\dot{H}^{\sigma}}\\
\leq &C\sup_{\tau\geq\tau_0}\tau^{\gamma-\frac{1+2\nu}{4\nu}(\sigma-s_c)-\epsilon}\int_\tau^{\infty} s^{\frac{1+2\nu}{4\nu}(\sigma-s_c)+\epsilon-2\gamma}\,\ud s\,\|f\|_{X}^2.
\end{aligned}
\end{equation}
Notice that 
\begin{equation}\label{equ 1}
2\gamma-\frac{1+2\nu}{4\nu}(\sigma-s_c)-\epsilon-1=(\gamma-\frac{1+2\nu}{4\nu}(\sigma-s_c)-\epsilon)+(\gamma-1)>0,
\end{equation}
by \eqref{def beta}. Then 
\begin{equation*}
\sup\limits_{\tau\geq \tau_0}\tau^{\gamma}\Big\|\int_{\tau}^{\infty}U(\tau,s)\mathcal{F}_2(f)(s)\ud s\Big\|_{\dot{H}^1\cap\dot{H}^{\sigma}}\leq C\tau_{0}^{1-\gamma}\|f\|_{X}^2.
\end{equation*}
According to the estimate \eqref{error est}, by $\sigma>1$, $t^{2\nu}=(2\nu\tau)^{-1}$ and \eqref{def epsilon1}, it follows that
\begin{equation}\label{error est fix epsilon1}
\|r\|_{\dot{H}^1\cap\dot{H}^{\sigma}}\leq C\big(\tau^{-1-K}+\tau^{-1-\frac{K+\gamma}{2}}\big).
\end{equation}
Using \eqref{linear est}, \eqref{def beta} and \eqref{error est fix epsilon1}, we obtain
\begin{equation}\label{contraction error}
\sup\limits_{\tau\geq \tau_0}\tau^{\gamma}\Big\|\int_{\tau}^{\infty}U(\tau,s)r(s)\ud s\Big\|_{\dot{H}^1\cap\dot{H}^{\sigma}}\leq C\tau_{0}^{-\frac{K-\gamma}{2}},
\end{equation}
where we used $\frac{1+2\nu}{4\nu}(\sigma-s_c)+\epsilon<\gamma<\frac{K+\gamma}{2}<K$.

Substituting \eqref{contraction nonlinear} and \eqref{contraction error} into \eqref{integral eq}, for $\tau_0$ sufficiently large
\begin{equation}\label{contraction 1}
\|J(f)\|_{X}\leq C\tau_0^{1-\gamma}+C\tau_{0}^{-\frac{K-\gamma}{2}}\leq \frac{1}{2}.
\end{equation} 
 
\medskip We next prove the contraction estimate. For any $f,g\in\mathbb{B}$, the local Lipschitz estimate for the nonlinear term gives
\begin{equation}\label{contraction difference}
\begin{aligned}
\big\|\mathcal{F}_2(f)-\mathcal{F}_2(g)\big\|_{\dot{H}^1\cap\dot{H}^{\sigma}}\leq& C\big(\|f\|_{\dot{H}^1\cap\dot{H}^{\sigma}}+\|g\|_{\dot{H}^1\cap\dot{H}^{\sigma}}\big)\|f-g\|_{\dot{H}^1\cap\dot{H}^{\sigma}}\\
\leq &C\tau^{-2\gamma}\|f-g\|_{X}.
\end{aligned}
\end{equation}
Then \eqref{linear est}, \eqref{def beta}, \eqref{equ 1} and \eqref{contraction difference} yield 
 \begin{equation}\label{contraction 2}
 \begin{aligned}
 \|J(f)-J(g)\|_{X}\leq& \frac{C}{2\gamma-\frac{1+2\nu}{4\nu}(\sigma-s_c)-\epsilon-1}\tau_0^{1-\gamma}\|f-g\|_{X}\leq \frac{1}{2}\|f-g\|_{X},
 \end{aligned}
 \end{equation}
 provided $\tau_0$ sufficiently large.
 
 \medskip Therefore, \eqref{contraction 1} and \eqref{contraction 2} means that for $\tau_0$ sufficiently large, $J$ is a contraction of the unit ball $\mathbb{B}$ into itself and consequently, has a unique fixed point $f$ that satisfies
 \begin{equation}\label{fixed point}
 \|f(\tau)\|_{\dot{H}^1\cap\dot{H}^{\sigma}}\leq \tau^{-\gamma},\,\,\,\,\tau\geq \tau_0.
 \end{equation}

Returning to the original variables, we obtain the exact solution
\begin{equation*}
\psi(t,x)=\psi^{ap}(t,x)+\lambda^{1/p}(t)f(\tau,\lambda(t)x).
\end{equation*}
For every integer $1\leq s\leq\sigma$, interpolation yields
\begin{equation*}
\begin{aligned}
\Big\|\lambda^{1/p}(t)f(\tau,\lambda(t)\,\cdot)\Big\|_{\dot H^s}&=\lambda^{s-s_c}(t)\|f\|_{\dot H^s}\\
&\leq C t^{2\nu\gamma-(\frac12+\nu)(s-s_c)}.
\end{aligned}
\end{equation*}
Indeed, by \eqref{def beta}
\begin{equation*}
2\nu\gamma-(\frac12+\nu)(s-s_c)\geq 2\nu\gamma-(\frac12+\nu)(\sigma-s_c)>0.
\end{equation*}
This implies
\begin{equation}\label{limite radiation}
\lambda^{s_c-s}(t)\Big\|\lambda^{1/p}(t)f(\tau,\lambda(t)\,\cdot)\Big\|_{\dot H^s}\leq Ct^{2\nu\gamma}\longrightarrow0,
\qquad 1\leq s\leq\sigma.
\end{equation}

\medskip We identify the limiting radiation by setting
$$\zeta^{\ast}(r)=f_0(r)-c_\infty r^{-\frac{1}{p}}.$$
The expansion \eqref{eq f0} implies, for every integer $l\geq0$,
\begin{equation*}
\big|\partial_r^l\zeta^{\ast}(r)\big|\leq C_l r^{-\frac1p+2\nu|\lambda_+|-l},\qquad 0<r<\delta.
\end{equation*}
Since $f_0$ is supported in $[0,2\delta]$, then
\begin{equation*}
\zeta^{\ast}(r)=-c_\infty r^{-\frac1p},\qquad r\geq2\delta.
\end{equation*}
Therefore, we deduce that
\begin{equation*}
\zeta^{\ast}\in\dot H^s(\mathbb R^d),\qquad s_c<s<s_c+2\nu|\lambda_+|.
\end{equation*}
Splitting the approximate radiation into the inner, self-similar, and outer regions, and using the differentiated profile expansions and the matching estimates \eqref{est psi in ss any N}, \eqref{match est out}, we obtain
\begin{equation}\label{approximate limit}
\begin{aligned}
&\left\|\lambda^{\frac1p}(t)\chi^{ap}(t,\lambda(t)\,\cdot)-\zeta^{\ast}\right\|_{\dot H^s}\\
&\quad\leq C\Big(t^{2\nu-(\frac12+\nu)(s-s_c)}+t^{\nu|\lambda_+|-\frac12(s-s_c)}t^{(\frac12-\epsilon_2)(s_c+2\nu|\lambda_+|-s)}\Big)
\longrightarrow0,
\end{aligned}
\end{equation}
provided that
$$s_c<s\leq\sigma,\qquad s<s_c+\frac{4\nu}{1+2\nu}.$$

Finally, the exact radiation satisfies
\begin{equation*}
\zeta(t,x)=\lambda^{\frac1p}(t)\chi^{ap}(t,\lambda(t)|x|)+\lambda^{\frac1p}(t)f(\tau,\lambda(t)x).
\end{equation*}
By \eqref{fixed point} and \eqref{approximate limit}, we deduce that
\begin{equation*}
\|\zeta(t)-\zeta^\ast\|_{\dot H^s}\longrightarrow0
\end{equation*}
for the stated range of $s$, proving \eqref{regularity zeta star main} and \eqref{radiation convergence main}. For every integer $1\leq s<s_c$, by \eqref{whole L2 est},
\begin{equation*}
\lambda^{s-s_c}(t)\Big\|\nabla_\rho^s\Big(\chi^{ap}(t,\rho)+\Big(1-\Theta\Big(\frac{\rho}{\lambda(t)\delta}\Big)\Big)Q(\rho)\Big)\Big\|_{L^2}\leq
C.
\end{equation*}
Together with the uniform $\dot H^s$ bound for $\lambda^{\frac1p}(t)\Theta(|x|/\delta)Q(\lambda(t)|x|)$ and \eqref{limite radiation}, this proves
\eqref{low singularity}. Thus, \eqref{limite radiation} together with Proposition \ref{main prop} gives Theorem \ref{main thm}.

  \begin{appendix}

%%%%%%%%%%%%%%%%%%%%%%%%%%%%%%%%%%%%%%%%%%%%%%%%%%%%%%%%%%%%%%%%%%%%%%%%%%%%%%%%%%%%%%%%%%%%%%%%%%%%%%%%%%%

\section{Weighted coercivity of the localized Morawetz form}\label{weighted coercivity}
\begin{lemma}\label{lem KA coercivity}
{\sl Let $d\geq11$ and $p\in\mathbb{N}^{\ast}$ satisfy $2p+1>p_{JL}$, and let $K_A$ be defined by \eqref{def KA}.
For every sufficiently small $\delta>0$, there exist constants $C_\delta,C>0$ such that, for all sufficiently large $A$ and every
smooth, compactly supported radial function $\psi$,
\begin{equation}\label{clam KA bdd}
K_{A}(\psi)\geq C_{\delta}\int_{\mathbb{R}^d}\frac{1}{1+|y|^{\delta}}\big(|\nabla\psi|^2+\frac{|\psi|^2}{|y|^2}\big)\ud y-\frac{C}{A^{\delta}}\langle H\psi,\sigma_3\psi\rangle.
\end{equation} }
\end{lemma}
\begin{proof}
For $L_\pm=-\Delta-V_\pm$, where $V_+=(2p+1)Q^{2p}$ and $V_-=Q^{2p}$, and a real valued smooth, compactly supported radial function $g$, integration by parts, justified by approximation at the origin, gives
\begin{equation}\label{claim equa}
\begin{aligned}
&\Big\langle Lg, \nabla\phi_A\cdot\nabla g+\frac{\Delta\phi_A}{2}g\Big\rangle\\
&\qquad
=\int_{\mathbb R^d}\phi_A''|\partial_rg|^2\,\ud y-\frac{1}{4}\int_{\mathbb R^d}
\Delta^2\phi_A\,g^2\,\ud y+\frac{1}{2}\int_{\mathbb R^d}\nabla V\cdot\nabla\phi_A\,g^2\,\ud y.
\end{aligned}
\end{equation}

\medskip We first estimate the potential contribution. Since $Q$ and $\Lambda Q$ both positive, then
 \begin{equation*}
\frac{1}{2} y\partial_yV_+=p(2p+1)Q^{2p-1}\Lambda Q-V_+\geq -V_+,
\end{equation*}
and similarly
\begin{equation*}
\frac{1}{2} y\partial_yV_-=pQ^{2p-1}\Lambda Q-V_-\geq -V_-.
\end{equation*}
Since $\phi_A'\geq0$, we therefore have
\begin{equation}\label{potential virial lower}
\frac{1}{2}\int_{\mathbb R^d}\nabla V_{\pm}\cdot\nabla\phi_A\,g^2\,\ud y
\geq -\int_{\mathbb R^d}\frac{\phi_A'}{|y|}V_{\pm}g^2\,\ud y.
\end{equation}
By the pointwise bound on $Q$,
\begin{equation*}
V_{-}(y)\leq V_+(y)\leq \frac{(2p+1)c_{\infty}^{2p}}{|y|^2}.
\end{equation*}
Taking this estimate into \eqref{potential virial lower}, then
\begin{equation*}\label{potential virial Hardy}
\frac{1}{2}\int_{\mathbb R^d}\nabla V_{\pm}\cdot\nabla\phi_A\,g^2\,\ud y\geq -(2p+1)c_{\infty}^{2p}\int_{\mathbb{R}^d}\chi_{A}(y)\frac{g^2}{|y|^{2+\delta}}\ud y.
\end{equation*}

\medskip We next estimate the first term in \eqref{claim equa}. Since
\begin{equation*}
\phi_A''(r)=(1-\delta)\chi_A(r)r^{-\delta}+\chi_A'(r)r^{1-\delta},
\end{equation*}
and
\begin{equation*}
\operatorname{supp}\chi_A'
\subset\{A\leq r\leq2A\},
\qquad
|\chi_A'(r)|\leq\frac{C}{A},
\end{equation*}
we have
\begin{equation}\label{phi second lower}
\begin{aligned}
\int_{\mathbb{R}^d}\phi_A''|\nabla g|^2\ud y\geq{}&
(1-\delta)\int_{\mathbb{R}^d}\chi_A\,\frac{|\nabla g|^2}{|y|^\delta}\ud y\\
&-\frac{C}{A}\cdot A^{1-\delta}\int_{A\leq |y|\leq2A}|\nabla g|^2\ud y.
\end{aligned}
\end{equation}

\medskip For any $\beta\in\mathbb{R}$, using the fact
\begin{equation*}
\int_{0}^{\infty}\chi_{A}(r)\, r^{-\delta}\big(\partial_r g+\frac{\beta}{r} g\big)^2 r^{d-1}\ud r\geq 0,
\end{equation*}
and taking $\beta=\frac{d-2-\delta}{2}$, we find the weighted Hardy inequality,
\begin{equation}\label{cutoff weighted Hardy}
\begin{aligned}
\int_{\mathbb{R}^d}\chi_A\,\frac{|\nabla g|^2}{|y|^\delta}\ud y
\geq{}&\big(\frac{d-2-\delta}{2}\big)^2\int_{\mathbb{R}^d}\chi_A\,\frac{g^2}{|y|^{2+\delta}}\ud y\\
&\quad-\frac{C}{A^\delta}\int_{A\leq |y|\leq2A}\frac{g^2}{|y|^2}\ud y.
\end{aligned}
\end{equation}

\medskip Note that at $\delta=0$, $\frac{(d-2)^2}{4}-(2p+1)c_{\infty}^{2p}=\frac{Discr}{4}>0$, then one can choose $\delta>0$ sufficiently small and then $\eta_{\delta}>0$ sufficiently small such that
\begin{equation*}
(1-\delta-\eta_{\delta})\big(\frac{d-2-\delta}{2}\big)^2-(2p+1)c_{\infty}^{2p}>0.
\end{equation*}
Applying \eqref{cutoff weighted Hardy} to \eqref{phi second lower} and write $1-\delta=(1-\delta-\eta_{\delta})+\eta_{\delta}$. We infer from \eqref{potential virial lower} that
\begin{equation}\label{0 energy est 13}
\begin{aligned}
\int_{\mathbb{R}^d}\phi_{A}''\,&|\nabla g|^2\ud y+\frac{1}{2}\int_{\mathbb{R}^d}\nabla V_{\pm}\nabla\phi_{A} \,g^2\ud y\\
&\geq C_{\delta}\int_{\mathbb{R}^d}\chi_{A}\,|y|^{-\delta}\big(|\nabla g|^2+\frac{g^2}{|y|^2}\big)\ud y-\frac{C}{A^{\delta}}\int_{A\leq |y|\leq 2A}\big(|\nabla g|^2+\frac{g^2}{|y|^2}\big)\ud y.
\end{aligned}
\end{equation}

\medskip \medskip We finally estimate the second term in \eqref{claim equa}. Since $\chi_A=1$ for $r\leq A$, a direct computation gives
\begin{equation*}
-\Delta^2\phi_A(r)=\delta(d-\delta)(d-2-\delta)r^{-2-\delta}\geq0,\qquad 0<r<A.
\end{equation*}
On the transition region, using
\begin{equation*}
|\partial_r^j\chi_A(r)|\leq C_jA^{-j}, \qquad 1\leq j\leq3,
\end{equation*}
we obtain
\begin{equation*}
|\Delta^2\phi_A(r)|\leq CA^{-2-\delta}\leq \frac{C}{A^\delta r^2},\qquad A\leq r\leq2A.
\end{equation*}
Moreover, $\phi_A'=0$ for $r\geq2A$, so $\Delta^2\phi_A=0$ there. Therefore,
\begin{equation}\label{0 energy est 2}
\begin{aligned}
-\frac14\int_{\mathbb R^d}\Delta^2\phi_A\,g^2\,\ud y\geq&\frac{\delta(d-\delta)(d-2-\delta)}4
\int_{|y|<A}\frac{g^2}{|y|^{2+\delta}}\,\ud y -\frac{C}{A^\delta}\int_{A\leq|y|\leq2A}\frac{g^2}{|y|^2}\,\ud y\\
\geq&-\frac{C}{A^\delta}\int_{A\leq|y|\leq2A}\frac{g^2}{|y|^2}\,\ud y.
\end{aligned}
\end{equation}

\medskip Substituting \eqref{0 energy est 13} and \eqref{0 energy est 2} into \eqref{claim equa}, we deduce
\begin{equation}\label{bdd g}
\begin{aligned}
&\Big\langle L_{\pm}g,\frac{\Delta \phi_{A}}{2}g+\nabla\phi_{A}\,\nabla g\Big\rangle\\
\geq& C_{\delta}\int_{\mathbb{R}^d}\chi_{A}\, |y|^{-\delta}\big(|\nabla g|^2+\frac{g^2}{|y|^2}\big)\ud y-\frac{C}{A^{\delta}}\int_{A\leq |y|\leq 2A} |\nabla g|^2+\frac{g^2}{|y|^2}\ud y.
\end{aligned}
\end{equation}
Since $A\gg 1$ and $\chi_A=1$ on $\{|y|\leq A\}$, then
\begin{equation*}
\frac{1}{1+r^{\delta}}\leq \left\{
\begin{aligned}
&C\,|y|^{-\delta},\qquad \text{for $0<|y|\leq A$},\\
&C\,A^{-\delta},\qquad\,\, \text{for $|y|\geq A$}.
\end{aligned}
\right.
\end{equation*}
Using this bound with \eqref{bdd energy 0} and Hardy's inequality, we deduce that 
\begin{equation*}
\begin{aligned}
\int_{\mathbb{R}^d}\frac{1}{1+|y|^{\delta}}&\big(|\nabla \psi|^2+\frac{|\psi|^2}{|y|^2}\big)\ud y\\
\leq& C\int_{\mathbb{R}^d}\chi_A\, |y|^{-\delta}\big(|\nabla \psi|^2+\frac{|\psi|^2}{|y|^2}\big)\ud y+\frac{C}{A^{\delta}}\langle H\psi,\sigma_3\psi\rangle.
\end{aligned}
\end{equation*}
Substituting this estimate and \eqref{bdd g} into \eqref{def KA} proves \eqref{clam KA bdd}.

\end{proof}
 \end{appendix}

 \end{document}